\RequirePackage{etoolbox}
\csdef{input@path}{%
 {sty/}
 {img/}
}%
\csgdef{huzhang}{bib/}
\documentclass{imsart}
\usepackage{bbm}
\usepackage{amsthm}
\usepackage{amsmath}
\usepackage{amsfonts}
\usepackage{natbib}
\usepackage{empheq}
\usepackage{upgreek}
\usepackage[colorlinks,citecolor=blue,urlcolor=blue,filecolor=blue,backref=page]{hyperref}
\usepackage{graphicx}
\def\Var{{\rm Var}}
\def\Cov{{\rm Cov}}
\def\N-IGP{{\hbox{\rm N-IGP}}}
\def\NGGP{{\rm NGGP}}
\def\GDP{{\rm GDP}}
\def\PDP{{\rm PDP}}
\startlocaldefs
\numberwithin{equation}{section}
\theoremstyle{plain}
\newtheorem{thm}{Theorem}[section]
\endlocaldefs

\newtheorem{theorem}[thm]{Theorem}
\newtheorem{corollary}[thm]{Corollary}
\newtheorem{remark}[thm]{Remark}
\newtheorem{assumption}[thm]{Assumption}
\newtheorem{example}[thm]{Example}
\newtheorem{definition}[thm]{Definition}
\newtheorem{prop}[thm]{Proposition}

\newcommand{\EE}{\mathbb{E}}

\newcommand{\cB}{\mathcal{B}}

\newcommand{\cF}{\mathcal{F}}

\newcommand{\RR}{\mathbb{R} }

\newcommand{\be}{\beta}

\newcommand{\XX}{\mathbb X}
\newcommand{\cX}{\mathcal X}

\newcommand{\al}{\alpha}

\newcommand{\om}{\omega}

\newcommand{\si}{\sigma}

\let\Section=\section
\def\section{\setcounter{equation}{0}\Section}

\def\RR{\mathbb{R} }
\def\NN{\mathbb{N}}
\def\EE{\mathbb{E}}

\def\cB{{\mathcal B}}

\def\cC{{\mathcal C}}

\def\cF{{\mathcal F}}

\def\fin{{\hfill $\Box$}}

\def\de{{\delta}}

\def\si{{\sigma}}

\def\Om{{\Omega}}
\def\Ga{{\Gamma}}

\def\th{{\theta}}

\def\Om{{\Omega}}

\def\NN{\mathbb N}

\def \RR{\mathbb R}
\def\DD{\mathbb D}

\def\EE{\mathbb E\ }

\def\QQ{\mathbb Q }

\def\cB{{\mathcal B}}

\def\cC{{\mathcal C}}

\def\cF{{\mathcal F}}

\def\fin{{\hfill $\Box$}}

\def\be{{\beta}}
\def\de{{\delta}}

\def\si{{\sigma}}

\def\Om{{\Omega}}
\def\al{{\alpha}}

\def\be{{\beta}}
\def\Ga{{\Gamma}}

\def\de{{\delta}}

\def\si{{\sigma}}

\def\ZZ{{\mathbb Z}}
\def\PP{{\mathbb P}}

\def\al{\alpha}

\def\Ad2{{\| A-\widetilde A \|_{L^{2}_{loc}} }}
\def\om{{\omega}}
\def\Om{{\Omega}}

\def\cB{{\mathcal B}}

\def\rX{{\mathcal X}}

\def\DD{{\mathbb D}}

\def\th{{\theta}}
\def\Beta{{\rm Beta}}
\def\DP{{\rm DP}}
\def\PDP{{\rm PDP}}
\def\Dir{{\rm Dir}}
\def\DPG{{\rm DPG}}

\begin{document}

\begin{frontmatter}
\title{Functional central limit theorems for stick-breaking priors}
\runtitle{CLT  for Stick-breaking Priors}

\begin{aug}
\author{\fnms{Yaozhong} \snm{Hu} \thanks{Department of Mathematical and Statistical  Sciences, University of Alberta at Edmonton, Alberta,  Canada,    T6G 2G1. Email: yaozhong@ualberta.ca}} 
\and 
\author{\fnms{Junxi} \snm{Zhang}\thanks{Department of Mathematical and Statistical  Sciences, University of Alberta at Edmonton, Alberta,  Canada,    T6G 2G1. Email: junxi3@ualberta.ca}}

\runauthor{Y. Hu and J.  Zhang}

%

\thankstext{t2}{Supported by an NSERC discovery fund and a startup fund of University of Alberta.}

\end{aug}

\begin{abstract}
We obtain the empirical strong law of large numbers, empirical Glivenko-Cantelli theorem, central limit theorem, 
functional central limit theorem for various nonparametric Bayesian priors
which  include the Dirichlet process with general stick-breaking weights,      
the Poisson-Dirichlet process,   the  normalized inverse Gaussian 
process,   the normalized generalized gamma 
process, and     the  generalized Dirichlet process.  
For the Dirichlet process with general stick-breaking weights, we introduce two general conditions such that the central limit theorem and 
functional central limit theorem hold. 
Except in the case of the 
generalized Dirichlet process,  since the finite dimensional 
distributions of these processes are either hard to obtain or are 
complicated to use even they are available,  
we use the      method of moments 
to obtain the convergence results.   
For the generalized Dirichlet process we use  its finite dimensional marginal distributions   to obtain the asymptotics  although 
the computations are highly  technical. 
\end{abstract}
%

\end{frontmatter}

\section{Introduction}
Ever since the work of    \citet{ferguson1973}   the   Dirichlet process     has  become   a critical tool in Bayesian nonparametric statistics  and has  found  applications in various areas, including 
machine learning, biological science, social science and so on.
One of the  important features of the Dirichlet process is that when the prior is a Dirichlet process   its  posterior is also a Dirichlet process  (see e.g. \citet{ferguson1973}).  This makes the complex computation in the Bayesian nonparametric analysis  possible and enables the Dirichlet process 
to become a backbone of the  Bayesian nonparametric statistics. 

To widen the applicability of the Bayesian nonparametric statistics, 
researchers have tried to extend the concept of Dirichlet process. One of these efforts is the introduction of   the stick-breaking process. 
The first breakthrough along this path 
is due to  \citet{sethuraman1994}  who shows  that  the Dirichlet process admits the stick-breaking representation
(see \eqref{1.1}-\eqref{1.111} in the next section), where the stick-breaking weights are independent and identically distributed 
(iid)  random variables satisfying the beta distribution
$\Beta(1, a)$ (throughout this paper the notation 
$ \Beta(\al,\be)$ denotes the beta distribution whose density is 
$g(x; \al,\be)=\frac{\Ga(\al)\Ga(\be)}{\Ga(\al+\be)}x^{\al-1} (1-x)^{\be-1}\,, \ 0<x< 1$).  Within this stick-breaking representation, we can extend 
the class of Dirichlet processes to many other   priors  by assuming that the stick-breaking weights are iid with other distributions; 
  satisfy some other kinds of dependence;   or satisfy some specific (joint) distributions.  Among various such extensions, let us mention 
  the following works which we shall deal with
  in this paper. \citet{perman1992} obtain a general formulae for sized-biased sampling from a Poisson point process where the size of a point is defined by an arbitrary strictly positive function.  From this  formulae, they identify 
   the stick-breaking representation of the two-parameter Poisson-Dirichlet process,  
  which admits a stick-breaking process with the stick-breaking weights $v_i \overset{ind}{\sim} \Beta(1-b,a+ib)$,  where  
  $b>0$, $a>-b$  and $i=1, 2, \cdots$.  
  \citet{favaro2012}  introduce 
the normalized inverse Gaussian process through  its stick-breaking 
representation by identifying the  explicit 
finite dimensional joint density functions
of its stick-breaking weights.  
\citet{favaro2016} present  the stick-breaking representation of a more general class of random measures called homogeneous normalized random measures with independent increments (NRMIs),   which include  the normalized generalized gamma process and the generalized Dirichlet process,    two widely used priors in Bayesian nonparametric statistics.

%

Strong law of large numbers,  central limit theorem 
and functional central limit theorem have   always been 
ones of the central topics in statistics and in probability theory. 
Without exception the  
  asymptotic behaviours  of the Dirichlet process and 
other Bayesian nonparametric priors
play  an important role  in the  Bayesian nonparametric 
analysis, for example in  the construction of asymptotic Bayesian confidence intervals, regression analysis and functional estimations. 
Compared to the vast literature in the  field of 
  parametric statistics relevant to   these issues
 the achievements in the field of Bayesian nonparametrics  
 are quite limited.  However, let us mention the following works
 pioneered this paper. 
\citet{sethuraman1982} discuss   the weak convergences 
of the Dirichlet measure  $P$ when its parameter measure 
(i.e the measure $aH$ in this paper) approaches to a non-zero measure or a zero measure respectively. \citet{lo1983} studies  the central limit theorem 
 of the posterior distribution of Dirichlet process which  is analogous to 
 our central limit theorem for the  Dirichlet process.  
  Based on this result, 
 \citet{lo1987}    obtains 
     the asymptotic confidence  bounds and establishes 
        the asymptotic validity of the Bayesian bootstrap method. 
        The above mentioned Lo's results are  extended to the mixtures of Dirichlet process by \citet{brunner1996}.   
\citet{james2008} reveals the consistency behaviour (the posterior distribution converges to the true distribution weakly) and the functional central limit theorem  
for the posterior distribution of the two-parameter Poisson-Dirichlet process (with fixed  $a$ and when the   sample size goes to infinity).   \citet{kim2004} show  that the Bernstein-von Mises theorem holds in survival models for the Dirichlet process, Beta process and Gamma process. \citet{dawson2006} establish 
 the large deviation principle
 for the Poisson-Dirichlet distribution and give the explicit rate functions when the parameter $a$ (which represents the mutation rate in the context of population genetics) approaches infinity. \citet{labadi2013} present the functional central limit theorem  for the normalized inverse Gaussian process on $D(\RR)$ when its parameter $a$ is large by using  its finite dimensional joint  density. 
 \citet{labadi2016}  obtain   the functional central limit theorem for the Dirichlet process on $D(\RR)$ when the parameter $a$ is large by using the finite dimensional joint density of Dirichlet process.

%
From the above mentioned works we see that 
there are only very limited results on the asymptotics of the stick-breaking 
processes.  { Relevant  to   the asymptotics  as  $a\rightarrow \infty$, 
 there have been established the  central limit  theorem 
 and functional central limit theorem   only for two 
processes:     the 
Dirichlet process and the  normalized inverse Gaussian  process.  }
The reason for the above limitation is   that 
the most commonly used technique 
appeals to the explicit forms of   the finite dimensional densities of the process itself.
 This method is effective  
only when the finite dimensional distributions have explicit forms and are 
possible  to handle.  
It cannot be applied to study  other processes  when the   explicit forms for the
finite dimensional marginal densities of the process itself are unavailable  
or they are  too complex to analyse even though they are available.  
 
This paper is to introduce the method of  moments into 
 this study and to provide 
 a systematic  study of the 
  asymptotics  as $a\rightarrow \infty$ 
for various   stick-breaking processes 
depending on a parameter $a>0$. We are  mainly 
 concerned  with three types of the asymptotics 
(strong law  of large numbers, central limit theorem, and functional central limit theorem)
for a number of processes, which  include  
the Dirichlet process  with general stick-breaking weights,  the   classical Dirichlet process  $\DP(a,H)$ (see e.g. \citet{ferguson1973}, \citet{ghosalbook2017}, \citet{regazzini2002}),   
the   two-parameter Poisson-Dirichlet process $\PDP(a,b,H)$ (also known as Pitman-Yor process,  \citet{pitman1997}),  the normalized inverse Gaussian process $\N-IGP(a,H)$  (see \citet{lijoi2005b}), the normalized generalized gamma process $\NGGP(\sigma,a,H)$    
(see \citet{brix1999}), 
and the generalized Dirichlet process $\GDP(a,r,H)$  (see \citet{lijoi2005a}).
   However, we will not concern with the large sample problem in this work.

For the generalized Dirichlet process
since the finite dimensional marginal distributions
of the process itself 
are  available we shall use them  to obtain the asymptotics directly
although  the computations are very technical. Let us point out that 
this process also admits a stick-breaking representation. However, 
it seems to us that it is more complex to use the method of moments 
than to use the   finite dimensional marginal distributions
of the process itself.   

Let us point out the following points of the paper about the well-known
Bayesian nonparametric priors.  
\begin{enumerate}
\item[(1)](for Dirichlet process) 
 Both the finite dimensional distributions of the 
stick-breaking weights and the  process itself are explicit and are 
easy to handle. 
Prior to this work the central limit theorem  and the functional central limit theorem
have been established  for this process by using the finite dimensional
distributions  of the process itself. 

For the Dirichlet process the stick-breaking weights $\{v_i\}$ are iid 
and follow the Beta distribution $\Beta(1, \alpha)$. We introduce the
concept of Dirichlet process with general  stick-breaking weights, where we still require the stick-breaking weights $\{v_i\}$  to be  iid 
but the distribution $g_a$ 
 they follow can be arbitrary. In this case there is no way to 
obtain the explicit form of the joint distributions of the process itself. 
We  use the method of moments  to establish  the central limit theorem and 
the functional central limit theorem for this process with very general
  distribution $g_a$.  For example, $g_a$ can be a Beta
distribution $\Beta(\rho_a, a)$, where $\rho_a$ is a function of $a$  
 such that $\rho_a/a\rightarrow 0$ as $a\rightarrow \infty$, which may have potential applications in    posterior Dirichlet process.   
In this case 
the joint distributions of the process itself is  unavailable
except in the case $\rho_a=1$,  i.e. in the case of Dirichlet process.

\item[(ii)](for the normalized inverse Gaussian process and for the generalized Dirichlet process) Both the finite dimensional distributions of the 
stick-breaking weights and that of the  process itself are explicit.
Prior to this work the central limit theorem and 
the functional central limit theorem
have been established  only for  the normalized inverse Gaussian process  
by using the finite dimensional
distributions of the process itself. We shall also use the 
 finite dimensional
distributions of the process itself to obtain the central limit theorem and 
the functional central limit theorem
   for the generalized Dirichlet process.  
We shall use the method of  moments to re-derive  the  central limit theorem and 
the functional central limit theorem
   for  the normalized inverse Gaussian process,
providing an alternative tool for this process.  

\item[(iii)](for the Poisson-Dirichlet process 
and  the normalized generalized gamma process)  
The finite dimensional distributions   of the 
stick-breaking weights are available but 
 the finite dimensional distributions of the  process itself are 
 not available. We   use the method of  moments to obtain the  central limit theorem and 
the functional central limit theorem for these processes. 
%
\end{enumerate}
 
Now we explain the organization of this paper.
 In Section 2 we   recall  some commonly used   stick-breaking processes, including 
the   classical Dirichlet process  $\DP(a,H)$,   
the   two-parameter Poisson-Dirichlet process $\PDP(a,b,H)$,  the normalized inverse 
Gaussian process $\N-IGP(a,H)$, the normalized generalized gamma process $\NGGP(\sigma,a,H)$, 
and the generalized Dirichlet process $\GDP(a,r,H)$. We will also introduce the Dirichlet process with general iid stick-breaking weights. Interested readers are
referred  to \citet{huzhang}   and references therein   for a 
recent survey of some of these processes  and their applications.  In this section 
we  also recall the concepts  of  the weak convergence with respect to  the Skorohod topology on $D(\RR^d)$.
     Although the method of the moments can
be used to obtain the strong law  of large numbers, central limit theorem, 
and the functional central limit theorem, the computations are still 
very sophisticated  and for different processes the computations of the moments
are  different. For this reason, we present  our moment results 
for  various stick-breaking processes,  including $\DPG(g,H)$, $\PDP(a,b,H)$, $\N-IGP(a,H)$, and $\NGGP(\sigma, a, H)$,  $\GDP(a,r, H)$,  separately in Section \ref{s.moment}.  We point out that we use finite dimensional distribution   method to obtain the moment results for  $\GDP(a,r, H)$, while we compute the moments directly through their stick-breaking weights for other  processes. 
   In Section \ref{s.main},
we state the strong law of large numbers, central  limit theorem,
and functional central limit theorem. 
The Dirichlet processes with general stick-breaking weights are new 
and we allow the stick-breaking weights to be some very general iid random variables  defined 
on $(0, 1)$.  With different choices of the stick-breaking weights we can obtain
various known stick-breaking processes.  
Because of this generality of the stick-breaking weights
we state one theorem on the central limit theorem and functional central limit
theorem for this type of  processes. 
We state a similar theorem for all  other 
processes 
(the Poisson-Dirichlet process $\PDP(a,b,H)$, the  
normalized inverse Gaussian process $\N-IGP(a,H)$, 
the normalized generalized gamma process $\NGGP(\sigma,a,H)$, 
   and the generalized Dirichlet process $\GDP(a,r,H)$).  The details of the proofs will be provided in a supplementary file. 

Finally,  let us emphasize that  all the processes we dealt with in this  
paper are actually   ``{\it random probability measures}".  However, 
we follow the convention in the literature to continue to call them 
``{\it processes}".


\section{Preliminary Notations}\label{s.pre}
\subsection{Definitions}
 Let $(\Om, \cF, \PP)$ be a complete probability 
 space and let $(\mathbb{X},\mathcal{X})$ be a measurable Polish space, namely,
$\mathbb{X}$ is a   separable complete metric space and $\mathcal{X}$ is the 
Borel $\sigma$-algebra of $\mathbb{X}$.  Let $H$ be a  
nonatomic probability  measure on  $(\mathbb{X},\mathcal{X})$    
(i.e.  $H(\{x\})=0$ for 
any  $x \in \mathbb{X}$).   
A random  measure  
is a 
mapping  $P$ from $\Om\times \rX$ to $\RR_+$ 
(we denote  this random measure by 
$P=(P(\om, A)$, $\om\in \Om\,, A\in  \rX )$)  such that
\begin{enumerate}
\item[(i)] when $\om\in \Om$ is fixed, $P(\om, \cdot)$ is a   measure 
on $(\XX, \rX)$; 
\item[(ii)]  when $A\in \rX$ is fixed, $P(\cdot, A)$ is a random
variable on $(\Om, \cF, \PP)$. 
\end{enumerate} 
Now we give the definition of the stick-breaking process 
(more appropriately a stick-breaking random probability measure). 
\begin{definition}\label{d.2.1}
A random measure $P =(P(\om, A),\om\in \Om, A\in \cX ) $ is said to be a \textit{stick-breaking process} with the \textit{base measure} $H$, if it has the following representation: 
\begin{empheq}[left=\empheqlbrace]{align} 
 &P=\sum_{i=1}^{\infty}w_i\delta_{\theta_i},\qquad \hbox{where} \label{1.1}\\
& w_1=v_1, \quad w_i=v_i\prod_{j=1}^{i-1}(1-v_j) \quad \text{for} \quad i= 2,3, \cdots,  \label{1.111}
\end{empheq}
where $\theta_i, i=1, 2, \cdots $ are iid random variables  defined on $(\Om, \cF, \PP)$
(Throughout the paper all random variables are defined on the probability space $(\Om, \cF, \PP)$) 
with values in $(\XX, \cX)$ such that for each $i$, the law of $\theta_i$ is 
$H$;      $\delta_{\theta_i} $ denotes the Dirac measure on 
$(\XX, \cX)$, this means $\de_{\th_i}(A)=
1$ if $\th_i\in A$ and $\de_{\th_i}(A)=
0$ if $\th_i\not\in A$ for any $A\in \cX$; and $v_i, i=1, 2, \cdots  $ are random variables
with values in $[0, 1]$, independent of $\{\th_i\}$, which are   called the \textit{stick-breaking weights}. 
\end{definition}
  
\begin{remark}
To make sure that $P$ is well-defined (namely,  
\eqref{1.1} is convergent),
one needs  to impose the condition that $\sum_{i=1}^{\infty}w_i=1$ almost surely,   which is  equivalent to the condition  
that  $\sum_{i=1}^{\infty}\mathbb{E}\left[\log(1-v_i)\right]=-\infty$ (e.g. \citet{ishwaran2001}).  
\end{remark} 

Since we assume that  $\{\th_i\}$ are iid and follow the distribution $H$,  if $H$ is given and fixed,   then the 
stick-breaking process $P$ depends only on the choice of $\{v_i\}$.
 The first milestone work  on the stick-breaking process is \citet{sethuraman1994}  where  it was  shown  that  the Dirichlet process admits the stick-breaking representation  \eqref{1.1}-\eqref{1.111}  with the stick-breaking weights of the form $v_i \overset{iid}{\sim} \Beta(1,a)$. 
We can use this characteristic  
 as the definition of the Dirichlet process.

\begin{definition}\label{d.2.3} Let $a>0$ and let $H$ be a nonatomic 
measure on $(\XX, \cX)$.  A random probability measure $P$ is called the Dirichlet 
process with parameter $(a,H)$, denoted by $P\sim \DP(a, H)$,
  if it has the representation \eqref{1.1}-\eqref{1.111}, where 
$v_i \overset{iid}{\sim} \Beta(1,a)$. 
\end{definition} 

\begin{remark} Throughout the entire paper, we shall assume that $a $
is a positive real number and $H$ is a nonatomic measure 
on $(\XX, \cX)$ unless otherwise specified.
\end{remark} 

With this definition we can write the original definition of the Dirichlet 
process as a proposition. To state this proposition, 
 we need  to recall the concept of the Dirichlet distribution. 
 Throughout this paper we use the following notation 
 to denote the standard     simplex in $\RR^n$:  
\begin{equation}
 \mathbb{S}_n=\left\{(s_1,\cdots, s_n)\in \RR^n:\quad s_i \geq 0,\quad \sum_{i=1}^ns_i=1\right\}\,.  \label{e.2.3a}
\end{equation} 
In case of no ambiguity we also write $ \mathbb{S}= \mathbb{S}_n$. 
  A random vector $(X_1, \cdots, X_n) \in \mathbb{S}$ is called to follow the  Dirichlet 
distribution with parameters $(\alpha_1,\cdots, \alpha_n)\in [0, \infty)^n$,
denoted by $(X_1,\cdots,X_n) \sim \Dir(\alpha_1, \cdots,\alpha_n)$, if the joint pdf of  $(X_1, \cdots,\\ X_n)$ is given by 
\begin{align*}
f(x_1,\cdots,x_{n})&=\frac{\Gamma(|\alpha|)}{\prod_{i=1}^n \Gamma(\alpha_i)}\prod_{i=1}^{n}x_i^{\alpha_i-1}\mathbbm{1}_{\mathbb{S}}(x_1,\cdots,x_n)\,,
\end{align*}
where $|\alpha|=\sum_{i=1}^n \alpha_i$, $\Gamma(a)=\int_0^\infty x^{a-1} dx \  (a>0)$,   is the gamma function, and $\mathbbm{1}_{\mathbb{S}}$ is the indicator function of the simplex $\mathbb{S}$.  
With this notion of Dirichlet distribution we can 
write the following proposition.   

\begin{prop}  \label{p.2.4} 
  A random probability measure $P$ is   the Dirichlet 
process with parameter $(a,H)$ if for any measurable partition $(A_1, \cdots, A_n)$ of  $\XX$
(i.e. $A_1, \cdots, A_n\in \mathcal{X}$, $A_1\cup \cdots \cup A_d=\XX$
and $A_i\cap A_j=\emptyset$ for $1\le i<j\le n$),
the random vector $(P(A_1), \cdots, P(A_n))$ follows the  Dirichlet  distribution with parameters 
$(a H(A_1), \cdots, a H(A_n))$. 
\end{prop} 

\begin{proof}  We refer to   \citet{sethuraman1994} or 
\citet{huzhang} for the proof of the equivalence between  
Definition \ref{d.2.3} and Proposition \ref{p.2.4}. 
\end{proof} 

In the literature   Proposition \ref{p.2.4} is usually  taken as the definition of the 
Dirichlet process. The reason that
we use the  stick-breaking representation as its definition  is for the consistency purpose  since most processes studied in this paper could be defined through the
stick-breaking representation \eqref{1.1}-\eqref{1.111}.

In the following definitions we  shall always assume   that 
 $P$ is  a random probability measure admitting  the representation 
\eqref{1.1}-\eqref{1.111}.   If we drop the specific distribution 
satisfied by $v_i$ in the definition \ref{d.2.3},  our  limiting results
below still hold. For this reason and for the potential applications in   practice 
 we introduce the concept of {\it  Dirichlet process with general stick-breaking weights}. 
\begin{definition}\label{d.2.5} $P$ is called the Dirichlet process with general stick-breaking weights,
denoted by $P\sim \DPG(g, H)$,
  if the stick-breaking weights  
$\{v_1, v_2, \cdots   \}$ in \eqref{1.1}-\eqref{1.111} are iid and follow 
 a general probability distribution
$g (x), 0<x<1$. 
\end{definition} 
Now we recall some other well-known processes studied in the literature, 
which are the subjects of this paper. 
\begin{definition}[\citet{pitman1997}]\label{d.2.6}   
Let     $b \in (0,1)$    and   let  $ -b<a<\infty$.   $P$ is called the  
two-parameter Poisson-Dirichlet process or
the  Pitman-Yor process, denoted by $\PDP(a, b, H)$,
  if the stick-breaking weights  satisfy the following:
\begin{equation}
\begin{cases}
\hbox{$v_1, v_2, \cdots$ are  independent}\,, \\
v_i \  {\sim}\ 
\text{\rm Beta}(1-b,a+ib), \quad  i=1, 2, \cdots  
\end{cases}\label{e.4.1} 
\end{equation} 
\end{definition}
\begin{definition}
  [\citet{favaro2012}] \label{d.2.7}  $P$ is called the normalized inverse Gaussian process with parameters $a$ and $H$, denoted by  $P \sim \hbox{\rm N-IGP}(a,H)$,   
if the joint distributions of the 
stick-breaking weights  $\{v_1, v_2, \cdots\}$ 
 are given through the following conditional probability densities  recursively:  
\begin{empheq}[left=\empheqlbrace]{align} 
& f_{v_1}(x) =\frac{a^{\frac{1}{2}}x^{-\frac{1}{2}}(1-x)^{-1}}{(2\pi)^{\frac{1}{2}}K_{-{\frac{1}{2}}}(a)}K_{-1}\left(\frac{a}{\sqrt{1-x}}\right),\label{5.1}\\
& f_{v_n|v_1,\cdots,v_{n-1}}(x) =\frac{a^{\frac{1}{2}}\prod_{i=1}^{n-1}(1-v_i)^{-\frac{1}{4}}x^{-\frac{1}{2}}(1-x)^{-\frac{5}{4}+\frac{n}{4}}}{(2\pi)^{\frac{1}{2}}K_{-{\frac{n}{2}}}\left(\frac{a}{\sqrt{\prod_{i=1}^{n-1}(1-v_i)}}\right)}   \nonumber\\
&\qquad\qquad\qquad\qquad\qquad  \times K_{-\frac{1}{2}-\frac{n}{2}}\left(\frac{a}{\sqrt{(1-x)\prod_{i=1}^{n-1}(1-v_i)}}\right),\nonumber\\
&\qquad\qquad\qquad\qquad\qquad\qquad  n=2, 3,   \cdots\label{5.2}
\end{empheq}
where $a>0$ and $K_\mu 
$ is the modified Bessel function of the third type 
(see e.g.  \citet{gradshteyn2014}).
\end{definition}
Similar to what  we did for  the  Dirichlet process, we present the original definition of the normalized inverse Gaussian process as a proposition.

\begin{prop}  \label{p.2.7} 
  A random probability measure $P$ is   the normalized inverse Gaussian process with parameter $(a,H)$ if for any measurable partition $(A_1, \cdots, A_n)$ of  $\XX$,
the random vector $(P(A_1), \cdots, P(A_n))$ follows the  normalized inverse Gaussian  distribution with parameters 
$(a H(A_1), \cdots, a H(A_n))$ given by  following form:  
\begin{align*}
f(x_1,\cdots,x_n)=& \frac{e^aa^n\prod_{i=1}^nH(A_i)}{2^{\frac{n}{2}-1}\pi^{\frac{n}{2}}}\times K_{-\frac{n}{2}}\left(\sqrt{\sum_{i=1}^n\frac{\left(aH(A_i)\right)^2}{x_i}}\right)\\
&\times \left(\sum_{i=1}^n\frac{\left(aH(A_i)\right)^2}{x_i}\right)^{-\frac{n}{4}}
\times \prod_{i=1}^nx_i^{-\frac{3}{2}}\times \mathbbm{1}_{\mathbb{S}}(x_1,\cdots,x_n)\,,
\end{align*}
where $\mathbb{S}$ is the simplex defined by \eqref{e.2.3a}. 
\end{prop}
\begin{proof}
We refer to \citet{favaro2012} for the proof of the equivalence between 
Definition \ref{d.2.7} and Proposition \ref{p.2.7}.
\end{proof}

\begin{definition}[\citet{favaro2016}] \label{d.2.8} $P$ is called the 
  normalized generalized gamma process with parameters $\sigma
  \in (0, 1)$, $a>0$ and $H$, denoted by  $P \sim \NGGP(\sigma,a,H)$,   
if  the finite dimensional joint distributions of the stick-breaking weights  $\{v_1, v_2, \cdots\}$ are given  by   the following conditional distributions:
\begin{empheq}[left=\empheqlbrace]{align}   
&f_{v_1}(x)=\frac{x^{-\sigma}(1-x)^{\sigma-1}e^a}{\Gamma(\sigma)\Gamma(1-\sigma)}\sum_{j=0}^{\infty}\frac{(1-\sigma)_{j}}{j!}\frac{a^{\frac{j}{\sigma}}}{(1-x)^j}\Gamma\left(1-\frac{j}{\sigma};\frac{a}{(1-x)^{\sigma}}\right),\label{6.1}\\
&f_{v_n|v_1,\cdots,v_{n-1}}(x)=\frac{\sigma\Gamma((n-1)\sigma)x^{-\sigma}(1-x)^{n\sigma-1}}{\Gamma(1-\sigma)\Gamma(n\sigma)}\nonumber\\
&\qquad\qquad\qquad  \frac{\sum_{j=0}^{\infty}\frac{(1-n\sigma)_{j}}{j!}\frac{a^{\frac{j}{\sigma}}}{(1-x)^j\prod_{i=1}^{n-1}(1-v_i)^j}\Gamma\left(n-\frac{j}{\sigma};\frac{a}{(1-x)^{\sigma}\prod_{i=1}^{n-1}(1-v_i)^{\sigma}}\right)}{\sum_{j=0}^{\infty}\frac{(1-(n-1)\sigma)_{j}}{j!}\frac{a^{\frac{j}{\sigma}}}{\prod_{i=1}^{n-1}(1-v_i)^j}\Gamma\left(n-1-\frac{j}{\sigma};\frac{a}{\prod_{i=1}^{n-1}(1-v_i)^{\sigma}}\right)},\nonumber\\
&\qquad\qquad\qquad\qquad\qquad\qquad  n=2, 3,  \cdots\,, \label{6.2} 
\end{empheq}
where $\Gamma(c,x)=\int_x^\infty 
    u^{c-1} e^{-u}du$ is the upper incomplete gamma function.
\end{definition} 

\begin{definition}\label{d.2.9} 
 We call a random probability measure $P$ on $(\Om,\mathcal{F})$   the generalized Dirichlet process with parameters $a>0$, $r\in \mathbb{N}^+$ and $H$, denoted by 
  $P \sim \GDP(a,r,H)$, if for any measurable partition $(A_1,\cdots,A_n)$ of $\mathbb{X}$, the joint density of $(P(A_1), \cdots, P(A_
 {n}))$  is  given by 
\begin{align}
 f(x_1, \cdots,x_{n})=&\frac{(r!)^a}{\prod_{i=1}^n \Gamma(ra_i)}\int_0^{\infty} t^{ra-1}e^{-rt}  \left[\prod_{j=1}^{n}\Phi_2^{(r-1)}\left(a_j{\bf{I}}_{r-1};ra_j;tx_j{\bf{J}}_{r-1}\right)\right]dt\nonumber\\
 &\times \prod_{i=1}^nx_i^{ra_i-1} \times \mathbbm{1}_{\mathbb{S}}(x_1,\cdots,x_n),\label{7.1}
\end{align}
where $a_i=aH(A_i)$; ${\bf I}_{r-1}=(1,\cdots,1)^T$, ${\bf J}_{r-1}=(1,\cdots,r-1)$ 
are $r-1$ 
dimensional vectors and $\Phi_2^{N}({\bf b};c;{\bf x})$ is the confluent form of the fourth Lauricella hypergeometric function (see e.g. \citet{exton1976}), and
  $\mathbb{S}$ is the simplex defined by \eqref{e.2.3a}. 
\end{definition}

It is trivial to verify  that the Dirichlet process is a special case of
the generalized Dirichlet process with parameter $r=1$.   Although 
the expression  
\eqref{7.1} looks very sophisticated,  its    mean,  variance, and    predictive  distribution  have been computed   (see e.g. \citet{lijoi2005b}).  
This process  also admits a stick-breaking 
representation (e.g.  \citet{favaro2016}).  
However,   the corresponding 
  stick-breaking representation is more complicated to use 
for our study  of the limiting theorems.  So,  we rather use  this 
sophisticated finite dimensional distributions than 
the more sophisticated  stick-breaking representation,   which we omit. 
 
As we are presenting  the functional central limit theorem
of  $P$,  we need  to  recall the   definition of the  Brownian bridge process
of  parameter $H$ (see e.g. \citet{kim2003} for more details). 
\begin{definition}Let $H$ be a  measure on $(\XX, \cX)$ 
and let   $B_H^o=(B_H^o (\om, A), \om\in \Om, A\in\cX)$ be a stochastic process
(random measure) with parameter $A\in \cX$.   It is called the Brownian 
bridge with parameter $H$ if the following two conditions are satisfied. 
\begin{enumerate}
\item[(i)] $B_H^o$ is Gaussian. Namely, for any elements $A_1, \cdots, A_n\in \cX$,  $B_H^o (  A_1  ), \cdots,
\\ 
B_H^o (  A_n )$ are   jointly centered  Gaussian random variables on the probability space $(\Om, \cF,  \PP)$. 
\item[(ii)] For any $A_1, A_2\in \cX$, the  
 covariance of $B(A_1)$ and $B(A_2)$ is given by 
\begin{equation}
\EE \left[ B(A_1) B(A_2)\right]=H(A_1\cap A_2)-H(A_1)H(A_2)\,. 
\end{equation}
\end{enumerate}
\end{definition}


%
To state the functional central limit 
theorem we also need the space $D(\RR^d)$ introduced in
 Section 3 of \citet{bickel1971}. The  characteristics  
  of    the elements (functions)  in $D(\RR^d)$ are given by  their continuity properties 
  described as follows. For   $1 \leq p\leq d$, let  $R_p$ be one of the relations $<$ or  $\geq$  and for $t 
  =(t_1, \cdots, t_d)\in \RR^d$  let $\mathcal{Q}_{R_1,\cdots,R_d}$ be the quadrant
\[
\mathcal{Q}_{R_1,\cdots,R_d}:=\left\{(s_1,\cdots,s_d)\in \RR^d :\, s_p R_p t_p,\, 1\leq p \leq d\right\}.
\]
Then, $x \in D(\RR^d)$ if and only if 
 (see e.g. \citet{straf1972}) for each $t\in \RR^d$,
the following two conditions hold:
 (i) $x_{\mathcal{Q}}=\lim_{s \rightarrow t,\, s \in \mathcal{Q}}x(s)$ exists for each of the $2^d$ quadrants $\mathcal{Q}=\mathcal{Q}_{R_1,\cdots, R_d}(t)$
(namely, for all the combinations that $R_1=``<"$,  or $``\ge" $, $\cdots$, $R_d=``<"$  or $ ``\ge"$), and (ii) $x(t)=x_{\mathcal{Q}_{\geq, \cdots, \geq}}$. In 
other words, $D(\RR^d)$ is the space of 
functions that are ``continuous from above with limits from below", which are similar   to  the space of 
the c\`{a}dl\`{a}g (French word abbreviation for 
``right continuous with left limits") functions
in one variable   (i.e. $d=1$). The metric   on $D(\RR^d)$ is introduced as follows. Let $\Lambda=\{\lambda:\RR^d\rightarrow \RR^d: \, \lambda(t_1,\cdots, t_d)=(\lambda_1(t_1),\cdots,\lambda_d(t_d))\}$, where each $\lambda_p: \RR \rightarrow \RR$ is continuous, strictly increasing and 
 has limits at both infinities.   Denote the  Skorohod  distance between $x,y \in D(\RR^d)$   by    
$$d(x,y)= \inf\{\min (\parallel x-y\lambda \parallel, \parallel \lambda \parallel):\, \lambda \in \Lambda\},$$
where $\displaystyle \parallel x-y\lambda \parallel= \sum_{n=1}^\infty 
\sup_{|t|\le n}    |x(t)-y(\lambda(t))| $ and $
\displaystyle \parallel \lambda \parallel=\sum_{n=1}^\infty 
\sup_{|t|\le n}   |\lambda(t)-t| $.  
 
\def\BB{\mathbb{B}}
Having introduced the metric  space 
$D(\RR^d)$  we can now explain  the concept of  weak convergence of a random measure
  on this space  with respect to its  
 Skorohod topology (the topology on $D(\RR^d)$ induced by the
Skorohod distance $d(x,y)$).
 Let $\QQ_a: \Om\times \cB(\RR^d)\rightarrow [0, 1]$  be a family of  random probability measures 
  depending on a parameter $a>  0$ and  let $\BB: 
  \Om\times \cB(\RR^d)\rightarrow [0, 1]$  be another random probability measure. 
 Define 
 \[
 \QQ_a(t_1, \cdots, t_d)=\QQ_a((-\infty, t_1]\times \cdots \times
 (-\infty, t_d])\,,
 \quad (t_1, \cdots, t_d)\in \RR^d\,. 
 \]
\begin{definition}
We say $\QQ_a$ converges to $\BB$ weakly on $D(\RR^d)$ with respect to the Skorohod topology,
denote  $\QQ_a\stackrel{weakly}\rightarrow \BB$ in $D(\RR^d)$,  if for any bounded continuous 
(continuous 
with respect to Skorohod topology) functional  $f:D(\RR^d)\rightarrow \RR$
we have
\begin{equation}
\lim_{a\rightarrow \infty} \EE\!\! \left[f(\QQ_a(\cdot, \cdots,\cdot))
\right]=\EE\!\! \left[f(\BB(\cdot, \cdots,\cdot))\right]\,. 
\end{equation}
  \end{definition}  
\subsection{
Notations}
Assume that  the random probability measure $P$ is  a stick-breaking process having
the stick-breaking representation \eqref{1.1}-\eqref{1.111}.   For any   $A\in \mathcal{X}$,
\begin{eqnarray}
\EE(P(A))
&=& \EE\left[ \sum_{i=1}^{\infty}w_i \delta_{\theta_i}(A)\right]=\sum_{i=1}^{\infty}\EE(w_i) \EE \left[\mathbbm{1} _A(\theta_i)\right]\nonumber\\
&=&\sum_{i=1}^{\infty}\EE(w_i) H(A)=
\EE\left[\sum_{i=1}^{\infty} w_i\right]  H(A)=H(A), \label{e.2.9a} 
\end{eqnarray} 
%
%
since $\sum_{i=1}^{\infty} w_i=1$ a.s. We can also obtain its variance as follows.
\begin{align}
	 \Var\left[P(A)\right]&= \mathbb{E}\left[\left(P(A)-H(A)\right)^2\right]=\mathbb{E}\left[\left(\sum_{i=1}^{\infty}w_i\left(\delta_{\theta_i}(A)-H(A)\right)\right)^2\right]\nonumber\\
	&=\mathbb{E}\left[\sum_{i=1}^{\infty}w_i^2\left(\delta_{\theta_i}(A)-H(A)\right)^2\right]\nonumber\\
	&\qquad +2\mathbb{E}\left[\sum_{1 \leq i<j<\infty}w_iw_j\left(\delta_{\theta_i}(A)-H(A)\right)\left(\delta_{\theta_j}(A)-H(A)\right)\right]\nonumber\\
	&=\mathbb{E}\left[\left(\delta_{\theta_i}(A)-H(A)\right)^2\right]\mathbb{E}\left[\sum_{i=1}^{\infty}w_i^2\right]\nonumber\\
	&=H(A)(1-H(A))\mathbb{E}\left[\sum_{i=1}^{\infty}w_i^2\right].\label{1.112}
\end{align} 
Based on the expectation and variance of $P$, 
we introduce the following  quantities that are investigated in the main theorems:  
\begin{align}
 & D_a(\cdot)=\frac{P(\cdot)-\mathbb{E}[P(\cdot)]}{\sqrt{\Var[P(\cdot)]}}=\frac{P(\cdot)-H(\cdot) }{\sqrt{H(A)(1-H(A))\mathbb{E}\left[\sum_{i=1}^{\infty}w_i^2\right]}},\label{1.2}
\end{align}
where the last identity follows from \eqref{e.2.9a}-\eqref{1.112}. 
Up to a constant we may  just consider the following quantity 
for the notational simplicity: 
\begin{align} 
 &Q_{H,a}(\cdot)=\frac{P(\cdot)-\mathbb{E}[P(\cdot)]}{\sqrt{\mathbb{E}\left[\sum_{i=1}^{\infty}w_i^2\right]}}.\label{1.16}
\end{align}

\section{Moment results}\label{s.moment}
We use the method of moments to show the announced 
asymptotics. This requires to have a nice 
estimates of the moments of the process $P$, which in turn  requires some nice
bounds for    the moments of $\{w_i\}_{i=1}^\infty$.
Thus, in this section we present the asymptotic behaviors of the joint moments of $w_i$'s for various  processes  introduced  in  the  previous    section. These results  will play a key role  in the proofs of  our main theorems. On the other hand, they also have their own interest. 

In the following proposition and  through out the paper we use the notation $p_{m:n}:=\sum_{i=m}^{n}p_i$ for $m \leq n$.

\begin{prop}\label{prop.3.1}
 Let $P \sim \DPG(g_a,H)$, i.e., the stick-breaking weights $v_i \overset{iid}{\sim} g_a(x), 0<x<1$, where $a>0$ is a certain  parameter. We assume that $v_i$ is not identically $0$.
 If $\displaystyle \lim_{a \rightarrow \infty} \frac{\mathbb{E}[v_1^{n+1}]}{\mathbb{E}[v_1^n]}=0$ for all $n \in \ZZ_+$
 (set of nonnegative integers), then for any nonnegative integers $m,n$,
\begin{empheq}[left=\empheqlbrace]{align} 
	&\mathbb{E}\left[v_i^n(1-v_i)^m\right]=\mathbb{E}[v_1^n]+o\left(\mathbb{E}[v_1^n]\right),\label{0.1}\\
	&\sum_{j=0}^{\infty}\left(\mathbb{E}\left[(1-v_i)^m\right]\right)^j=\frac{1}{m\mathbb{E}[v_1]}+o\left(\frac{1}{m\mathbb{E}[v_1]}\right)\,. \label{0.2}
\end{empheq} 
Furthermore, let the sequence $\{w_i\}_{i=1}^{\infty}$ be defined as in \eqref{1.111}, and let $p_1, \cdots, p_k$ be nonnegative integers. Then
\begin{equation}
	\begin{aligned}
    &\mathbb{E}\left[\sum_{1\leq i_1<i_2<\cdots<i_k<\infty}w_{i_1}^{p_1}w_{i_2}^{p_2}\cdots w_{i_k}^{p_k}\right]\\
    &=\frac{\mathbb{E}[v_1^{p_1}]\cdots\mathbb{E}[v_1^{p_k}]}{p_{1:k}p_{2:k} \cdots p_{k:k}\left(\mathbb{E}[v_1]\right)^k}+o\left(\frac{\mathbb{E}[v_1^{p_1}]\cdots\mathbb{E}[v_1^{p_k}]}{\left(\mathbb{E}[v_1]\right)^k}\right).\label{0.3}
	\end{aligned}
\end{equation}
In particular, when   $p_j=2$ for all $j\in\{ 1, \cdots, k\}$
(hence $p_{1:k}=2k$), the asymptotics \eqref{0.3} becomes  
\begin{equation}
\begin{aligned}
&\mathbb{E}\left[\sum_{1\leq i_1<i_2<\cdots<i_k<\infty}w_{i_1}^2w_{i_2}^2\cdots w_{i_k}^2\right]=\frac{1}{2^kk!}\left(\frac{\mathbb{E}[v_1^2]}{\mathbb{E}[v_1]}\right)^k+o\left(\left(\frac{\mathbb{E}[v_1^2]}{\mathbb{E}[v_1]}\right)^k\right).\label{0.4}
\end{aligned}
\end{equation}
\end{prop}

\begin{prop}\label{prop.3.2}
Let $P \sim \PDP(a,H)$.  Namely,  let the stick-breaking weights $v_1, v_2, \cdots $ be  
 given by  \eqref{e.4.1} and let $w_1, w_2, \cdots$ 
be constructed from these $v_i$'s by \eqref{1.111}. Let $p_1, \cdots, p_k$ be nonnegative  integers. 
Then,  we have the following identity. 
    \begin{align}
    &\mathbb{E}\left[\sum_{1\leq i_1<i_2<\cdots<i_k<\infty}w_{i_1}^{p_1}w_{i_2}^{p_2}\cdots w_{i_k}^{p_k}\right]\nonumber\\
    &=\frac{1}{(a+kb)(a+1)_{(p_{1:k}-1)}}\prod_{i=1}^k\  \frac{(1-b)_{p_i}(a+bi)}{ p_{i:k}-(k-i+1)b }\,. \label{1.2.3}
    \end{align}
In particular, when   $p_j=2$ for all $j\in\{ 1, \cdots, k\}$, the
above  expectation becomes 
\begin{align}
\mathbb{E}\left[\sum_{1\leq i_1<i_2<\cdots<i_k<\infty}w_{i_1}^2w_{i_2}^2\cdots w_{i_k}^2\right]=\frac{(1-b)^k(a+b)\cdots\left(a+b(k-1)\right)}{k!(a+1)\cdots(a+2k-1)}.\label{1.2.4}
\end{align}
\end{prop}

 \begin{prop}\label{prop.3.3}
Let $P \sim \N-IGP(a,H)$.  Namely,  let   the stick-breaking weights  $\{v_i\}_{i=1}^{\infty}$  be  given by \eqref{5.1}-\eqref{5.2}  and let $\{w_i\}_{i=1}^{\infty}$
be given by \eqref{1.111}.   Then,  for any  positive integer $p$, we have
\begin{align}
&\mathbb{E}\left[\sum_{n=1}^{\infty}w_n^p\right]={O}\left(\frac{1}{a^{p-1}}\right)
\quad \hbox{as $a\rightarrow \infty$}\,,\label{5.01}\\
&\mathbb{E}\left[\sum_{1 \leq i_1<i_2<\cdots <i_k<\infty}w_{i_1}^{p_1}w_{i_2}^{p_2}\cdots w_{i_k}^{p_k}\right]={O}\left(\frac{1}{a^{p_{1:k}-k}}\right) \quad \hbox{as $a\rightarrow \infty$}\,. \label{5.9}
\end{align}
Furthermore,  when $p=p_1=\cdots=p_k=2$, we have 
\begin{align}
&\mathbb{E}\left[\sum_{n=1}^{\infty}w_n^2\right]=\frac{1}{a}+o\left(\frac{1}{a}\right)
\quad \hbox{as $n\rightarrow \infty$}\,,\label{5.02}\\
&\mathbb{E}\left[\sum_{1 \leq i_1<i_2<\cdots <i_k<\infty}w_{i_1}^{2}w_{i_2}^{2}\cdots w_{i_k}^{2}\right]=\frac{1}{k!a^k}+o\left(\frac{1}{a^k}\right)\,.\label{5.10}
\end{align}
[namely, the leading coefficient in \eqref{5.01} is $1$ and the leading coefficient in \eqref{5.9} is 
$\frac{1}{k! }$.]
\end{prop}

\begin{prop}\label{prop.3.4}
Let $P\sim \NGGP(\sigma,a,H)$.  Namely,  let
 the distribution of the stick-breaking weights $\{v_1, v_2, \cdots\}$ be given by \eqref{6.1}-\eqref{6.2} and let $\{w_i\}_{i=1}^{\infty}$ be defined by  \eqref{1.111}.  For any  positive integers $p$ and $p_1,\cdots, p_k$, we have
\begin{align}
&\mathbb{E}\left[\sum_{n=1}^{\infty}w_n^p\right]={O}\left(\frac{1}{a^{p-1}}\right) \quad\hbox{as $a\rightarrow \infty$},\label{6.3}\\
&\mathbb{E}\left[\sum_{1 \leq i_1<i_2<\cdots <i_k<\infty}w_{i_1}^{p_1}w_{i_2}^{p_2}\cdots w_{i_k}^{p_k}\right]={O}\left(\frac{1}{a^{p_{1:k}-k}}\right) \quad\hbox{as $a\rightarrow \infty$}.\label{6.4}
\end{align}
Furthermore, when $p=p_1=\cdots=p_k=2$ 
   and when $\sigma=\frac{1}{m}$ for some arbitrarily fixed 
    integer $m \geq 2$, we have  
\begin{align}
&\mathbb{E}\left[\sum_{n=1}^{\infty}w_n^2\right]=\frac{1}{a}+o\left(\frac{1}{a}\right) \quad\hbox{as $a\rightarrow \infty$},\label{6.5}\\
&\mathbb{E}\left[\sum_{1 \leq i_1<i_2<\cdots <i_k<\infty}w_{i_1}^{2}w_{i_2}^{2}\cdots w_{i_k}^{2}\right]=\frac{1}{k!a^k}+o\left(\frac{1}{a^k}\right) \quad\hbox{as $a\rightarrow \infty$}.\label{6.6}
\end{align} 
\end{prop}

\begin{prop}\label{prop.3.5}
Let $P\sim \GDP(a,r,H)$ and let $p_1,\cdots,p_k$ be nonnegative integers. Then, as $a \rightarrow \infty$,
 \begin{align}
 \mathbb{E}\left[\sum_{1 \leq i_1<i_2<\cdots <i_k<\infty}w_{i_1}^{p_1}w_{i_2}^{p_2}\cdots w_{i_k}^{p_k}\right]= {O}\left(\frac{1}{a^{p_{1:k}-k}}\right)\,. \label{e.2.11}
 \end{align}
 In particular, when   $p_j=2$ for all $j\in\{ 1, \cdots, k\}$, the
above  expectation becomes 
\begin{align}
\mathbb{E}\left[\sum_{1\leq i_1<i_2<\cdots<i_k<\infty}w_{i_1}^2w_{i_2}^2\cdots w_{i_k}^2\right]=\frac{\left[\frac{\sum_{k=1}^r\left(\frac{1}{k}\right)^2}{\left(\sum_{j=1}^r\frac{1}{j}\right)^2}\right]^k}{k!a^k}+o\left(\frac{1}{a^k}\right)\,. \label{e.2.12}
\end{align}
\end{prop}

\section{Main results}\label{s.main}
\subsection{Empirical strong law of large numbers}
The empirical strong law of large numbers
and the empirical Glivenko-Cantelli theorem play   undoubtedly  
  important roles in statistics. 
In this subsection we show the empirical strong law of large numbers
and the empirical Glivenko-Cantelli theorem for the various processes introduced
in Section \ref{s.pre}. 
But before we state our theorem, we need an  additional condition on the stick-breaking weights $v_i$ in the case of  Dirichlet process   with general stick-breaking weights. 

\begin{assumption}\label{a.5.1} 
Let the iid stick-breaking  weights $\{v_i\}$ satisfy 
     \begin{equation}\mathbb{E}[v_i^p]=\frac{C_{p}}{a^{k_p}}+o\left(\frac{1}{a^{k_p}}\right)\quad \hbox{as $a\rightarrow \infty$} \,, \label{e.3.27} 
     \end{equation}  
for any $p\in \NN$,  where $  k_p $ is a positive sequence satisfying  $jk_i \geq ik_j$ for $i \geq j$  and $  C_{ p} $ is a sequence of 
  finite constants, independent of $a$. 
  \end{assumption} 

\begin{theorem}\label{law of large number}
Let $P$ be  one of the Dirichlet 
process $\DP(a,H)$,   the Dirichlet process with general stick-breaking weights $\DPG(g_a,H)$  satisfying  Assumption \ref{a.5.1},  the two-parameter Poisson-Dirichlet process $\PDP(a,b,H)$, the  normalized inverse Gaussian process $\N-IGP(a,H)$, the normalized generalized gamma process $\NGGP(\sigma,a,H)$, 
   and the  generalized Dirichlet process $\GDP(a,r,H)$.    
     Assume that $a=n^{\tau} $  for some arbitrarily fixed $\tau>0$.   Then, as $n\rightarrow\infty$,
\begin{align}
P(A) \ \stackrel{a.s.}{\rightarrow}\   H(A)  \label{e.2.9} 
\end{align}
for any measurable $A\in \mathcal{X}$.
\end{theorem}
Once we have the empirical strong law of large numbers for  $P$, we can deduce  the empirical Glivenko-Cantelli theorem for  $P$ 
(see e.g.  Theorem 20.6 in \citet{billingsley1995}  for a general discussion).
\begin{theorem}\label{Glivenko-Cantelli}
Let $(\mathbb{X},\mathcal{X})=(\mathbb{R},\mathcal{B}(\mathbb{R}))$. Let $P$ be
one of the Dirichlet 
process $\DP(a,H)$,  the Dirichlet process with general stick-breaking weights $\DPG(g_a,H)$  satisfying  Assumption \ref{a.5.1}, 
   the two-parameter Poisson-Dirichlet process \\$\PDP
   (a,b,H)$, the normalized inverse Gaussian process $\N-IGP(a,H)$, the normalized generalized gamma process $\NGGP(\sigma,a,H)$, 
   and  the  generalized Dirichlet process $\GDP(a,r,H)$.    Assume that $a=n^{\tau} $  for some arbitrarily fixed $\tau>0$.  Then, as $n\rightarrow\infty$,
\begin{align*}
\sup_{x\in \mathbb{R}} |P\left((-\infty,x]\right)- H\left((-\infty,x]\right)|\overset{a.s.}{\rightarrow} 0\,. 
\end{align*} 
\end{theorem}

\subsection{Central limit theorems  and functional central limit theorems} 

In this subsection, we state the central limit theorems      corresponding to the strong law of large 
numbers of the form \eqref{e.2.9}.

We shall state the central limit theorems and
functional  central limit theorems for various processes 
 as the following  three theorems. 
The first one is for the Dirichlet processes with general iid stick-breaking weights defined   by Definition \ref{d.2.5}.
We will assume mild  convergence conditions on the 
stick-breaking  weights.  
\begin{theorem}\label{theorem 0.1}
Let $P$ be a stick-breaking process {
with general stick-breaking weights 
 $v_1, v_2, \cdots$ (whose distributions) depending on a parameter $a>0$ (we omit the explicit 
dependence on $a$ of the $v_i$'s). }  
 Let $D_a$   
and $Q_{H,a}$ be defined by  \eqref{1.2} and (\ref{1.16})
respectively. 
 Assume that the stick-breaking weights $v_1, v_2, \cdots  $ are iid and satisfy 
 the following two conditions. 
\begin{itemize}
     \item[(i)] For all $n \in \ZZ^+$, we have
     \begin{equation}
     \displaystyle \lim_{a \rightarrow \infty} \frac{\mathbb{E}[v_1^{n+1}]}{\mathbb{E}[v_1^n]}=0\,. \label{e.3.18} 
     \end{equation} 
     \item[(ii)] For any multi-index $(p_1, \cdots, p_k)$ such that $p_i \geq 2$ and   
     $  \frac{p_{1:k}}{2}>k $, where $p_{1:k}=\sum_{i=1}^k p_i$,   we have
     \begin{align}
     \lim_{a \rightarrow \infty}\frac{\left(\mathbb{E}[v_1]\right)^{\frac{p_{1:k}}{2}-k}\prod_{i=1}^k\mathbb{E}[v_1^{p_i}]}{\left(\mathbb{E}[v_1^2]\right)^{\frac{p_{1:k}}{2}}}=0.\label{0.10}
     \end{align}
\end{itemize}  
    Then we have the following results. 
  \begin{enumerate}
\item[(i)](Central limit theorem) Let $A_1,A_2,\cdots, A_n$ be any disjoint   measurable subsets of $\mathbb{X}$.  
Then, as $a \rightarrow \infty$,
\begin{equation}
\left(D_a(A_1),D_a(A_2),\cdots,D_a(A_n)\right) \overset{d}{\rightarrow} (X_1,X_2,\cdots,X_n)\,, \label{e.3.19} 
\end{equation}
where $(X_1,X_2,\cdots,X_n) \sim N (0, \Sigma)$ and $\Sigma
=(\si_{ij})_{1\le i,j\le n}$ is  given by 
\begin{align}
    \sigma_{ij}= \begin{cases}
    1 & \qquad \mbox{if $i=j$}\,,  \\
    -\sqrt{\frac{H(A_i)H(A_j)}{\left(1-H(A_i)\right)\left(1-H(A_j)\right)}}& \qquad \mbox{if $i \neq j$}\,.  \\
\end{cases} \label{e.2.24a}
\end{align}
\item[(ii)](Functional central limit theorem)
Let $(\XX, \cX)=(\RR^d, \cB(\RR^d))$ be the $d$-dimensional Euclidean space
with the Borel $\si$-algebra. Then 
\begin{equation}
Q_{H,a}  \stackrel{{\rm weakly}}{\rightarrow }  \quad {B_H^o}\qquad  {\rm in}\qquad D(\RR ^d)\  \label{e.2.24} 
\end{equation}  
with respect to the Skorohod topology. 
\end{enumerate} 
\end{theorem}
\begin{remark} For central limit theorem we use $D_a$ 
because each component converges to a standard Gaussian. 
For functional central limit theorem we use $Q_{H,a}$
since it converges to a Brownian bridge with parameter $H$. 
We can presumably use  $D_a$ 
(or  $Q_{H,a}$) in both \eqref{e.2.24a} and \eqref{e.2.24}
with a scaling.  
\end{remark} 

%

The  conditions (i) and (ii) in Theorem \ref{theorem 0.1}  are  implied 
 by many other conditions. One  of them  is given below.
 \begin{remark}\label{remark1} 
 Assumption \ref{a.5.1} implies the conditions (i) and (ii) in Theorem \ref{theorem 0.1}. 
\end{remark}
\begin{proof}
It is obviously  
that $\{k_p\}$ is an increasing sequence, and thus the condition (i)
of  Theorem \ref{theorem 0.1}
 (i.e \eqref{e.3.18}) holds. 

For any nonnegative  integer $m$, let $\mathfrak{N}$ be  a 
certain collection of integers $j's$ such that $\sum_{j\in \mathfrak{N}}j=m$. The condition (ii) in Theorem \ref{theorem 0.1} is equivalent to the following statement:

If $j \geq 2$  and
 $|\mathfrak{N|<\frac{m}{2}}$,  then 
     \begin{align}
     \frac{m(k_2-k_1)}{2}<\sum_{j\in \mathfrak{N}}k_j-|\mathfrak{N}|k_1 .  \label{0.17}
     \end{align} 
Thus, to prove \eqref{0.10}  it is sufficient to show $\frac{mk_2}{2}<\sum_{j\in \mathfrak{N}}k_j$.  This is  a simple
consequence of   $jk_i \geq ik_j$ for $i \geq j$. In fact, taking  $i=2$,  we have for all  $j \geq 2$,  $2k_j \geq jk_2$ holds and thus  we have
$\sum_{j\in \mathfrak{N}}2k_j \geq \sum_{j\in \mathfrak{N}}jk_2$, which implies $\frac{mk_2}{2}<\sum_{j\in \mathfrak{N}}k_j$. Hence we have \eqref{0.17}. 
\end{proof}

The conditions (i) and (ii) in  Theorem \ref{theorem 0.1} are satisfied by many 
  interesting processes including the Dirichlet process.  We give three examples to illustrate the applicability  of our above theorem. 
\begin{corollary}\label{ex.3.7} 
Let $P \sim \DP(a,H)$ be  defined  as  in Definition \ref{d.2.3},  namely, $v_i \overset{iid}{\sim} \hbox{\rm Beta}(1,a)$.  Let $D_a$   
and $Q_{H,a}$ be defined by  \eqref{1.2} and (\ref{1.16})
respectively.  We have the following results. 
 \begin{enumerate}
\item[(i)] Let $A_1,A_2,\cdots, A_n$ be any disjoint   measurable subsets of $\mathbb{X}$.  
Then, as $a \rightarrow \infty$,
\begin{equation}
\left(D_a(A_1),D_a(A_2),\cdots,D_a(A_n)\right) \overset{d}{\rightarrow} (X_1,X_2,\cdots,X_n)\,, \label{e.2.27} 
\end{equation} 
where $(X_1,X_2,\cdots,X_n) \sim N (0, \Sigma)$ and $\Sigma
=(\si_{ij})_{1\le i,j\le n}$ is  given by 
\begin{align} 
    \sigma_{ij}= \begin{cases}
    1 & \qquad \mbox{if $i=j$}\,,  \\
    -\sqrt{\frac{H(A_i)H(A_j)}{\left(1-H(A_i)\right)\left(1-H(A_j)\right)}}& \qquad \mbox{if $i \neq j$}\,.  \\
\end{cases} 
\end{align}
\item[(ii)] 
Let $(\XX, \cX)=(\RR^d, \cB(\RR^d))$ be the $d$-dimensional Euclidean space
with the Borel $\si$-algebra. Then 
\begin{equation}
Q_{H,a}  \stackrel{{\rm weakly}}{\rightarrow }  \quad {B_H^o}\qquad  {\rm in}\qquad D(\RR ^d)\ 
\end{equation}  
with respect to the Skorohod topology.
\end{enumerate} 
\end{corollary}
\begin{proof}
It is sufficient to verify the condition 
\eqref{e.3.27} in Assumption \ref{a.5.1}. Since $v_i \overset{iid}{\sim} \hbox{Beta}(1,a)$, we have
for any positive integer $p$,
\begin{align}\nonumber
	&\mathbb{E}[v_i^p]=\frac{\Gamma(a+1)\Gamma(p+1)}{\Gamma(1)\Gamma(a+p+1)}=\frac{p!}{(a+1)\cdots(a+p)}=\frac{p!}{a^p}+o\left(\frac{1}{a^p}\right).
\end{align}
Hence, $k_p=p$ and $C_p=p!$. Obviously, for $i \geq j$, $jk_i \geq ik_j$   always holds true.
\end{proof}
\begin{remark}
Since  the posterior of the Dirichlet process is still a Dirichlet process, the above result  can be applied to  the posterior process in the  Bayesian nonparametric models   when the prior is the Dirichlet process   for the following situations:
(i) with large sample size and finite parameter $a$;  (ii)  with large 
parameter  $a$ and finite sample size, (iii)  
with parameter  $a$ and sample size  both large. 
\end{remark}

The assumption  of the   $\Beta(1, a)$-distribution 
in Corollary \ref{ex.3.7} 
can be replaced   by a general $\Beta(\rho_a, a)$, where  
$\rho_a/a\rightarrow 0$. In fact,  in this case,   we have 
\[
\EE[v_1^n]=\left(\frac{\rho_a}{a}\right)^n+o\left(\left(\frac{\rho_a}{a}\right)^n\right)\,. 
\]
It is easy to verify that
 the conditions \eqref{e.3.18}-\eqref{0.10} in 
 Theorem \ref{theorem 0.1} are satisfied. Thus we have
\begin{corollary}\label{c.3.20} 
Let $P \sim \DPG(g_a,H)$ be  defined  as  
in Definition \ref{d.2.5},  where  $v_i \overset{iid}{\sim} \hbox{\rm Beta}(\rho_a,a)$
with $\displaystyle \lim_{a\rightarrow \infty}\frac{\rho_a} a=0$.
  Let $D_a$   
and $Q_{H,a}$ be defined by  \eqref{1.2} and (\ref{1.16})
respectively.  
We have the following results. 
 \begin{enumerate}
\item[(i)] Let $A_1,A_2,\cdots, A_n$ be any disjoint   measurable subsets of $\mathbb{X}$.  
Then, as $a \rightarrow \infty$,
\begin{equation}
\left(D_a(A_1),D_a(A_2),\cdots,D_a(A_n)\right) \overset{d}{\rightarrow} (X_1,X_2,\cdots,X_n)\,, \label{e.2.27.1} 
\end{equation} 
where $(X_1,X_2,\cdots,X_n) \sim N (0, \Sigma)$ and $\Sigma
=(\si_{ij})_{1\le i,j\le n}$ is  given by 
\begin{align} 
    \sigma_{ij}= \begin{cases}
    1 & \qquad \mbox{if $i=j$}\,,  \\
    -\sqrt{\frac{H(A_i)H(A_j)}{\left(1-H(A_i)\right)\left(1-H(A_j)\right)}}& \qquad \mbox{if $i \neq j$}\,.  \\
\end{cases} 
\end{align}
\item[(ii)] 
Let $(\XX, \cX)=(\RR^d, \cB(\RR^d))$ be the $d$-dimensional Euclidean space
with the Borel $\si$-algebra. Then  
\begin{equation}
Q_{H,a}  \stackrel{{\rm weakly}}{\rightarrow }  \quad {B_H^o}\qquad  {\rm in}\qquad D(\RR ^d)\ 
\end{equation}  
with respect to the Skorohod topology.  
\end{enumerate}  
\end{corollary} 
\begin{remark} It is not clear yet what is the finite dimensional distribution of stick-breaking process $P$ if the corresponding 
stick-breaking weights $v_i\stackrel{ind}{\sim } \Beta(\rho_a, a)$. 
\end{remark}

The next corollary  is about the asymptotic behaviour of  the  prior $P$, when the corresponding stick-breaking weights $v_i$ follow a linear combination of Beta distributions, whose  precise meaning is give below. 
\begin{definition} Let $s$ be any positive integer  and 
let $\{r_1, \cdots, r_s\} $  and $\{t_1, \cdots, t_s\} $ be two sets of  positive real numbers such that   $\sum_{\ell=1}^s t_\ell =1$.   Let $u_{1, 1}, \cdots, 
u_{1, s}, u_{2,1},\\
 \cdots, u_{2,s}, \cdots $ be independent and 
let $u_{i, \ell}  {\sim} \text{\rm Beta}(1,a^{r_\ell })\,, i=1, 2, \cdots, \ell=1, \cdots, s$.  Then the random variables
\begin{equation}
	v_i=\sum_{\ell =1}^s t_{ \ell}u_{i, \ell}  \,, \quad i=1, 2, \cdots, 
	\label{e.3.29}
	\end{equation} 
are  called   linear combinations of the beta distributions. 
\end{definition} 

\begin{corollary}
Let $P$ be the stick-breaking process as defined in Definition
 \ref{d.2.1}, where the weight $v_i$ is  a linear combination of the beta distributions
defined by \eqref{e.3.29}. 
  Let $D_a$   
and $Q_{H,a}$ be defined by  \eqref{1.2} and (\ref{1.16})
respectively.        Then the following statements hold true. 
 \begin{enumerate}
\item[(i)] Let $A_1,A_2,\cdots, A_n$ be any disjoint   measurable subsets of $\mathbb{X}$.  
Then, as $a \rightarrow \infty$,
\begin{equation}
\left(D_a(A_1),D_a(A_2),\cdots,D_a(A_n)\right) \overset{d}{\rightarrow} (X_1,X_2,\cdots,X_n)\,, \label{e.2.31} 
\end{equation} 
where $(X_1,X_2,\cdots,X_n) \sim N (0, \Sigma)$ and $\Sigma
=(\si_{ij})_{1\le i,j\le n}$ is  given by 
\begin{align}
    \sigma_{ij}= \begin{cases}
    1 & \qquad \mbox{if $i=j$}\,,  \\
    -\sqrt{\frac{H(A_i)H(A_j)}{\left(1-H(A_i)\right)\left(1-H(A_j)\right)}}& \qquad \mbox{if $i \neq j$}\,.  \\
\end{cases} 
\end{align}
\item[(ii)] 
Let $(\XX, \cX)=(\RR^d, \cB(\RR^d))$ be the $d$-dimensional Euclidean space
with the Borel $\si$-algebra. Then 
\begin{equation}
Q_{H,a}  \stackrel{{\rm weakly}}{\rightarrow }  \quad {B_H^o}\qquad  {\rm in}\qquad D(\RR ^d)\ 
\end{equation}  
with respect to the Skorohod topology. 
\end{enumerate} 

\end{corollary}
\begin{proof}
By the independence of $ \{u_{i, \ell}\}_{\ell=1}^s$, we can compute  the $p$-th moment of $v_i$ as follows. 
\begin{align}
	\nonumber
	&\mathbb{E}[v_i^p]=\mathbb{E}\left[\left(\sum_{\ell=1}^s t_{\ell}u_{i, \ell}\right)^p\right]=\sum_{{q_1, \cdots, q_s\in \ZZ_+}\atop {q_1+\cdots+q_s=p}} \binom{p}{q_1, \cdots, q_s} \prod_{\ell=1}^s \mathbb{E}\left[(t_{\ell}u_{i, \ell})^{q_\ell}\right]\\\nonumber
	&=\sum_{{q_1, \cdots, q_s\in \ZZ_+}\atop {q_1+\cdots+q_s=p}} \binom{p}{q_1, \cdots, q_s} \prod_{\ell=1}^s t_{\ell}^{q_\ell}\left(\frac{q_\ell!}{a^{q_\ell r_\ell}}+o\left(\frac{1}{a^{q_\ell r_\ell }}\right)\right)=\frac{t^pp!}{a^{pr}}+o\left(\frac{1}{a^{pr}}\right)\,, 
\end{align}
where $r=\min(r_1, \cdots, r_s)$.   
Taking  $k_p=pr$ and $C_{p}=t^pp!$ in Assumption \ref{a.5.1} we see the condition  $i \geq j$, $jk_i \geq ik_j$ is always verified.
\end{proof}

\begin{remark}
Let us return to  Corollary 
 \ref{ex.3.7}.  This is a typical case  and we take a close look of the 
  density  $f_a(x)=a(1-x)^{a-1}$, $0\le x\le 1$, of the Beta distribution $\Beta(1,a)$. 

For any continuous  function $g:\RR\rightarrow \RR$,  it is easy to verify that 
\[
\int_{\RR} \left[g(x)-g(1)\right] f_a(x) dx=\int_0^1 \left[g(x)-g(1)\right] f_a(x) dx\rightarrow 0 \quad \hbox{as $a\rightarrow \infty$}\,. 
\]
This means that  $\int_{\RR} g(x)  f_a(x) dx \rightarrow g(1)$. In other word,
$f_a$ converges to the Dirac delta function $\delta(x -1)$.
This observation hints that when the  distribution $f_a$ of $v_i$'s 
converges to the Dirac delta function $\delta(x-1)$, or the random variable 
$v_i$ converges in distribution to $1$ (as $a\rightarrow \infty$)
we should have the convergence of the random process $Q_{H,a}$. 
But we still need to impose some more technical conditions. We give a further illustration by the following corollary. 
\end{remark}

\begin{corollary}
Let the stick-breaking process $P$ be defined as in Definition \ref{d.2.1},  where  the corresponding $v_i$ follows the  following distribution:
\begin{equation}
\nonumber
    f_b(x)=  \begin{cases}
    \frac{1-g(b)}{b} & \mbox{if $0<x\leq b$}; \\
    \frac{g(b)}{1-b} & \mbox{if $b < x \leq 1$}, \\
\end{cases} 
\end{equation}
where $g(b)=e^{-\frac{1}{b^\epsilon}}$, for a certain  arbitrarily fixed $ \epsilon>0$. Then, as $b \rightarrow 0$, the conditions
\eqref{e.3.18} and \eqref{0.10} of  Theorem \ref{theorem 0.1} hold
with $a=\frac1b$.  
Thus the statements  \eqref{e.3.19} and \eqref{e.2.24} 
of  Theorem \ref{theorem 0.1} hold true. 
\end{corollary} 
\begin{proof} Before we proceed to the proof. Let us note the obvious fact 
that $f_b$
converges to the Dirac delta distribution $\delta(x-1)$. 

For any $ n>0$, we see $\lim_{b \rightarrow 0}\frac{g(b)}{b^n}=0$. A trivial calculation implies that for any positive integer $p$,
\begin{align}
	\nonumber
	\mathbb{E}[v_i^p]=\frac{b^p+g(b)\sum_{i=0}^pb ^i}{p+1}=\frac{b^p}{p+1}+o(b^p)\,. 
\end{align}
An application of Assumption \ref{a.5.1} with  $k_p=p$ (let $a=\frac{1}{b}$)
yields the desired statement. 
\end{proof}

When the stick-breaking  weights are iid,   Theorem \ref{theorem 0.1}
that we obtained for the stick-breaking process $P$ 
covers   very general  situation and the conditions 
\eqref{e.3.18}-\eqref{0.10} are minimal and are easy to verify. 
But 
when  the stick-breaking  weights are not iid  the situation
becomes much more sophisticated like in other statistical situations. 
We  shall consider  some well-known processes 
introduced earlier in Section \ref{s.pre}. For these processes 
the  explicit forms   of  the 
 joint finite dimensional distributions  of the stick-breaking weights,   although complicated, 
   are given.   
We can state  similar  results to that  in  
  Theorem \ref{theorem 0.1} in one theorem for all these processes. 
\begin{theorem}\label{theorem 2.17}
Let $P$ be  one of  the  Poisson-Dirichlet process $\PDP(a,b,H)$, the normalized inverse Gaussian process $\N-IGP(a,H)$, the normalized generalized gamma process $\NGGP(\sigma,a,H)$, 
   and the generalized Dirichlet process $\GDP(a,r,H)$.  
Then, we have the following results.
\begin{enumerate}
\item[(i)] 
As $a \rightarrow \infty$,
\begin{equation}
\left(D_a(A_1),D_a(A_2),\cdots,D_a(A_n)\right) \overset{d}{\rightarrow} (X_1,X_2,\cdots,X_n)\,, \label{e.2.34} 
\end{equation} 
where $(X_1,X_2,\cdots,X_n) \sim N (0, \Sigma)$ and $\Sigma
=(\si_{ij})_{1\le i,j\le n}$ is  given by 
\begin{align}
    \sigma_{ij}= \begin{cases}
    1 & \qquad \mbox{if $i=j$} \\
    -\sqrt{\frac{H(A_i)H(A_j)}{\left(1-H(A_i)\right)\left(1-H(A_j)\right)}}& \qquad \mbox{if $i \neq j$}\,.  \\
\end{cases} 
\end{align}
\item[(ii)] 
Let $(\XX, \cX)=(\RR^d, \cB(\RR^d))$ be the $d$-dimensional Euclidean space
with the Borel $\si$-algebra. Then 
\begin{equation}
Q_{H,a}  \stackrel{{\rm weakly}}{\rightarrow }  \quad {B_H^o}\qquad  {\rm in}\qquad D(\RR ^d)\ 
\end{equation}  
with respect to the Skorohod topology. 
\end{enumerate} 
\end{theorem} 


As long as the central limit theorem of $P$ is obtained, it is trivial to use  the delta-method 
to show similar theorem for the nonlinear functional 
 of this process. Using   Theorem 3.9.4 in \citet{van1996}, we  can state the following theorem. 
\begin{theorem}\label{delta method}
Let $P$ be  one of $\DP(a,H)$, $\PDP(a,b,H)$, $\N-IGP(a,H)$,
$\NGGP(\\
\sigma,a,H)$, $\GDP(a,r,H)$ or the Dirichlet process with general stick-breaking weights satisfying \eqref{e.3.18}-\eqref{0.10} of  
Theorem \ref{theorem 0.1}. Let $\DD$ be the metric space of all probability measures
on $(\XX,  \cX)$  with the total variation  distance.   
Let 
 $\phi :  \DD  \rightarrow \mathbb{R}^d$ be a continuous functional which is   Hadamard differentiable on $\DD$.
 Then, as $a \rightarrow \infty$, we have 
 \[
\frac{1}{\sqrt{\mathbb{E}[\sum_{i=1}^{\infty}w_i^2]}}\left(\phi\left(P(\cdot)\right)-\phi\left(H(\cdot)\right)\right)\overset{weakly}{\rightarrow}\phi'_{H(\cdot)}\left(B_H^o\right)\,. 
\]
 
\end{theorem}
One application of the above theorem  is the limiting distribution of the empirical quantile process of $P$.  
\begin{example}
Suppose $(\XX, \cX)= (\mathbb{R},\mathcal{B}(\mathbb{R}))  $ and 
suppose that   $P$ is   one of $\DP(a,H)$, $\PDP(a,b,H)$, $\N-IGP(a,H)$, $\NGGP(\sigma,a,H)$, $\GDP(a,r,H)$ or $P$ is the Dirichlet process with general stick-breaking weights satisfying \eqref{e.3.18}-\eqref{0.10} of  Theorem \ref{theorem 0.1}.
Let   $H$  be absolutely  continuous  
 with positive derivative $h$. 
By   Lemma 3.9.23 of \citet{van1996}, we have  
\begin{align}
\frac{1}{\sqrt{\mathbb{E}[\sum_{i=1}^{\infty}w_i^2]}}\left(P^{-1}(\cdot)-H^{-1}(\cdot)\right)\overset{weakly}{\rightarrow}-\frac{B^o(\cdot)}{h\left(H^{-1}(\cdot)\right)}=G(\cdot)\,,  
\end{align} 
where $H^{-1}(s)=\inf \{t: H(t)\geq s\}$. 
  The limiting process $G$ is a Gaussian process with zero-mean and 
with covariance function 
\begin{align*}
\Cov\left(G\left((0,s]\right),G\left((0,t]\right)\right)=\frac{s \wedge t-st}{h\left(H^{-1}\left((0,s]\right)\right)h\left(H^{-1}\left((0,t]\right)\right)} 
\end{align*}
for $s,t \in \mathbb{R}$,
\end{example}

\section{Concluding remarks} 
The  method of moments  used in this paper could be applied to the study of asymptotics for some Bayesian nonparametric posterior processes in the following situations:
(i) when  the  parameter $a$  is finite and  the   sample size is large;  (ii)  when  the  
parameter  $a$ is large  and  the  sample size is finite;  (iii)  
when  the  parameter  $a$ and the sample size  are both large.  Interesting examples are  the posterior processes of the normalize random measures with independent increments,  which include the normalized generalized inverse Gaussian process  studied in    \citet{james2009}.  We may also apply our method of moments to  study  the posterior distributions of the hNRMI  studied  by \citet{favaro2016},  which include  the generalized Dirichlet process and the normalized generalized gamma process.
\bibliographystyle{ba}
\bibliography{huzhang}

 \newpage

\section{Supplemental Appendix}
In this section, we present   the proofs for the propositions and theorems appeared  in this  paper.
\subsection{Proof of Proposition \ref{prop.3.1}}
\begin{proof} Using the binomial expansion and using the fact that $v_i\in [0, 1]$  we have 
\begin{align}
     \mathbb{E}\left[v_i^n(1-v_i)^m\right]&=\sum_{k=0}^{m}\binom{m}{k}(-1)^k\mathbb{E}[v_i^{n+k}]\nonumber\\
    &=\mathbb{E}[v_i^n]-m\mathbb{E}[v_i^{n+1}]+\cdots+(-1)^m\mathbb{E}[v_i^{m+n}]
    \nonumber\\
    &=\mathbb{E}[v_1^n]+{O}\left(\mathbb{E}[v_1^{n+1}]\right) =\mathbb{E}[v_1^n]+o\left(\mathbb{E}[v_1^{n }]\right)\,, \label{e.3.5} 
\end{align}
where the last  equality follows from the assumption $\displaystyle \lim_{a \rightarrow \infty} \frac{\mathbb{E}[v_1^{n+1}]}{\mathbb{E}[v_1^n]}=0$ for all $n \in \ZZ_+$. 
Since $v_i \in [0,1]$ and since we   assume that $v_0$ is not  identically 
zero,  we have 
$\mathbb{E}[(1-v_i)^m] \in [0,1)$ and
\begin{align}
     \sum_{j=0}^{\infty}\left(\mathbb{E}\left[(1-v_i)^m\right]\right)^j&=\frac{1}{1-\mathbb{E}\left[(1-v_i)^m\right]}\nonumber\\ 
    &=\frac{1}{1-\sum_{j=0}^{m}\binom{m}{j}(-1)^j\mathbb{E}[v_i^j]}\nonumber\\
    &=\frac{1}{m\mathbb{E}[v_1]+\sum_{j=2}^{m}\binom{m}{j}(-1)^j\mathbb{E}[v_i^j]}
    \nonumber\\
    &=\frac{1}{m\mathbb{E}[v_1]}+o\left(\frac{1}{m\mathbb{E}[v_1]}\right)\,, 
    \label{e.3.6} 
\end{align}   
where   the last  equality also follows from the assumption $\displaystyle \lim_{a \rightarrow \infty} \frac{\mathbb{E}[v_1^{n+1}]}{\mathbb{E}[v_1^n]}=0$ for all $n \in \ZZ_+$.  This proves \eqref{0.1}-\eqref{0.2}.

Now we use \eqref{0.1}-\eqref{0.2} to show \eqref{0.3}. 
Denote  $$\mathcal{I}=\mathbb{E}\left[\sum_{1\leq i_1<i_2<\cdots<i_k<\infty}w_{i_1}^{p_1}w_{i_2}^{p_2}\cdots w_{i_k}^{p_k}\right]\,.$$ By the construction of the stick-breaking sequence  $\{w_i\}_{i=1}^{\infty}$, we may rewrite $\mathcal{I}$ as
\begin{equation}
    \begin{aligned}
    \nonumber
    &\mathcal{I}=\\
    &\sum_{1\leq i_1<i_2<\cdots<i_k<\infty}\mathbb{E}\left[v_{i_1}^{p_1}\prod_{\ell_1=1}^{i_1-1}\left(1-v_{\ell_1}\right)^{p_1}\cdots v_{i_m}^{p_m}\prod_{\ell_m=1}^{i_m-1}(1-v_{\ell_m})^{p_m}\cdots v_{i_k}^{p_k}\prod_{\ell_k=1}^{i_k-1}(1-v_{\ell_k})^{p_k}\right].\\
    \end{aligned}
\end{equation}
Since $1\leq i_1<i_2<\cdots<i_k<\infty$, we can rearrange $\mathcal{I}$ by putting $v$'s with the same index together to obtain
\begin{equation}
	\begin{aligned}
	&\mathcal{I}=\sum_{1\leq i_1<i_2<\cdots<i_k<\infty}\mathbb{E}\Bigg[v_{i_1}^{p_1}(1-v_{i_1})^{p_{2:k}}\prod_{\ell_1=1}^{i_1-1}(1-v_{\ell_1})^{p_{1:k}}v_{i_2}^{p_2}(1-v_{i_2})^{p_{3:k}}\\
    &\prod_{\ell_2=i_1+1}^{i_2-1}(1-v_{\ell_2})^{p_{2:k}}\cdots v_{i_m}^{p_m}(1-v_{i_m})^{p_{m+1:k}}\prod_{\ell_m=i_{m-1}+1}^{i_m-1}(1-v_{\ell_m})^{p_{m:k}}\\
    &\cdots v_{i_{k-1}}^{p_{k-1}}(1-v_{i_{k-1}})^{p_k}\prod_{\ell_{k-1}=i_{k-2}+1}^{i_{k-1}-1}(1-v_{\ell_{k-1}})^{p_k-1:k}v_{i_k}^{p_k}\prod_{\ell_k=i_{k-1}+1}^{i_k-1}(1-v_{\ell_k})^{p_k}\Bigg].\label{0.5}
	\end{aligned}
\end{equation}
From the independence of  $\{v_1, v_2, \cdots\}$   it follows 
\begin{equation}
	\begin{aligned}
	\nonumber
	 \mathcal{I} = \sum_{i_1=1}^{\infty}&\mathbb{E}\left[v_{i_1}^{p_1}\left(1-v_{i_1}\right)^{p_{2:k}}\right]\prod_{\ell_1=1}^{i_1-1}\mathbb{E}\left[\left(1-v_{\ell_1}\right)^{p_{1:k}}\right]\sum_{i_2=i_1+1}^{\infty}\mathbb{E}\left[v_{i_2}^{p_2}\left(1-v_{i_2}\right)^{p_{3:k}}\right]\\
    &\prod_{\ell_2=i_1+1}^{i_2-1}\mathbb{E}\left[\left(1-v_{\ell_2}\right)^{p_{2:k}}\right]\cdots \sum_{i_m=i_{m-1}+1}^{\infty}\mathbb{E}\left[v_{i_m}^{p_m}\left(1-v_{i_m}\right)^{p_{m+1:k}}\right]\\
    &\prod_{\ell_m=i_{m-1}+1}^{i_m-1}\mathbb{E}\left[\left(1-v_{\ell_m}\right)^{p_{m:k}}\right]\cdots \sum_{i_{k-1}=i_{k-2}+1}^{\infty}\mathbb{E}\left[v_{i_{k-1}}^{p_{k-1}}\left(1-v_{i_{k-1}}\right)^{p_k}\right]\\
    &\prod_{\ell_{k-1}=i_{k-2}+1}^{i_{k-1}-1}\mathbb{E}\left[\left(1-v_{\ell_{k-1}}\right)^{p_{k-1:k}}\right]\sum_{i_k=i_{k-1}+1}^{\infty}\mathbb{E}\left[v_{i_k}^{p_k}\right]\prod_{\ell_k=i_{k-1}+1}^{i_k-1}\mathbb{E}\left[\left(1-v_{\ell_k}\right)^{p_k}\right].
	\end{aligned}
\end{equation}
Denoting  the general factor in the above expression by
\[
\mathcal{S}_m=\sum_{i_m=i_{m-1}+1}^{\infty}\mathbb{E}\left[v_{i_m}^{p_m}\left(1-v_{i_m}\right)^{p_{m+1:k}}\right]\prod_{\ell_m=i_{m-1}+1}^{i_m-1}\mathbb{E}\left[\left(1-v_{\ell_m}\right)^{p_{m:k}}\right],
\]
for $m\in \{1, 2, \cdots, k\}$,  we can write 
\begin{equation}
\mathcal{I}=\mathcal{S}_1\mathcal{S}_2\cdots\mathcal{S}_k\,.
\label{e.3.8} 
\end{equation} 
From  the fact that  $\{v_1, v_2, \cdots\}$ are identical and  by \eqref{0.1}-\eqref{0.2} we have  for $m=1, \cdots, k$, 
\begin{equation}
	\begin{aligned}
	\nonumber
	&\mathcal{S}_m=\mathbb{E}\left[v_1^{p_m}\left(1-v_1\right)^{p_{m+1:k}}\right]\sum_{i_m=i_{m-1}+1}^{\infty} \Big(\mathbb{E}\left[(1-v_1)^{p_{m:k}}\right]\Big)^{i_m-i_{m-1}-1}\\
	&=\bigg(\mathbb{E}\left[v_1^{p_m}\right]+o\left(\mathbb{E}\left[v_1^{p_m}\right]\right)\bigg)\left(\frac{1}{p_{m:k}\mathbb{E}[v_1]}+o\left(\frac{1}{\mathbb{E}[v_1]}\right)\right)\\
	&=\frac{\mathbb{E}[v_1^{p_m}]}{p_{m:k}\mathbb{E}[v_1]}+o\left(\frac{\mathbb{E}[v_1^{p_m}]}{p_{m:k}\mathbb{E}[v_1]}\right).
	\end{aligned}
\end{equation}
Substituting this estimate into \eqref{e.3.8}, we see 
\begin{equation}
	\begin{aligned}
	\nonumber
	&\mathcal{I}=\prod_{m=1}^k\left(\frac{\mathbb{E}[v_1^{p_m}]}{p_{m:k}\mathbb{E}[v_1]}+o\left(\frac{\mathbb{E}[v_1^{p_m}]}{p_{m:k}\mathbb{E}[v_1]}\right)\right)\\
	&=\frac{\mathbb{E}[v_1^{p_1}]\cdots\mathbb{E}[v_1^{p_m}]}{p_{1:k}p_{2:k} \cdots p_{k:k}(\mathbb{E}[v_1])^k}+o\left(\frac{\mathbb{E}[v_1^{p_1}]\cdots\mathbb{E}[v_1^{p_m}]}{(\mathbb{E}[v_1])^k}\right).
	\end{aligned}
\end{equation}
This prove \eqref{0.3}. 
If we take  $p_j=2$ for all $j\in\{ 1, \cdots, k\}$,   then 
  $p_{1:k}=2k$.  The identity \eqref{0.4} is hence  a straightforward consequence of \eqref{0.3}.
\end{proof}
\subsection{Proof of Proposition \ref{prop.3.2}}
Since  $v_1, v_2, \cdots$  are no longer  identically distributed, the results established in   the 
 previous proof  cannot be applied and  we  need   
 some new computations. 
We shall still use the general method of moments. To this end,
  we need   first to recall   
some results about the  hypergeometric functions   
  and we refer to  \citet{aomoto2011} for further reading.

\begin{definition}
The hypergeometric function $\ {}_2F_1(a,b, c;x)$ (of parameters $a, b, c\in \mathbb{C}$, 
 the complex plane)
is defined by the series
\[
\ {}_2F_1(a,b, c;x)=\sum_{n=0}^{\infty}\frac{(a)_n(b)_n}{(c)_nn!}x^n
\]
for $|x| < 1$, where $(q)_n$ is the Pochhammer symbol  defined by
\begin{equation}
\nonumber
    (q)_n = \left\{ \begin{array}{lcl}
    0 & \quad \mbox{for} & n=0 \\
    q(q+1)...(q+n-1) & \quad \mbox{for} & n>0 \,. \\
\end{array}\right.
\end{equation} 
This function is defined for $|x|<1$ and may  be extended to
$x=1$ and/or $x=-1$  by continuation.  
\end{definition}
We need the following result obtained by Gauss:  when   ${\rm Re}(c-a-b)>0$ (real part of 
$c-a-b$), the hypergeometric function can be extended to 
 $x=1$ and its value at this point is given by 
\begin{align} 
\ {}_2F_1(a,b, c;1)=\frac{\Gamma(c)\Gamma(c-a-b)}{\Gamma(c-a)\Gamma(c-b)}.\label{1.2.1}
\end{align}
We introduce a variant of   the hypergeometric function that will be needed  in the following calculations.
\begin{definition}
For any $b<2$, $n\in \mathbb{N}^+$, $a>0$, $m>0$, $c>0$, define the \textit{increasing coefficient hypergeometric function} ${}_2Q_1\left((a,b),c,m,n;x\right)$ by the series
$${}_2Q_1\left((a,b),c,m,n;x\right)=\sum_{k=0}^{\infty}\prod_{\ell=1}^{n-1}\left(a+b(k+\ell)\right)\frac{\left(\frac{a}{b}+1\right)_k(c)_k}{\left(\frac{a+m}{b}+1\right)_kk!}x^k$$ for $|x|<1$ and we may extend the definition to   $x=1$ and/or $x=-1$ by continuation.  
In the above product we use the convention 
that  $\prod_{\ell =1}^{0}c_\ell =1$.
\end{definition}
The next proposition describes  a Gauss type    result for the increasing coefficient hypergeometric function.
\begin{prop}\label{lemma1.2.1}
Let $b<2$, $n\in \mathbb{N}^+$, $a>0$, $m>0$, $c>0$. Then,  the increasing coefficient hypergeometric function can be extended to
 $x=1$ and its value at this point is given by 
\[
{}_2Q_1\left((a,b),c,m,n;1\right)=\prod_{\ell=1}^{n-1}(a+b\ell)\frac{a+m}{m-nb}\,.
\]
\end{prop}
\begin{proof} By  \eqref{1.2.1}  we have 
\begin{align}
\nonumber
    &{}_2Q_1\left((a,b),c,m,n;1\right)=\prod_{\ell=1}^{n-1}(a+b\ell)\sum_{k=0}^{\infty}\prod_{i=1}^{k}\frac{a+(n+i-1)b}{a+ib+m}\nonumber\\
    &=\prod_{\ell=1}^{n-1}(a+b\ell)\ \ {}_2F_1\left(\frac{a}{b}+n,1, \frac{a+m}{b}+1;1\right)=\prod_{\ell=1}^{n-1}(a+b\ell)\frac{a+m}{m-nb},\label{1.2.2}
    \end{align} 
    proving the proposition.
\end{proof}
Now we are in the position of proving    Proposition 3.2.

\noindent {\it Proof of Proposition 3.2}\quad 
Denote $$\mathcal{S}_m=\sum_{i_m=i_{m-1}+1}^{\infty}\mathbb{E}\left[v_{i_m}^{p_m}(1-v_{i_m})^{p_{m+1:k}}\right]\prod_{\ell_m=i_{m-1}+1}^{i_m-1}\mathbb{E}\left[(1-v_{\ell_m})^{p_{m:k}}\right],$$
for $m\in \{1, 2, \cdots, k\}$.
Then we can write 
\[
\mathcal{I}:= \mathbb{E}\left[\sum_{1\leq i_1<i_2<\cdots<i_k<\infty}w_{i_1}^{p_1}w_{i_2}^{p_2}\cdots w_{i_k}^{p_k}\right]=
\mathcal{S}_1\mathcal{S}_2\cdots\mathcal{S}_k\,.
\]
We shall compute  $\mathcal{I}$ by  computing 
 $\mathcal{S}_m\mathcal{S}_{m+1}\cdots\mathcal{S}_k $  recursively on $m=k, k-1, \cdots,  2, 1$. 
First, by using \eqref{1.2.1} we have 
    \begin{align}
     \mathcal{S}_k&=\sum_{i_k=i_{k-1}+1}^{\infty}\frac{(1-b)_{p_k}}{(1+a+b(i_k-1))_{p_k}}\prod_{\ell_k=i_{k-1}+1}^{i_k-1}\frac{(a+b\ell_k)_{p_k}}{(1+a+b(\ell_k-1))_{p_k}}\nonumber\\
    &=\frac{(1-b)_{p_k}}{(1+a+bi_{k-1})_{p_k}}\ {}_2F_1\left(\frac{a}{b}+i_{k-1}+1,1, \frac{a+p_k}{b}+i_{k-1}+1;1\right)\nonumber\\
    &=\frac{(1-b)_{p_k}}{(p_k-b)(1+a+bi_{k-1})_{p_k-1}}.\label{1.2.5}
    \end{align}
Now we want to compute $\mathcal{S}_m\mathcal{S}_{m+1}\cdots\mathcal{S}_k $
assuming that  we have already computed 
$\mathcal{S}_{m+1}\cdots\mathcal{S}_k $.  To make thing clear we will explain 
how to compute $\mathcal{S}_1 \cdots\mathcal{S}_k $
from the expression of 
$\mathcal{S}_{2}\cdots\mathcal{S}_k $.  General case is similar. 
We assume 
\begin{equation}
    \begin{aligned}
    \nonumber
    &\mathcal{S}_2\cdots\mathcal{S}_k
    =\frac{(a+b(i_1+1))\cdots\left(a+b(i_1+k-2)\right)}{(1+a+bi_1)_{p_{2:k}-1}}\prod_{i=2}^{k}\frac{(1-b)_{p_i}}{{p_{i:k}-(k-i+1)b}}.
    \end{aligned}
\end{equation}
Then, by Proposition \ref{lemma1.2.1}
\begin{equation}
    \begin{aligned}
    \nonumber
    \mathcal{I}&=\mathcal{S}_1\mathcal{S}_2\cdots\mathcal{S}_k=(1-b)_{p_1}\prod_{i=2}^{k}\frac{(1-b)_{p_i}}{p_{i:k}-(k-i+1)b}\\
   &\qquad\qquad \sum_{i=1}^{\infty}\frac{(a+bi_1)\cdots(a+b(i_1+k-2))}{(1+a+b(i_1-1))_{p_{1:k}}}\prod_{\ell_1=1}^{i_1-1}\frac{(a+b\ell_1)_{p_{1:k}}}{(1+a+b(\ell_1-1))_{p_{1:k}}}\\
   &=\frac{(1-b)_{p_1}}{(a+1)_{p_{1:k}}}\prod_{i=2}^{k}\frac{(1-b)_{p_i}}{p_{i:k}-(k-i+1)b}\ {}_2Q_1\left((a,b),1,p_{1:k},k;1\right)\\
    &=\frac{(1-b)_{p_1}}{(a+1)_{p_{1:k}}}\prod_{i=2}^{k}\frac{(1-b)_{p_i}}{p_{i:k}-(k-i+1)b}\prod_{\ell=1}^{k-1}(a+b\ell)\frac{a+p_{1:k}}{p_{1:k}-kb}\\
   &=\frac{(a+b)\cdots(a+b(k-1))}{(1+a)_{p_{1:k}-1}}\prod_{i=1}^{k}\frac{(1-b)_{p_i}}{p_{i:k}-(k-i+1)b}\\
   &=\frac{1}{(a+kb)(a+1)_{(p_{1:k}-1)}}\prod_{i=1}^k\frac{(1-b)_{p_i}(a+bi)}{p_{i:k}-(k-i+1)b}.
    \end{aligned}
\end{equation}
This proves \eqref{1.2.3}. 

When  $p_j=2$ for all $j\in\{ 1, \cdots, n\}$, we have easily 
\begin{equation}
\begin{aligned}
\nonumber
    &\mathcal{I}=\mathbb{E}\left[\sum_{1\leq i_1<i_2<\cdots<i_n<\infty}w_{i_1}^2w_{i_2}^2\cdots w_{i_n}^2\right]=\frac{(1-b)^n(a+b)\cdots(a+b(n-1))}{n!(a+1)\cdots(a+2n-1)} 
\end{aligned}
\end{equation}
proving \eqref{1.2.4}. 
\fin
\subsection{Proof of Proposition \ref{prop.3.3}}
\begin{proof}
By the stick-breaking representation of $\hbox{\rm N-IGP}(a,H)$ and   the 
formula  (3.471.9) 
 in \citet{gradshteyn2014}, we find  
 that the joint distribution of $\{v_i\}_{i=1}^n$ can be written as 
\begin{align}
f(v_1,\cdots,v_n)=\frac{e^aa^{n+1}}{(2\pi)^{\frac{n+1}{2}}}\prod_{i=1}^nv_i^{-\frac{1}{2}}(1-v_i)^{-\frac{n+3-i}{2}}\int_0^{\infty}t^{-\frac{n+3}{2}}e^{-\frac{t}{2}-\frac{a^2}{2t\prod_{i=1}^n(1-v_i)}}dt.
\end{align}
Thus, by using Fubini's theorem, we have 
\begin{align}
 \mathcal{I}=&\mathbb{E}\left[\sum_{n=1}^{\infty}w_n^p\right]=\mathbb{E}\left[\sum_{n=1}^{\infty}v_n^p\prod_{i=1}^{n-1}(1-v_i)^p\right]\nonumber\\
 =&\sum_{n=1}^{\infty}\int_0^{\infty}\int_0^1\cdots\int_0^1\frac{e^aa^{n+1}}{(2\pi)^{\frac{n+1}{2}}}\left(\prod_{i=1}^{n-1}v_i^{-\frac{1}{2}}(1-v_i)^{p-\frac{n+3-i}{2}}\right)v_n^{p-\frac{1}{2}}(1-v_n)^{-\frac{3}{2}}\nonumber\\
&\qquad \qquad t^{-\frac{n+3}{2}}e^{-\frac{t}{2}-\frac{a^2}{2t\prod_{i=1}^n(1-v_i)}}dv_1\cdots dv_ndt\,. \label{5.3}
\end{align}
To evaluate  the above  multiple integral, we shall use the   formula (3.471.2)
 in  \citet{gradshteyn2014}:
\begin{align}
\int_0^1 (1-v)^{\eta-1}v^{\mu-1}e^{-\frac{\beta}{1-v}}dv=\Gamma(\mu)\beta^{\frac{\eta-1}{2}}e^{-\frac{\beta}{2}}W_{\frac{1-2\mu-\eta}{2},\frac{\eta}{2}}(\beta),\label{5.4}
\end{align}
where $W$ is the Whittaker function. For large $\beta$, by the formula  (9.227) in \citet{gradshteyn2014}, we have
\begin{align}
\int_0^1 (1-v)^{\eta-1}v^{\mu-1}e^{-\frac{\beta}{1-v}}dv=\Gamma(\mu)\beta^{-\mu}e^{-\beta}\left(1+o\left(\frac{1}{\beta}\right)\right)\,, \label{5.5}
\end{align}
as $\beta\rightarrow \infty$.  In particular, when $\mu=\frac{1}{2}$,  we have 
\begin{align}
\int_0^1 (1-v)^{\eta-1}v^{-\frac{1}{2}}e^{-\frac{\beta}{1-v}}dv=\Gamma\left(\frac{1}{2}\right)\beta^{-\frac{1}{2}}e^{-\beta}\left(1+o\left(\frac{1}{\beta}\right)\right).\label{5.6}
\end{align}
Denote  $\beta_{i}=\frac{a^2}{2t(1-v_n)\prod_{\ell=1}^i(1-v_\ell)}$ for $i \in \{1,
\cdots, n-1\}$. We  
 rewrite \eqref{5.3} as
\begin{align}
 \mathcal{I}=&\sum_{n=1}^{\infty}\int_0^{\infty}\int_0^1\cdots\int_0^1\frac{e^aa^{n+1}}{(2\pi)^{\frac{n+1}{2}}}t^{-\frac{n+3}{2}}e^{-\frac{t}{2}}\left(\prod_{i=1}^{n-2}v_i^{-\frac{1}{2}}(1-v_i)^{p-\frac{n+3-i}{2}}\right)\nonumber\\
&\quad v_n^{p-\frac{1}{2}}(1-v_n)^{-\frac{3}{2}} \int_0^1v_{n-1}^{-\frac{1}{2}}(1-v_{n-1})^{p-\frac{4}{2}}e^{-\frac{\beta_{n-2}}{(1-v_{n-1})}}dv_{n-1}dv_1\cdots dv_{n-2} dv_n dt.\nonumber
\end{align}
Integrating with respect to $v_{n-1}$ by applying \eqref{5.6}  yields 
\begin{align}
&\mathcal{I}=\sum_{n=1}^{\infty}\int_0^{\infty}\int_0^1\cdots\int_0^1\frac{e^aa^{n+1}}{(2\pi)^{\frac{n+1}{2}}}\Gamma\left(\frac{1}{2}\right)\left(\frac{a^2}{2t(1-v_n)}\right)^{-\frac{1}{2}}\nonumber\\
& \qquad  \left(1+o\left(\frac{1}{a}\right)\right)t^{-\frac{n+3}{2}}e^{-\frac{t}{2}}(1-v_n)^{-\frac{3}{2}} v_n^{p-\frac{1}{2}}\left(\prod_{i=1}^{n-2}v_i^{-\frac{1}{2}}(1-v_i)^{p-\frac{n+2-i}{2}}\right)\nonumber\\
&\qquad 
\int_0^1v_{n-2}^{-\frac{1}{2}}(1-v_{n-2})^{p-\frac{4}{2}}e^{-\frac{\beta_{n-3}}{(1-v_{n-2})}} dv_1\cdots dv_{n-2} dv_ndt.\nonumber
\end{align}
We repeatedly  apply the above procedure    to integrate $v_{n-2}$, $v_{n-3}$, $\cdots$, $v_1$, each time   using    \eqref{5.6}.  
After these computations we   obtain 
\begin{align}
 \mathcal{I}&= \sum_{n=1}^{\infty}\int_0^{\infty}\int_0^1\frac{e^aa^{n+1}}{(2\pi)^{\frac{n+1}{2}}}\left(\Gamma\left(\frac{1}{2}\right)\left(\frac{a^2}{2t(1-v_n)}\right)^{-\frac{1}{2}}\right)^{n-1}\left(1+o\left(\frac{1}{a}\right)\right) \nonumber\\
&\qquad\qquad\qquad \times  t^{-\frac{n+3}{2}}e^{-\frac{t}{2}}v_n^{p-\frac{1}{2}}(1-v_n)^{-\frac{3}{2}}dv_{n}dt\nonumber\\
 &= \sum_{n=1}^{\infty}\int_0^{\infty}\int_0^1\frac{e^aa^2}{2\pi}v ^{p-\frac{1}{2}}(1-v )^{\frac{n-4}{2}}e^{-\frac{a^2}{2t(1-v )}}t^{-2}e^{-\frac{t}{2}}\left(1+o\left(\frac{1}{a}\right)\right)dv dt.\label{5.61}
\end{align}
Again by  using 
Fubini's theorem and  by the fact that $v  \in (0,1)$,   we   take the sum to   obtain  
\begin{align}
 \mathcal{I}&=\int_0^{\infty}\int_0^1\frac{e^aa^2}{2\pi}v ^{p-\frac{1}{2}}\left(\sum_{n=1}^{\infty}(1-v )^{\frac{n-4}{2}}\right)e^{-\frac{a^2}{2t(1-v )}}t^{-2}e^{-\frac{t}{2}}\left(1+o\left(\frac{1}{a}\right)\right)dv dt\nonumber\\
&=\int_0^{\infty}\int_0^1\frac{e^aa^2}{2\pi}v ^{p-\frac{3}{2}}\left(1+(1-v )^{\frac{1}{2}}\right)e^{-\frac{a^2}{2t(1-v  )}}t^{-2}e^{-\frac{t}{2}}\left(1+o\left(\frac{1}{a}\right)\right)dv dt.\label{5.62}
\end{align}
Then,  the result   \eqref{5.01} is obtained by first applying \eqref{5.5} when we  integrate  the integral with respect to 
  $v $, and then by applying   the formula (3.471.9) 
 in \citet{gradshteyn2014} when we integrate  $t$. Here,  we also use the approximation that for fixed $\nu$ and for 
 large $a$, $K_{\nu}(a)=\sqrt{\frac{\pi}{2}}a^{-1/2}e^{-a}\left(1+o\left(\frac{1}{a}\right)\right)$.

\noindent When $p=2$, we want to show that the leading coefficient in \eqref{5.01} is $1$. This needs some more delicate computations.   First, we have 
\begin{align}
&\mathcal{I}=\sum_{n=1}^{\infty}\int_0^{\infty}\int_0^1\cdots\int_0^1\frac{e^aa^{n+1}}{(2\pi)^{\frac{n+1}{2}}}\left(\prod_{i=1}^{n-1}v_i^{-\frac{1}{2}}(1-v_i)^{-\frac{n-1-i}{2}}\right)v_n^{\frac{3}{2}}(1-v_n)^{-\frac{3}{2}} \nonumber\\
&\qquad\qquad\qquad \times t^{-\frac{n+3}{2}}e^{-\frac{t}{2}-\frac{a^2}{2t\prod_{i=1}^n(1-v_i)}}dv_1\cdots dv_ndt.\nonumber
\end{align}
Notice the fact that the power of $v_{n-4}$  in the above integrand  is $-\frac{3}{2}$ and  to compute the integral with respect to $v_{n-4}$ we can use the  following nice integral identity: 
\begin{align}
\int_0^1 (1-v)^{-\frac{3}{2}}v^{-\frac{1}{2}}e^{-\frac{\beta}{1-v}}dv=\Gamma(\frac{1}{2})\beta^{-\frac{1}{2}}e^{-\beta}.
\end{align}
After this  integration with respect to $v_{n-4}$, we obtain an expression for $v_{n-5}$ which  also has this form and we then integrate $v_{n-5}$ and so on.  This procedure can continue until integrating $v_1$.  Hence, we 
 compute  the integrals for  $v_{n-4}$, and then for 
 $v_{n-5}$,  $\cdots$ and then for  $v_1$ recursively to obtain 
\begin{align}
 \mathcal{I}=&\sum_{n=1}^{\infty}\int_0^{\infty}\int_0^1\int_0^1 \int_0^1\int_0^1\frac{e^aa^5}{(2\pi)^{\frac{5}{2}}} (v_{n-3}v_{n-2}v_{n-1})^{-\frac{1}{2}}v_n^{\frac{3}{2}}\nonumber\\
&\qquad\qquad  \times (1-v_{n-3})^{\frac{n-6}{2}}(1-v_{n-2})^{\frac{n-5}{2}} (1-v_{n-1})^{\frac{n-4}{2}}(1-v_n)^{\frac{n-7}{2}}\nonumber\\
&\qquad\qquad \times e^{-\frac{a^2}{2t(1-v_{n-3})(1-v_{n-2})(1-v_{n-1})(1-v_n)}} t^{-\frac{7}{2}}e^{-\frac{t}{2}}dv_{n-3}dv_{n-2}dv_{n-1}dv_ndt.
\end{align}
By Fubini's theorem and by the fact that $v_i \in (0,1)$, we can take the sum first [There is no need to sum up the index $n$ in $v_{n}, v_{n-1}, v_{n-2}, v_{n-3}$
since we can call them by other notations. But 
for consistency we still keep  these notations.] 
\begin{align}
 \mathcal{I}=&\int_0^{\infty}\int_0^1\int_0^1 \int_0^1\int_0^1\frac{e^aa^5}{(2\pi)^{\frac{5}{2}}} (v_{n-3}v_{n-2}v_{n-1})^{-\frac{1}{2}}v_n^{\frac{3}{2}}(1-v_{n-3})^{-\frac{5}{2}}(1-v_{n-2})^{-\frac{4}{2}}\nonumber\\
&\qquad \times(1-v_{n-1})^{-\frac{3}{2}}(1-v_n)^{-\frac{6}{2}}\frac{1+\sqrt{(1-v_{n-3})(1-v_{n-2})(1-v_{n-1})(1-v_n)}}{1-(1-v_{n-3})(1-v_{n-2})(1-v_{n-1})(1-v_n)}\nonumber\\
&\qquad \times e^{-\frac{a^2}{2t(1-v_{n-3})(1-v_{n-2})(1-v_{n-1})(1-v_n)}} t^{-\frac{7}{2}}e^{-\frac{t}{2}}dv_{n-3}dv_{n-2}dv_{n-1}dv_ndt.
\end{align}
Now  we  integrate  $t$ by  using the formula (3.471.9) 
in \citet{gradshteyn2014} to obtain 
\begin{align*}
 \mathcal{I}=&\int_0^1\int_0^1 \int_0^1\int_0^1\frac{2e^aa^{\frac{5}{2}}}{(2\pi)^{\frac{5}{2}}} (v_{n-3}v_{n-2}v_{n-1})^{-\frac{1}{2}}v_n^{\frac{3}{2}}(1-v_{n-3})^{-\frac{5}{4}}(1-v_{n-2})^{-\frac{3}{4}}\nonumber\\
&\qquad (1-v_{n-1})^{-\frac{1}{4}}(1-v_n)^{-\frac{7}{4}}\frac{1+\sqrt{(1-v_{n-3})(1-v_{n-2})(1-v_{n-1})(1-v_n)}}{1-(1-v_{n-3})(1-v_{n-2})(1-v_{n-1})(1-v_n)}\nonumber\\
&\qquad K_{-\frac{5}{2}}\left(\frac{a}{\sqrt{(1-v_{n-3})(1-v_{n-2})(1-v_{n-1})(1-v_n)}}\right)dv_{n-3}dv_{n-2}dv_{n-1}dv_n\,. 
\end{align*}
Approximating  the above modified Bessel function 
$K_{-\frac{5}{2}}$ of the third type  by the formula (8.451.6)  in \citet{gradshteyn2014} we have 
\begin{align*}
 \mathcal{I}=&\int_0^1\int_0^1 \int_0^1\int_0^1\frac{e^aa^2}{(2\pi)^2} (v_{n-3}v_{n-2}v_{n-1})^{-\frac{1}{2}}v_n^{\frac{3}{2}}(1-v_{n-3})^{-\frac{2}{2}}(1-v_{n-2})^{-\frac{1}{2}}\nonumber\\
&\qquad \times (1-v_{n-1})^{-\frac{0}{2}}(1-v_n)^{-\frac{3}{2}}\frac{1+\sqrt{(1-v_{n-3})(1-v_{n-2})(1-v_{n-1})(1-v_n)}}{1-(1-v_{n-3})(1-v_{n-2})(1-v_{n-1})(1-v_n)}\nonumber\\
&\qquad \times e^{-\frac{a^2}{\sqrt{(1-v_{n-3})(1-v_{n-2})(1-v_{n-1})(1-v_n)}}}\left(1+o\left(\frac{1}{a}\right)\right)dv_{n-3}dv_{n-2}dv_{n-1}dv_n.
\end{align*}
To evaluate the above integral we make the following variable substitutions. 
\begin{align*}
\begin{cases}
v_n=1-\frac{a^2}{(y_0+a)^2}\,;\\
v_{n-1}=1-\frac{(y_0+a)^2}{(y_0+y_1+a)^2}\,;\\
v_{n-2}=1-\frac{(y_0+y_1+a)^2}{(y_0+y_1+y_2+a)^2}\,;\\
v_{n-3}=1-\frac{(y_0+y_1+y_2+a)^2}{(y_0+y_1+y_2+y_3+a)^2}\,.
\end{cases} 
\end{align*}
The integral $\mathcal{I}$ can then be written as 
\begin{align}
 \mathcal{I}=&\int_0^1\int_0^1 \int_0^1\int_0^1\frac{e^aa^2}{(2\pi)^2} \left(1-\frac{a^2}{(y_0+a)^2}\right)^{\frac{3}{2}} \left(1-\frac{(y_0+a)^2}{(y_0+y_1+a)^2}\right)^{-\frac{1}{2}}\nonumber\\
&\qquad \times \left(1-\frac{(y_0+y_1+a)^2}{(y_0+y_1+y_2+a)^2}\right)^{-\frac{1}{2}}
\left(1-\frac{(y_0+y_1+y_2+a)^2}{(y_0+y_1+y_2+y_3+a)^2}\right)^{-\frac{1}{2}}\nonumber\\
&\qquad \times \left(\frac{(y_0+y_1+y_2+a)^2}{(y_0+y_1+y_2+y_3+a)^2}\right)^{-\frac{2}{2}}\left(\frac{(y_0+y_1+a)^2}{(y_0+y_1+y_2+a)^2}\right)^{-\frac{1}{2}}\nonumber\\
&\qquad \times \left(\frac{(y_0+a)^2}{(y_0+y_1+a)^2}\right)^{-\frac{0}{2}}\left(\frac{a^2}{(y_0+a)^2}\right)^{-\frac{3}{2}}\frac{1+\frac{a}{(y_0+y_1+y_2+y_3+a)}}{1-\frac{a^2}{(y_0+y_1+y_2+y_3+a)^2}}\nonumber\\
&\qquad \times \frac{2^4a^2}{(y_0+a)(y_0+y_1+a)(y_0+y_1+y_2+a)(y_0+y_1+y_2+y_3+a)^3}\nonumber\\
&\qquad e^{-y_0-y_1-y_2-y_3-a}\left(1+o\left(\frac{1}{a}\right)\right)dy_0dy_1dy_2dy_3.\label{5.7}
\end{align}
When  $a$ is large, we have the following asymptotics: 
\begin{eqnarray*}
&&1-\frac{a^2}{(y_0+a)^2}=\frac{2ay_0+y_0^2}{a^2+2ay_0+y_0^2}=\frac{2y_0}{a}+o\left(\frac{1}{a}\right)\,;\\
&&1-\frac{(y_0+a)^2}{(y_0+y_1+a)^2}=\frac{2y_1}{a}+o\left(\frac{1}{a}\right)\,;\\
&&1-\frac{(y_0+y_1+a)^2}{(y_0+y_1+y_2+a)^2}=\frac{2y_2}{a}+o\left(\frac{1}{a}\right)\,;\\
&&1-\frac{(y_0+y_1+y_2+a)^2}{(y_0+y_1+y_2+y_3+a)^2}=\frac{2y_3}{a}+o\left(\frac{1}{a}\right)\,;\\
&&\frac{a^2}{(y_0+a)^2}=\frac{(y_0+a)^2}{(y_0+y_1+a)^2}=\frac{(y_0+y_1+a)^2}{(y_0+y_1+y_2+a)^2}\\
&&\qquad\qquad \quad  =\frac{(y_0+y_1+y_2+a)^2}{(y_0+y_1+y_2+y_3+a)^2}=1+o\left(\frac{1}{a}\right)\,;\\
&&1+\frac{a}{(y_0+y_1+y_2+y_3+a)}=2+o\left(\frac{1}{a}\right)\,;\\
&&1-\frac{a^2}{(y_0+y_1+y_2+y_3+a)^2}=\frac{2(y_0+y_1+y_2+y_3)}{a}+o\left(\frac{1}{a}\right)\,;\\
&&\frac{2^4a^2}{(y_0+a)(y_0+y_1+a)(y_0+y_1+y_2+a)(y_0+y_1+y_2+y_3+a)^3}\\
&&\qquad\qquad\qquad =\frac{2^4}{a^4}+o\left(\frac{1}{a^4}\right)\,. 
\end{eqnarray*}
Substituting the above asymptotics into    \eqref{5.7},
we see that when $a$ is large the integral $\mathcal{I}$ is 
\begin{align}
 \mathcal{I}=&\int_0^1\int_0^1 \int_0^1\int_0^1\frac{e^aa^2}{(2\pi)^2} \left(\frac{2y_0}{a}\right)^{\frac{3}{2}}\left(\frac{2y_1}{a}\right)^{-\frac{1}{2}}\left(\frac{2y_2}{a}\right)^{-\frac{1}{2}}
\left(\frac{2y_3}{a}\right)^{-\frac{1}{2}} \nonumber\\
&\qquad   \times\frac{a}{(y_0+y_1+y_2+y_3)} \frac{2^4}{a^4}e^{-y_0-y_1-y_2-y_3-a}\left(1+o\left(\frac{1}{a}\right)\right)dy_0dy_1dy_2dy_3\nonumber\\
  =&\int_0^1\int_0^1 \int_0^1\int_0^1\frac{4}{a\pi^2}\frac{y_0^{\frac{3}{2}}y_1^{-\frac{1}{2}}y_2^{-\frac{1}{2}}
y_3^{-\frac{1}{2}}}{y_0+y_1+y_2+y_3}\nonumber\\
&\qquad  \times e^{-y_0-y_1-y_2-y_3}\left(1+o\left(\frac{1}{a}\right)\right)dy_0dy_1dy_2dy_3\,. \label{5.8}
\end{align}
The above integral can be further evaluated by  making  the following  variable substitutions: 
\begin{align*}
\begin{cases}
y_0=t_0\,;\\
y_0+y_1=t_1\,;\\
y_0+y_1+y_2=t_2\,;\\
y_0+y_1+y_2+y_3=t_3\,.
\end{cases} 
\end{align*}
With these substitutions,  \eqref{5.8} can be written as 
\begin{align}
 \mathcal{I}
=&\int_0^{\infty}\int_0^{t_3} \int_0^{t_2}\int_0^{t_1}\frac{4}{a\pi^2}\frac{t_0^{\frac{3}{2}}(t_1-t_0)^{-\frac{1}{2}}(t_2-t_1)^{-\frac{1}{2}}
(t_3-t_2)^{-\frac{1}{2}}}{t_3}\nonumber\\
&\qquad\qquad \times e^{-t_3}\left(1+o\left(\frac{1}{a}\right)\right)dt_0dt_1dt_2dt_3.\nonumber
\end{align}
Now we    integrate    $t_0$ by letting $t_0=ut_1$,
 \begin{align}
 \mathcal{I}
&=\int_0^{\infty}\int_0^{t_3} \int_0^{t_2}\int_0^1\frac{4}{a\pi^2}\frac{t_1(ut_1)^{\frac{3}{2}}(t_1-ut_1)^{-\frac{1}{2}}(t_2-t_1)^{-\frac{1}{2}}
(t_3-t_2)^{-\frac{1}{2}}}{t_3}\nonumber\\
& \qquad\qquad \times e^{-t_3}\left(1+o\left(\frac{1}{a}\right)\right)dudt_1dt_2dt_3\nonumber\\
&=\int_0^{\infty}\int_0^{t_3} \int_0^{t_2}\frac{4}{a\pi^2}\frac{
{ \Gamma(\frac{5}{2}) \Gamma(\frac{1}{2})} }{\Gamma(\frac{6}{2})}\frac{t_1^2(t_2-t_1)^{-\frac{1}{2}}
(t_3-t_2)^{-\frac{1}{2}}}{t_3}\nonumber\\
& \qquad\qquad \times e^{-t_3}\left(1+o\left(\frac{1}{a}\right)\right)dt_1dt_2dt_3.\nonumber
\end{align}
Similarly, we integrate  $t_1$, $t_2$, $t_3$ one by one 
in this order to obtain 
  \begin{align}
 \mathcal{I}&=\frac{4}{a\pi^2}{  \frac{\Gamma(\frac{5}{2})  \Gamma(\frac{1}{2})}{\Gamma(\frac{6}{2})}\frac{\Gamma(\frac{6}{2}) \Gamma(\frac{1}{2})}{\Gamma(\frac{7}{2})}\frac{\Gamma(\frac{7}{2}) \Gamma(\frac{1}{2})}{\Gamma(\frac{8}{2})}\Gamma(\frac{6}{2})} \left(1+o\left(\frac{1}{a}\right)\right)\nonumber\\
&=\frac{1}{a}+o\left(\frac{1}{a}\right),
\end{align}
completing the proof    of \eqref{5.02}.

To prove  the results \eqref{5.9} and \eqref{5.10}, we
denote 
\[
\mathcal{L}=\mathbb{E}\left[\sum_{1 \leq i_1<i_2<\cdots <i_k<\infty}w_{i_1}^{p_1}w_{i_2}^{p_2}\cdots w_{i_k}^{p_k}\right]\,.
\]
We have 
\begin{align}
&\mathcal{L}=\sum_{1\leq i_1<i_2<\cdots<i_k<\infty}\mathbb{E}\Bigg[v_{i_1}^{p_1}(1-v_{i_1})^{p_{2:k}}\prod_{\ell_1=1}^{i_1-1}(1-v_{\ell_1})^{p_{1:k}}v_{i_2}^{p_2}(1-v_{i_2})^{p_{3:k}}\nonumber\\
    &\qquad\qquad \times 
    \prod_{\ell_2=i_1+1}^{i_2-1}(1-v_{\ell_2})^{p_{2:k}}\cdots v_{i_m}^{p_m}(1-v_{i_m})^{p_{m+1:k}}\prod_{\ell_m=i_{m-1}+1}^{i_m-1}(1-v_{\ell_m})^{p_{m:k}}\nonumber\\
 &\qquad\qquad \cdots v_{i_{k-1}}^{p_{k-1}}(1-v_{i_{k-1}})^{p_k}\prod_{\ell_{k-1}=i_{k-2}+1}^{i_{k-1}-1}(1-v_{\ell_{k-1}})^{p_k-1:k}\nonumber\\
 &\qquad\qquad v_{i_k}^{p_k}\prod_{\ell_k=i_{k-1}+1}^{i_k-1}(1-v_{\ell_k})^{p_k}\Bigg].\label{5.11}
\end{align}
Using the explicit form of  the joint density of $v_1,\cdots,v_{i_k}$, we have
\begin{align}
 \mathcal{L} &=\sum_{1\leq i_1<i_2<\cdots<i_k<\infty}\int_0^{\infty}\int_0^1\cdots\int_0^1\frac{e^aa^{i_k+1}}{(2\pi)^{\frac{i_k+1}{2}}}t^{-\frac{i_k+3}{2}}e^{-\frac{t}{2}} e^{-\frac{a^2}{2t\prod_{j=1}^{i_k}(1-v_j)}}\nonumber\\
&\qquad\qquad \times \prod_{\ell_1=1}^{i_1-1}v_{\ell_1}^{-\frac{1}{2}}(1-v_{\ell_1})^{p_{1:k}-\frac{i_k+3-\ell_1}{2}}v_{i_1}^{p_1-\frac{1}{2}}(1-v_{i_1})^{p_{2:k}-\frac{i_k+3-i_1}{2}}  \nonumber\\
&\qquad\qquad  \cdots \prod_{\ell_m=i_{m-1}+1}^{i_m-1}v_{\ell_m}^{-\frac{1}{2}}(1-v_{\ell_m})^{p_{m:k}-\frac{i_k+3-\ell_m}{2}}v_{i_m}^{p_m-\frac{1}{2}}(1-v_{i_m})^{p_{m+1:k}-\frac{i_k+3-i_m}{2}}   \nonumber\\
&\qquad\qquad \cdots \prod_{\ell_k=i_{k-1}+1}^{i_k-1}v_{\ell_k}^{-\frac{1}{2}}(1-v_{\ell_k})^{p_{k}-\frac{i_k+3-\ell_k}{2}}v_{i_k}^{p_k-\frac{1}{2}}(1-v_{i_k})^{-\frac{3}{2}} dv_1\cdots dv_{i_k}dt.\label{5.12}
\end{align}
Notice that the integrals of $v_{i_{k-1}+1}, v_{i_{k-1}+2}, \cdots v_{i_k}$ with the sum of $i_k$ from $i_{k-1}+1$ to $\infty$ is the same form as \eqref{5.3}.
 Thus, by the computation  \eqref{5.61}, we have
\begin{align}
 \mathcal{L} =&\sum_{1\leq i_1<i_2<\cdots<i_k<\infty}\int_0^{\infty}\int_0^1\cdots\int_0^1\frac{e^aa^{i_k+1}}{(2\pi)^{\frac{i_k+1}{2}}}t^{-\frac{i_k+3}{2}}e^{-\frac{t}{2}} e^{-\frac{a^2}{2t(1-v_{i_k})\prod_{j=1}^{i_{k-1}}(1-v_j)}}\nonumber\\
&\qquad \prod_{\ell_1=1}^{i_1-1}v_{\ell_1}^{-\frac{1}{2}}(1-v_{\ell_1})^{p_{1:k}-\frac{i_k+3-\ell_1}{2}}v_{i_1}^{p_1-\frac{1}{2}}(1-v_{i_1})^{p_{2:k}-\frac{i_k+3-i_1}{2}}\cdots \nonumber\\
&\qquad \prod_{\ell_{k-1}=i_{k-2}+1}^{i_{k-1}-1}v_{\ell_{k-1}}^{-\frac{1}{2}}(1-v_{\ell_{k-1}})^{p_{k}-\frac{i_k+3-\ell_{k-1}}{2}}v_{i_{k-1}}^{p_{k-1}-\frac{1}{2}}(1-v_{i_{k-1}})^{-\frac{3}{2}}\nonumber\\
&\qquad  \left(\Gamma(\frac{1}{2})\left(\frac{a^2}{2t(1-v_{i_k})\prod_{j=1}^{i_{k-1}}(1-v_j)}\right)^{-\frac{1}{2}}\right)^{i_k-i_{k-1}-1}\nonumber\\
&\qquad \times \left(1+o\left(\frac{1}{a}\right)\right)dv_1\cdots dv_{i_{k-1}}dv_{v_{i_k}}dt\label{5.13}\\
 =&\sum_{1\leq i_1<i_2<\cdots<i_{k-1}<\infty}\sum_{i_k=i_{k-1}+1}^{\infty}\int_0^{\infty}\int_0^1\cdots\int_0^1\frac{e^aa^{i_{k-1}+2}}{(2\pi)^{\frac{i_{k-1}+2}{2}}}t^{-\frac{i_{k-1}+4}{2}}e^{-\frac{t}{2}} \nonumber\\
&\qquad \prod_{\ell_1=1}^{i_1-1}v_{\ell_1}^{-\frac{1}{2}}(1-v_{\ell_1})^{p_{1:k}-\frac{i_{k-1}+4-\ell_1}{2}}v_{i_1}^{p_1-\frac{1}{2}}(1-v_{i_1})^{p_{2:k}-\frac{i_{k-1}+4-i_1}{2}}\cdots \nonumber\\
&\qquad \prod_{\ell_{k-1}=i_{k-2}+1}^{i_{k-1}-1}v_{\ell_{k-1}}^{-\frac{1}{2}}(1-v_{\ell_{k-1}})^{p_{k}-\frac{i_{k-1}+4-\ell_{k-1}}{2}}v_{i_{k-1}}^{p_{k-1}-\frac{1}{2}}(1-v_{i_{k-1}})^{-\frac{4}{2}}\nonumber\\
&\qquad  v_{i_k}^{p_k}(1-v_{i_k})^{\frac{i_k-i_{k-1}-4}{2}}e^{-\frac{a^2}{2t(1-v_{i_k})\prod_{j=1}^{i_{k-1}}(1-v_j)}}\nonumber\\
&\qquad\qquad \times \left(1+o\left(\frac{1}{a}\right)\right)dv_{v_{i_k}}dv_1\cdots dv_{i_{k-1}}dt.\label{5.14}
\end{align}
By a similar calculation to   that of  \eqref{5.62},
\begin{align}
 \mathcal{L}&=
 \sum_{1\leq i_1<i_2<\cdots<i_{k-1}<\infty}\int_0^{\infty}\int_0^1\cdots\int_0^1{O}\left(\frac{1}{a}\right)\frac{e^aa^{i_{k-1}+1}}{(2\pi)^{\frac{i_{k-1}+1}{2}}}t^{-\frac{i_{k-1}+3}{2}}e^{-\frac{t}{2}} \nonumber\\
&\qquad \times \prod_{\ell_1=1}^{i_1-1}v_{\ell_1}^{-\frac{1}{2}}(1-v_{\ell_1})^{p_{1:k}-\frac{i_{k-1}+3-\ell_1}{2}}v_{i_1}^{p_1-\frac{1}{2}}(1-v_{i_1})^{p_{2:k}-\frac{i_{k-1}+3-i_1}{2}}  \nonumber\\
&\qquad \cdots  \prod_{\ell_{k-1}=i_{k-2}+1}^{i_{k-1}-1}v_{\ell_{k-1}}^{-\frac{1}{2}}(1-v_{\ell_{k-1}})^{p_{k}-\frac{i_{k-1}+3-\ell_{k-1}}{2}}v_{i_{k-1}}^{p_{k-1}-\frac{1}{2}}(1-v_{i_{k-1}})^{-\frac{3}{2}}\nonumber\\
 &\qquad \times e^{-\frac{a^2}{2t\prod_{j=1}^{i_{k-1}}(1-v_j)}}dv_{v_{i_k}}dv_1\cdots dv_{i_{k-1}}dt.\label{5.15}
\end{align}
We can perform the analogous computations for $i_{k-1}$,
$i_{k-2}$, $\cdots$, $i_1$ in this order repeatedly  to obtain 
 \eqref{5.9}.

When $p_1=\cdots=p_k=2$, similar computations to   that in the  proof of \eqref{5.02} can be carried out  to obtain \eqref{5.10}. 
\end{proof}
\subsection{Proof of Proposition \ref{prop.3.4}}
\begin{proof}
By   the identities   $\Gamma(c,x)=e^{-x}x^c\int_0^{\infty}e^{-xu}(1+u)^{c-1}du$ and $\sum_{j=0}^{\infty}\frac{(n)_{j}}{j!}x^j=(1-x)^n$, we can rewrite 
 the joint density of stick-breaking weights $v_1,\cdots,v_n$
 as 
\begin{align}
&f(v_1,\cdots,v_n)=\frac{a^n\sigma^{n-1}}{[\Gamma(1-\sigma)]^n\Gamma(n\sigma)}\prod_{i=1}^nv_i^{-\sigma}(1-v_i)^{-(n-i)\sigma-1}e^{-\frac{a}{\prod_{i=1}^n(1-v_i)^{\sigma}}}\nonumber\\
&\qquad\qquad\qquad\qquad  \times
\int_0^{\infty}(1-(1+t)^{-\frac{1}{\sigma}})^{n\sigma-1}(1+t)^{n-1}e^{-\frac{at}{\prod_{i=1}^n(1-v_i)^{\sigma}}}dt.\label{6.7}
\end{align}
We make the substitution  $t=\frac {\prod_{i=1}^n(1-v_i)^{\sigma}s}{a}$ in the above integral. Then,  
when $a$ is large,  namely, when $t$ is small, we have 
\[
(1+t)^{-\frac{1}{\sigma}} 
\asymp 1-\frac{t}{\sigma} =  1-\frac{\prod_{i=1}^n(1-v_i)^{\sigma}}{\sigma a}s\,, 
\]
where and throughout this paper we use $\mu\asymp \nu$ to represent
  $\lim \frac{\mu}{\nu}=1$.
  
The integral in \eqref{6.7} is then approximated by 
\begin{align}
\int_0^{\infty}(1-(1+t)^{-\frac{1}{\sigma}})^{n\sigma-1}(1+t)^{n-1}e^{-\frac{at}{\prod_{i=1}^n(1-v_i)^{\sigma}}}dt \asymp \frac{\prod_{i=1}^n(1-v_i)^{n\sigma^2}\Gamma(n\sigma)}{\sigma^{n\sigma-1}a^{n\sigma}}.\label{6.8}
\end{align}
Thus, for large $a$,   the joint density of $v_1,\cdots,v_n$ 
has the following asymptotics: 
\begin{align}
&f(v_1,\cdots,v_n)\asymp \frac{(a\sigma)^{n-n\sigma}}{[\Gamma(1-\sigma)]^n}\prod_{i=1}^nv_i^{-\sigma}(1-v_i)^{n\sigma^2-(n-i)\sigma-1}e^{-\frac{a}{\prod_{i=1}^n(1-v_i)^{\sigma}}}.\label{6.9}
\end{align}
Now  the  equalities \eqref{6.3} and \eqref{6.4} follows from  the same arguments as that  in the proof of Proposition \ref{prop.3.3}  and from the use of  the following identity, which holds true 
for any $q\in \mathbb{R}$:
\begin{align}
&\int_0^1x^{-\sigma}(1-x)^qe^{-\frac{a}{(1-x)^{\sigma}}}dx\nonumber\\
&\qquad =\frac{e^{-a}}{a\sigma}\int_0^{\infty}(1-(\frac{a}{a+s})^{\frac{1}{\sigma}})^{-\sigma}(\frac{a}{a+s})^{\frac{q+1}{\sigma}+1}e^{-s}ds\nonumber\\
&\qquad = e^{-a}(a\sigma)^{\sigma-1}(1+o(\frac{1}{a}))\int_0^1s^{-\sigma}e^{-s}ds\nonumber\\
&\qquad=e^{-a}(a\sigma)^{\sigma-1}\Gamma(1-\sigma)\left(1+o\left(\frac{1}{a}\right)\right),\label{6.10}
\end{align}
where the last equality   follows from 
the substitution  $s=\frac{a}{(1-x)^{-\sigma}}-a$ and the following asymptotics:
\[\left(\frac{a}{a+s}\right)^{\frac{1}{\sigma}}=\left(1-\frac{s}{a+s}\right)^{\frac{1}{\sigma}}=1-\frac{s}{\sigma(a+s)}+o\left(\frac{1}{a}\right)\,.
\]

To obtain the exact asymptotics in the case  when $\sigma=\frac{1}{m}$ and $p=p_1=\cdots=p_k=2$, we first prove   \eqref{6.5} using 
 the same argument    
as that in the proof of    \eqref{5.02}. The only differences are as follows. First, we   integrate   $v_{n-m}$, $v_{n-m-1}$, $\cdots$,   $v_1$    recursively
in this order by using \eqref{6.10}.  After these integrations, it remains to  integrate  the variables  $v_{n-m+1},\cdots,v_n$. This multiple integral is now evaluated 
simultaneously by using   the  substitution 
\[
v_{n-i}=1-\left(\frac{a+y_0+\cdots+y_{i-1}}{a+y_0+\cdots+y_i}\right)^{\frac{1}{\sigma}},  \qquad  i =0,\cdots,m-1 \,.
\]
 The identity  \eqref{6.6} follows from  \eqref{6.5} by  the same argument as that in the proof of \eqref{5.10}.
\end{proof}
\subsection{Proof of Proposition \ref{prop.3.5}}
\begin{proof}
When $P \sim \GDP(a,r,H)$, we will prove the weak convergence of $D_a$ in the part (i) of Theorem \ref{theorem 2.17}.
Combining with the fact that $\GDP(a,r,H)$ also admits the general stick-breaking representation as in \eqref{1.1} and \eqref{1.111}, we can obtain the  desired results by using  the same argument as that in the proof of \eqref{e.3.19} in Theorem \ref{theorem 0.1}.
\end{proof}

\subsection{Proof of Theorem \ref{law of large number}}
\begin{proof}
\begin{align}
 &\mathbb{P}\left(|P(A)-H(A)|>\epsilon\right) 
\le \frac{\EE \left[|P(A)-H(A)|^m\right]}{\epsilon^m}\nonumber\\
 &\qquad\qquad=\frac{\EE \left[|\sum_{i=1}^{\infty}w_i\left(\delta_{\theta_i}(A)-H(A)\right)|^m\right]}{\epsilon^m}\nonumber\\
&\qquad\qquad=\sum \frac{c(p_1,\cdots,p_k)}{\epsilon^m}\prod_{i=1}^k\EE \left[|\delta_{\theta_i}(A)-H(A)|^{p_i}\right]\nonumber\\
&\qquad\qquad\qquad\qquad \times  \mathbb{E}\left[\sum_{1\leq i_1<i_2<\cdots<i_k<\infty}w_{i_1}^{p_1}w_{i_2}^{p_2}\cdots w_{i_k}^{p_k}\right]  \label{adde.2.16}
\,, 
\end{align}
where the first 
 sum is taken over all combinations of nonnegative integers $\{p_1, \cdots, p_k\}$   such that  $k \in \{1, \cdots, m\}$ and  $\sum_{i=1}^k p_i=m$ and   $c(p_1,\cdots,p_k)$ are the corresponding constants. By the discussion of Case 1 in the proof of Theorem \ref{theorem 0.1}, it is easy to see that $p_i \geq 2$ for all $i$ and thus $k \leq \frac{m}{2}$.

First, assume  $P$ is one of $\DP(a,H)$,  $\PDP(a,b,H)$,  $\N-IGP(a,H)$,  $\NGGP(\sigma,a,H)$, $\GDP(a,r,H)$.  We choose   $m=\lfloor \frac{4}{\tau} \rfloor$, where $\lfloor x \rfloor$ is the smallest integer that is greater  than or equal to $x$. Then,
 from Propsition \ref{prop.3.2}-\ref{prop.3.5}, we have 
\begin{align*}
 \sum\mathbb{E}\left[\sum_{1\leq i_1<i_2<\cdots<i_k<\infty}w_{i_1}^{p_1}w_{i_2}^{p_2}\cdots w_{i_k}^{p_k}\right] \asymp \frac{1}{a^{m-k}}\leq \frac{1}{a^{\frac{m}{2}}  }=\frac{1}{n^{\frac{m\tau}{2}}}\leq\frac{1}{n^2}\,.  
\end{align*}
If $p \sim \DPG(g_a,H)$,  then we  can choose $m$ 
such that  for all $1\le k\le m/2$, 
\[
\sum_{1\le i\le k, p_1+\cdots+p_k=m} q_{p_i}- kq_1  
\ge \frac{m (q_2-q_1)}{2}  \geq \frac{2}{\tau}\,. 
\]
 Then, from Proposition \ref{prop.3.1},
\begin{align*}
\sum \mathbb{E}\left[\sum_{1\leq i_1<i_2<\cdots<i_k<\infty}w_{i_1}^{p_1}w_{i_2}^{p_2}\cdots w_{i_k}^{p_k}\right] \asymp \frac{1}{a^{\sum_{i=1}^k q_{p_i}-kq_1}}\leq\frac{1}{n^2}\,.
\end{align*} 
 Since the series $\sum_{n=1}^{\infty}\frac{1}{n^2}$ converges, it follows 
\begin{align}
\sum_{n=1}^{\infty}\mathbb{P}\left(|P(A)-H(A)|>\epsilon\right) <\infty\ 
\label{8.1}
\end{align}
for any of the processes presented in the theorem. 
This  implies \eqref{e.2.9}  by the    Borel-Cantelli lemma.
\end{proof}
\subsection{Proof of Theorem \ref{theorem 0.1}}
Before we proceed to the proof of   Theorem \ref{theorem 0.1},   
we need a  preparatory result  about   the joint moments of multivariate normal distribution. To state this result, we introduce the following notations.
Let $n$ be a positive integer and let 
$\vec{p}=(p_{ij}, 1\le i<j\le n)$ be a multi-index. Denote
\begin{equation}
|\vec{p}|=\sum_{1\le i<j\le n} p_{ij}
\end{equation}
and 
denote 
\begin{equation}
|\vec{p}|_m= 
\begin{cases}
 \sum_{j>1}p_{1j}  &\qquad \hbox{when $m=1$,}\\ 
 \sum_{j>m}p_{mj}+\sum_{i<m}p_{im}  &\qquad \hbox{when $m= 2, \cdots, n-1$}, \\
 \sum_{i<n}p_{in}&\qquad \hbox{when $m=n$}\,. 
\end{cases} 
\end{equation}
The following proposition is about the joint moments of Gaussian random variables.
Similar or more general results may be found in literature 
under  the terminology of ``Feynman diagram"
(e.g. \citet[Theorem 5.7]{hu2017book} and references therein).  
But we could not find the exact  result we need. So,  we give 
 the following proposition. 
\begin{prop}\label{prop.6.4}
Let the random vector $(X_1,X_2,\cdots,X_n)$ follow  a multivariate normal distribution $N_n(0,\Sigma)$, where $ \Sigma =\left(\sigma_{ij}=\EE(X_iX_j)\right)_{1\le i,j\le n}$ and $\si_{ii}=\sigma_i^2$.
For any nonnegative integers $r_1,\cdots, r_n$, the joint $(r_1,\cdots, r_n)$ moments of $(X_1,X_2,\cdots,X_n)$ is given by the following formulas. 
\begin{enumerate}
\item[(i)] When $\sum_{\ell=1}^n r_\ell$ is an odd integer,  we have   
\begin{equation}
\mathbb{E}\left[X_1^{r_1}\cdots X_n^{r_n}\right]=0\,. 
\end{equation}
\item[(ii)] When   $\sum_{\ell=1}^n r_\ell$ is an even integer, we have 
(we use the convention that $0^0=1$)
\begin{align}
     \mathbb{E}&\left[X_1^{r_1} \cdots X_n^{r_n}\right]\nonumber\\
     =&\sum_{\vec{p} \in \cC}\frac{r_1! \cdots r_n!\ 
     \prod_{m=1}^n \sigma_m^{r_m-|\vec{p}|_m}  \ \prod_{1\le i<j\le n} 
     \si_{ij}^{p_{ij}} }{
     2^{\frac{|r|}2  -|\vec{p}|}\ \prod_{m=1}^n ((r_m-|\vec{p}|_m)/2 )! \ \prod_{
     1\le i<j\le n}p_{ij}!}\,, 
    \label{0.8}
\end{align}
where  $|r|=r_1+\cdots+r_n$ and 
\begin{align}
\cC=\bigg\{  \vec{p}= \{p_{ij}, &1\le i< j\le n\}\,; \quad 
0\le p_{ij}\le r_i\wedge r_j\,,  \nonumber\\
&
r_m- |\vec{p}|_m \ \hbox{  $m=1, \cdots, n$,  \ are   all even}\bigg\} \,. 
\label{e.3.13}
\end{align}
\end{enumerate} 
\end{prop}
\begin{proof}We shall use the moment generating function 
 to prove it.
On   one hand, we see  
\begin{align}
    \mathbb{E}\left[e^{\sum_{\ell=1}^n t_\ell X_\ell}\right]=\sum_{r_1,\cdots,r_n=0}^\infty \frac{t_1^{r_1}\cdots t_n^{r_n}}{r_1!\cdots r_n!}\mathbb{E}\left[X_1^{r_1}\cdots X_n^{r_n}\right]\,. \label{0.6}
\end{align}
On the other hand, by the moment generating function formula for multivariate 
normal variables, we have  
\begin{align}
    &\mathbb{E}\left[e^{\sum_{\ell=1}^n t_\ell X_\ell}\right]=e^{\frac{1}{2}\mathbb{E}\left[(\sum_{\ell=1}^n t_\ell X_\ell)^2\right]}\\
    &\qquad =\exp\left\{\frac{1}{2}\sum_{\ell=1}^n \mathbb{E}\left[(t_\ell X_\ell)^2\right]+\sum_{1\le i<j\le n}\mathbb{E}\left[t_it_jX_iX_j\right]\right\}\nonumber\\
    &\qquad =\prod_{\ell=1}^n \exp\left\{\frac{t_\ell^2 \si_\ell^2 }{2}   \right\}
   \prod_{1\le i<j \le n } \exp\left\{t_it_j \si_{ij} \right\}\nonumber\\
    &\qquad =\sum_{{0\le p_l, p_{ij}<\infty\atop  1\le \ell \le n}\atop{ 1\le i<j\le n}}
    \prod_{\ell=1}^n \frac{(t_\ell\sigma_\ell)^{2p_\ell}}{2^{p_\ell}p_\ell!} \prod_{ 1\le i<j \le n }\frac{(t_it_j\sigma_{ij})^{p_{ij}}}{p_{ij}!}\nonumber\\
    &\qquad =\sum_{{0\le p_l, p_{ij}<\infty\atop  1\le \ell\le n}\atop{ 1\le i<j\le n}}
     \frac{t_1^{2p_1+|\vec{p}|_1}\cdots 
     t_n^{2p_n+|\vec{p}|_m}}{2^{p_1+\cdots+p_n}} \prod_{\ell=1}^n\frac{\sigma_\ell^{2p_\ell}}{p_\ell!}\prod_{1\le i<j\le n}\frac{\sigma_{ij}^{p_{ij}}}{p_{ij}!}\,. \label{0.7} 
\end{align}
Comparing the coefficients of $t_1^{r_1}\cdots t_n^{r_n}$ of the 
two representations (\ref{0.6}) and (\ref{0.7})   we  obtain 
\begin{equation}
\mathbb{E}\left[X_1^{r_1}\cdots X_n^{r_n}\right]
=\frac{ r_1!\cdots r_n!}{2^{p_1+\cdots+p_n}} \prod_{\ell=1}^n\frac{\sigma_\ell^{2p_\ell}}{p_\ell!}\prod_{1\le i<j \le n}\frac{\sigma_{ij}^{p_{ij}}}{p_{ij}!} \,, 
\label{e.3.16} 
\end{equation} 
where $p_{\ell}, p_{ij}$, $1\le \ell\le n, 1\le i<j\le n$ satisfies the
relation $r_\ell=2p_\ell+|\vec{p}|_\ell$ for $\ell=1, \cdots, n$,  
which implies that 
$\sum_{\ell=1}^n r_\ell=2(\sum_{\ell=1}^n p_\ell+\sum_{1\le  i,j\le n}p_{ij})$ is an  even number. This also proves that  $\mathbb{E}\left[X_1^{r_1}\cdots X_n^{r_n}\right]=0$  when $\sum_{\ell=1}^n r_\ell$ is an odd integer.
This proves part (i) of the proposition. 
When $\sum_{\ell=1}^n r_\ell$ is even, by the fact that $\{p_1, \cdots, p_n\}$ are nonnegative integers and by the relationship
 $p_\ell=\frac{r_\ell-|\vec{p}|_\ell}{2}$ for $\ell=1, \cdots, n$, \, 
one see the summation in \eqref{e.3.16}  is over the set $\cC$ defined by  \eqref{e.3.13}.
\end{proof}

Now, we are in the position to prove Theorem \ref{theorem 0.1}.

\noindent {\it Proof of Theorem \ref{theorem 0.1}.}\quad 
We first prove the part (i) of this theorem.

For any $A \in \mathcal{X}$, the variance of $P(A)$ is given by \eqref{1.112}.
Using   Proposition \ref{prop.3.1},  we have 
\begin{align*}
	\Var(P(A))=H(A)\left(1-H(A)\right)\left(\frac{\mathbb{E}[v_1^2]}{2\mathbb{E}[v_1]}+o\left(\frac{\mathbb{E}[v_1^2]}{\mathbb{E}[v_1]}\right)\right).
\end{align*}
From the definition of $D_a(A)$ we have 
\begin{align*}
D_a(A)=\left[H(A)\left(1-H(A)\right)\left(\frac{\mathbb{E}[v_1^2]}{2\mathbb{E}[v_1]}+o\left(\frac{\mathbb{E}[v_1^2]}{\mathbb{E}[v_1]}\right)\right)\right]^{-\frac{1}{2}}\left(P(A)-H(A)\right).
\end{align*}
By   the Cram\'{e}r-Wold theorem
(e.g. \citet[Theorem 29.4]{Billingsley}), to show \eqref{e.3.19} 
it is sufficient to show
that  for any $ (t_1,\cdots,t_n) \in \mathbb{R}^n$
$$\sum_{i=1}^n t_iD_a(A_i) \overset{d}{\rightarrow} \sum_{i=1}^n t_iX_i,$$
where and throughout the remaining part of the paper  
 $(X_1, \cdots, X_n)$ are jointly Gaussian  with mean zero and covariance given by  \eqref{e.2.24a}. 
For any positive integer $ n$ and a nonnegative  integer sequence $\{r_i\}_{i=1}^n$, consider the joint moments  of $D_a(A_1),D_a(A_2),\cdots,D_a(A_n)$: 
\begin{align}
    &\mathbb{E}\left[\prod_{i=1}^n D_a^{r_i}(A_i)\right]\nonumber\\
    &=\prod_{i=1}^n\left[H(A_i)(1-H(A_i))\left(\frac{\mathbb{E}[v_1^2]}{2\mathbb{E}[v_1]}+o\left(\frac{\mathbb{E}[v_1^2]}{\mathbb{E}[v_1]}\right)\right)\right]^{-\frac{r_i}{2}}\mathbb{E}\left[\prod_{i=1}^n \left(P(A_i)-H(A_i)\right)^{r_i}\right]\nonumber\\
    &=\prod_{i=1}^n\left[H(A_i)(1-H(A_i))\left(\frac{\mathbb{E}[v_1^2]}{2\mathbb{E}[v_1]}+o\left(\frac{\mathbb{E}[v_1^2]}{\mathbb{E}[v_1]}\right)\right)\right]^{-\frac{r_i}{2}}\times \nonumber\\
    &\qquad \qquad \mathbb{E}\left[\prod_{i=1}^n \left(\sum_{j_i=1}^\infty w_{j_i}\left(\delta_{\theta_{j_i}}(A_i)-H(A_i)\right)\right)^{r_i}\right]\nonumber\\
    &=\sum \prod_{i=1}^n\left[H(A_i)(1-H(A_i))\left(\frac{\mathbb{E}[v_1^2]}{2\mathbb{E}[v_1]}+o\left(\frac{\mathbb{E}[v_1^2]}{\mathbb{E}[v_1]}\right)\right)\right]^{-\frac{r_i}{2}}
    \nonumber\\
    &\qquad\qquad C ( q;s_1,\cdots,s_n,  s_{1,1}, \cdots, s_{q,n}) 
    I ( q;s_1,\cdots,s_n,  s_{1,1}, \cdots, s_{q,n}) 
,\label{0.11}
\end{align}
where the   sum  is  over all  the  nonnegative integers $\{q;s_1,\cdots,s_n; s_{j,i}: j \in \{1,\cdots, q\}; i \in \{1,\cdots, n\}\}$ satisfying 
\begin{itemize}
    \item[(i)]\quad    $1\le q \le \sum_{i=1}^n r_i $; 
    \item[(ii)]\quad   $s_j=\sum_{i=1}^n s_{j,i}$ for all $j\in \{1,\cdots,q\}$
    \item[(iii)]\quad   $r_i=\sum_{j=1}^q s_{j,i}$ for all $i\in \{1,\cdots, n\}$
    \item[(iv)]\quad   $\sum_{j=1}^q s_j=\sum_{i=1}^n r_i$
    \item[(v)]\quad   $C ( q;s_1,\cdots,s_n,  s_{1,1}, \cdots, s_{q,n})$ are  some constants (that are found later on) depending on $q;s_1,\cdots,s_n,  s_{1,1}, \cdots, s_{q,n}$  
\end{itemize}
and 
\begin{eqnarray}
&&I ( q;s_1,\cdots,s_n,  s_{1,1}, \cdots, s_{q,n})\nonumber\\
&&\qquad\qquad := \mathbb{E}\Bigg[ \sum_{1\leq e_1<\cdots<e_q<\infty}\prod_{j=1}^q      
     w_{e_j}^{s_j}\left(\delta_{\theta_{e_j}}(A_1)-H(A_1)\right)^{s_{j,1}} \nonumber\\
     &&\qquad\qquad\qquad\qquad \cdots(\delta_{\theta_{e_j}}(A_n)-H(A_n))^{s_{j,n}}\Bigg] \,. \label{e.3.22} 
\end{eqnarray}
%
%
%
%
We will divide the discussion of  the limit as $a\rightarrow \infty$
of the terms in \eqref{0.11} into  three cases. 

\noindent \textbf{Case 1:}\quad 
\textit{There exists at least one $j\in\{1, \cdots, q\}$ such that  $s_j=1$ 
or     
there exists at least one pair $(j,k)$ such that  $s_{j,k} =1$.} 

From  the fact that 
\[
\mathbb{E}\left[\delta_{\theta_i}(A_i)-H(A_i)\right]=
\mathbb{E}\left[ \mathbbm{1}_{A_i}(\th_i)\right] -H(A_i) =\int _{A_i} dH-H(A_i)=0\,,
\]
we see that  in this case 
the corresponding terms in the sum 
of $I ( q;s_1,\cdots,s_n,  s_{1,1},\\
 \cdots, s_{q,n})$   are 
identically equal to 
$0$. 

\noindent \textbf{Case 2}\quad  
  \textit{  $\sum_{i=1}^n r_i$ is odd or $\sum_{i=1}^n r_i$ is even but  $\frac{\sum_{i=1}^n r_i}{2}>q$.} 

We substitute \eqref{e.3.22} 
into \eqref{0.11} and we consider the expectations of $v_1$'s. 
 By Proposition \ref{prop.3.1}, when  excluding the 
 terms discussed in  \textit{Case 1}  the remaining terms   corresponding to
 this case 
 have the following 
 asymptotics 
\begin{align}
{O}\left(\left[\frac{\mathbb{E}[v_1^2]}{\mathbb{E}[v_1]}\right]^{-\frac{\sum_{i=1}^n r_i}{2}}\frac{\prod_{j=1}^q \mathbb{E}[v_1^{s_j}]}{\left(\mathbb{E}[v_1]\right)^q}\right)={O}\left(\frac{\left(\mathbb{E}[v_1]\right)^{\frac{\sum_{j=1}^q s_j}{2}-q}\prod_{j=1}^q \mathbb{E}[v_1^{s_j}]}{\left(\mathbb{E}[v_1^2]\right)^{\frac{\sum_{j=1}^q s_j}{2}}}\right).\label{0.12}
\end{align}  
From the  assumption \eqref{0.10}, it follows that when 
$\sum_{i=1}^n r_i$ is even and when $q<\frac{\sum_{i=1}^n r_i}{2}$,  the expectation of the terms in the sum of $I ( q;s_1,\cdots,s_n,  s_{1,1}, \cdots, s_{q,n})$
 will converge to $0$ as $a \rightarrow \infty$.

Similarly, when  $\sum_{i=1}^n r_i$ is odd, 
since $q$ is an integer and $s_j \geq 2$ for all $j$, we always have $q<\frac{\sum_{i=1}^n r_i}{2}$. Therefore, the expectation of the corresponding terms 
satisfying the condition that   $\sum_{i=1}^n r_i$ is odd
in the sum of $I ( q;s_1,\cdots,s_n,  s_{1,1}, \cdots,\\
 s_{q,n})$ 
will always converge  to $0$ as $a \rightarrow \infty$.

\noindent \textbf{Case 3}\quad 
  \textit{   $\sum_{i=1}^n r_i$ is even  and   $\frac{\sum_{i=1}^n r_i}{2}= q$.}

 The only terms that may not converge to zero are the terms 
  that are not covered in      \textit{Case 1} and \textit{Case 2}.
This means that the only terms that have   nontrivial limits  
are  the terms 
satisfying the condition that
\[
\sum_{i=1}^n r_i \  \hbox{    is even and }\quad q=\frac{\sum_{i=1}^n r_i}{2}\,. 
\]
In this case, we must have 
\[
s_1=\cdots=s_q=2\quad{\rm and}\quad
s_{j,i} \in \{0,1,2\}
\]
for all  $j \in \{1,\cdots,q\}$ and for all 
$i \in \{1,\cdots,n\}$. And thus, for $\ell \in \{1, \cdots, \frac{\sum_{i=1}^n r_i}{2}\}$ the factors in $I ( q;s_1,
\cdots,s_n,  s_{1,1}, \cdots, s_{q,n})$ are either of the form  $w_{e_\ell}^2(\delta_{\theta_{e_\ell}}(A_i)-H(A_i))(\delta_{\theta_{e_\ell}}(A_j)-H(A_j))$ 
for $1\le i<j\le n$  (we call this factor the {\textit{ $(ij)$-mixed term}}, and if
there is no ambiguity to omit   the pre-index $(ij)$, we call this kind of factor the \textit{mixed term})   or of the form   $w_{e_\ell}^2(\delta_{\theta_{e_\ell}}(A_d)-H(A_d))^2$ (which we call  the \textit{ power two  term of $A_d$}) for $d \in \{1,\cdots,n\}$. The corresponding coefficients  $C( q;s_1,\cdots,s_n,  s_{1,1}, \cdots, s_{q,n})$   are computed   as follows.

For each $\ell \in \{1, \cdots, \frac{\sum_{i=1}^n r_i}{2}\}$, let 
\[
p_{ij}:=\#\left\{(i,j); 1\le i<j\le n\,, \     s_{\ell,i}=s_{\ell,j}=1     \right\}\,. 
\]
{
Namely,  for each  pair of $i,j$ such that $1\le i<j\le n$, $p_{ij}$ is the number of $(ij)$-mixed terms in the product  of $I ( q;s_1,\cdots,s_n,  s_{1,1}, \cdots, s_{q,n})$. 
Notice that, in order to obtain a $(ij)$-mixed term, we need to multiply the form $w_{e_{\ell}}(\delta_{\theta_{e_{\ell}}}(A_i)-H(A_i))$ (we call this form the {\textit{ power $1$ term of $A_i$}} and there are $r_i$ power $1$ terms of $A_i$) and the form $w_{e_{\ell}}(\delta_{\theta_{e_{\ell}}}(A_j)-H(A_j))$  (we call this form the {\textit{ power $1$ term of $A_j$}} and there are $r_j$ power $1$ terms of $A_i$). Moreover, 
 there are $p_{ij}!$ ways to get as many as $p_{ij}$ $(ij)$-mixed terms. Therefore, there are totally $\prod_{ 1\le i<j \le n}p_{ij}!$ mixed terms in $I ( q;s_1,\cdots,s_n,  s_{1,1}, \cdots, s_{q,n})$.

Now for each $d \in \{1,\cdots,n\}$, after picking 
up  $p_{dj}$ power $1$ terms of $A_d$, there will be $\frac{r_d-\sum_{j>d}p_{dj}-\sum_{i<d}p_{id}}{2}$ power $2$ terms of $A_d$ left for us to pick up.
Thus,  for each $d \in \{1,\cdots,n\}$, there are $r_d$ power $1$ terms of $A_d$, in which $\sum_{j>d}^np_{dj}+\sum_{i<d}p_{id}$ of them will be used to construct the mixed terms and  $r_d-\sum_{j>d}^np_{dj}-\sum_{i<d}p_{id}$ of them will be used to construct the power $2$ terms of $A_d$. 
They have  to satisfy the conditions  (denoted by $\mathcal{C}$)
\begin{align}
\nonumber
\begin{cases}
r_1-\sum_{j>1}p_{1j}\ \  \text{is even},\\
r_{\ell}-\sum_{j>\ell}p_{\ell j}-\sum_{i<\ell}p_{i\ell} \ \ \text{is even}, &\mbox{for } \ell \in \{2, \cdots, n-1\},\\
r_n-\sum_{i<n}p_{in} \ \ \text{is even}.
\end{cases}
\end{align}
Thus, for each $d \in \{1,\cdots,n\}$, we can construct the mixed terms and power $2$ terms of $A_d$ as follows.
Since there are $\sum_{j>d}^np_{dj}+\sum_{i<d}p_{id}$  mixed terms, we choose
\\ $(p_{1d},\cdots,p_{(d-1)d},p_{d(d+1)},\cdots,p_{dn})$  (some of them could be $0$) out of $r_d$ and then combine the 
remaining  $r_d-\sum_{j>d}^np_{dj}-\sum_{i<d}p_{id}$ power $1$ terms of $A_d$ as the power $2$ terms of $A_d$.
Thus, the number of terms  from    the above steps is 
$$\frac{{r_d \choose p_{1d},\cdots,p_{(d-1)d},p_{d(d+1)},\cdots,p_{dn}} {r_d-\sum_{j>d}^np_{dj}-\sum_{i<d}p_{id} \choose 2} {r_d-\sum_{j>d}^np_{dj}-\sum_{i<d}p_{id}-2 \choose 2}\cdots {2 \choose 2}}{\left(\frac{r_d-\sum_{j>d}^np_{dj}-\sum_{i<d}p_{id}}{2}\right)!}$$
 without ordering.  After the above steps, all the mixed terms and  power $2$ terms of $A_d$ can be ordered in $\frac{\sum_{l=1}^n r_l}{2}!$ ways.
Noticing that $p_{d_1d_2} \in \{0,1,\cdots,r_{d_1}\wedge r_{d_2}\}$, the coefficient  $C( q;s_1,\cdots,s_n,  s_{1,1}, \cdots, s_{q,n})$   is then 
\begin{align}\nonumber
    &C( q;s_1,\cdots,s_n,  s_{1,1}, \cdots, s_{q,n})\nonumber\\
    &=\sum_{\substack{p_{ij}=0,1\le i<j \le n  \\
    (i,j)\in \mathcal{C}}}^{r_i \wedge r_j}\ \prod_{d=1}^n \frac{{r_d \choose p_{1d},\cdots,p_{(d-1)d},p_{d(d+1)},\cdots,p_{dn}} {r_d-\sum_{j>d}^np_{dj}-\sum_{i<d}p_{id} \choose 2} \cdots {2 \choose 2}}{\left(\frac{r_d-\sum_{j>d}^np_{dj}-\sum_{i<d}p_{id}}{2}\right)!}\nonumber\\
    &\qquad\qquad \left(\prod_{1\le i<j \le n}p_{ij}!\right)\left(\frac{\sum_{\ell=1}^n r_\ell}{2}\right)!\nonumber\\
    &=\sum_{\substack{p_{ij}=0,1\le i<j \le n  \\
    (i,j)\in \mathcal{C}}}^{r_i \wedge r_j}\ \prod_{d=1}^n \frac{{r_d \choose p_{1d},\cdots,p_{(d-1)d},p_{d(d+1)},\cdots,p_{dn}} \left(r_d-\sum_{j>d}^np_{dj}-\sum_{i<d}p_{id}\right)!}{\left(\frac{r_d-\sum_{j>d}^np_{dj}-\sum_{i<d}p_{id}}{2}\right)!2^{\frac{r_d-\sum_{j>d}^np_{dj}-\sum_{i<d}p_{id}}{2}}}\nonumber\\
    &\qquad\qquad  \left(\prod_{1\le  i,j \le n}p_{ij}!\right)\left(\frac{\sum_{\ell=1}^n r_\ell}{2}!\right)\nonumber\\
    &=\sum_{\substack{p_{ij}=0,1\le i<j \le n  \\
    (i,j)\in \mathcal{C}}}^{r_i \wedge r_j} \ \frac{r_1! \cdots r_n!}{\left(\frac{r_1-\sum_{j>1}p_{1j}}{2}\right)! \cdots \left(\frac{r_m-\sum_{j>m}p_{mj}-\sum_{i<m}p_{im}}{2}\right)! \cdots \left(\frac{r_n-\sum_{i<n}p_{in}}{2}\right)!}\nonumber\\
    &\qquad\qquad  \frac{\left(\frac{\sum_{\ell=1}^n r_\ell}{2}\right)!}{2^{\frac{1}{2}\sum_{\ell=1}^nr_\ell-\sum_{1\le  i,j\le n}p_{ij}}\left(\prod_{1\le i<j\le n }p_{ij}!\right)}.\label{eq.3.21}
\end{align}
Consequently, $I ( q;s_1,\cdots,s_n,  s_{1,1}, \cdots, s_{q,n})$ can be rewritten by Proposition   \ref{prop.3.1} as
\begin{align}
&I ( q;s_1,\cdots,s_n,  s_{1,1}, \cdots, s_{q,n})=\nonumber\\
&\EE\left[\sum_{1\leq e_1<\cdots<e_{\frac{\sum_{i=1}^n r_i}{2}}<\infty}\prod_{j=1}^{\frac{\sum_{i=1}^n r_i}{2}} w_{e_j}^2\right]\prod_{d=1}^n\left(\mathbb{E}\left[\left(\delta_{\theta}(A_d)-H(A_d)\right)^2\right]\right)^{\frac{r_d-\sum_{j>d}p_{dj}-\sum_{i<d}p_{id}}{2}}\nonumber\\
      &\qquad \quad \prod_{1\le i<j\le n }\left(\mathbb{E}\left[\left(\delta_{\theta}(A_i)-H(A_i)\right)\left(\delta_{\theta}(A_j)-H(A_j)\right)\right]\right)^{p_{ij}}
 \nonumber\\      
      &\left(\frac{1}{2^{\frac{\sum_{i=1}^n r_i}{2}}\left(\frac{\sum_{i=1}^n r_i}{2}\right)!}\left(\frac{\mathbb{E}[v_1^2]}{\mathbb{E}[v_1]}\right)^{\frac{\sum_{i=1}^n r_i}{2}}+o\left(\left(\frac{\mathbb{E}[v_1^2]}{\mathbb{E}[v_1]}\right)^{\frac{\sum_{i=1}^n r_i}{2}}\right)\right)\nonumber\\
 &\prod_{d=1}^n \left(H(A_d)(1-H(A_d))\right)^{\frac{r_d-\sum_{j>d}p_{dj}-\sum_{i<d}p_{id}}{2}}\prod_{1\le i<j \le n
 }(-H(A_i)H(A_j))^{p_{ij}}\,. \label{eq.3.23}
\end{align}
From  \eqref{0.11} and  \eqref{eq.3.21}, \eqref{eq.3.23}, we can 
compute  $\mathbb{E}\left[\prod_{i=1}^n D_a^{r_i}(A_i)\right]$ as follows.
\begin{align}\nonumber
	&\mathbb{E}[\prod_{i=1}^n D_a^{r_i}(A_i)]=\prod_{i=1}^n\left[H(A_i)(1-H(A_i))\left(\frac{\mathbb{E}[v_1^2]}{2\mathbb{E}[v_1]}+o\left(\frac{\mathbb{E}[v_1^2]}{\mathbb{E}[v_1]}\right)\right)\right]^{-\frac{r_i}{2}}\\\nonumber
    &\left(\frac{1}{2^{\frac{\sum_{i=1}^n r_i}{2}}\left(\frac{\sum_{i=1}^n r_i}{2}\right)!}\left(\frac{\mathbb{E}[v_1^2]}{\mathbb{E}[v_1]}\right)^{\frac{\sum_{i=1}^n r_i}{2}}+o\left(\left(\frac{\mathbb{E}[v_1^2]}{\mathbb{E}[v_1]}\right)^{\frac{\sum_{i=1}^n r_i}{2}}\right)\right)\nonumber\\
    &\sum_{{p_{ij}=0, 1\le i<j\le n\atop (i,j)\in \cC}}^{r_i \wedge r_j}\frac{r_1! \cdots r_n!}{\left(\frac{r_1-\sum_{j>1}p_{1j}}{2}\right)! \cdots \left(\frac{r_m-\sum_{j>m}p_{mj}-\sum_{i<m}p_{im}}{2}\right)! \cdots \left(\frac{r_n-\sum_{i<n}p_{in}}{2}\right)!}\nonumber\\
    &\frac{\left(\frac{\sum_{i=1}^n r_i}{2}\right)!}{2^{\frac{1}{2}\sum_{\ell=1}^nr_\ell-\sum_{1\le i<j\le n }p_{ij}}\left(\prod_{1\le i<j\le n }p_{ij}!\right)} \nonumber\\
    &\prod_{d=1}^n \left(H(A_d)(1-H(A_d))\right)^{\frac{r_d-\sum_{j>d}p_{dj}-\sum_{i<d}p_{id}}{2}}\prod_{1\le i<j\le n }(-H(A_i)H(A_j))^{p_{ij}}\nonumber\\
    &=\sum_{{p_{ij}=0, 1\le i<j\le n\atop (i,j)\in \cC}}^{r_i \wedge r_j}\frac{r_1! \cdots r_n!}{\left(\frac{r_1-\sum_{j>1}p_{1j}}{2}\right)! \cdots \left(\frac{r_m-\sum_{j>m}p_{mj}-\sum_{i<m}p_{im}}{2}\right)! \cdots \left(\frac{r_n-\sum_{i<n}p_{in}}{2}\right)!}\nonumber\\
    &\frac{1}{2^{\frac{1}{2}\sum_{\ell=1}^nr_\ell-\sum_{1\le i<j\le n }p_{ij}}\left(\prod_{1\le i<j\le n }p_{ij}!\right)} \nonumber\\
    &\prod_{1\le i<j\le n }\left(-\sqrt{\frac{H(A_i)H(A_j)}{(1-H(A_i))(1-H(A_j))}}\right)^{p_{ij}}+o(1)\nonumber\\
    &\overset{a \rightarrow \infty}{\rightarrow} \sum_{{p_{ij}=0, 1\le i<j\le n\atop (i,j)\in \cC}}^{r_i \wedge r_j}\frac{r_1! \cdots r_n!}{\left(\frac{r_1-\sum_{j>1}p_{1j}}{2}\right)! \cdots \left(\frac{r_m-\sum_{j>m}p_{mj}-\sum_{i<m}p_{im}}{2}\right)! \cdots \left(\frac{r_n-\sum_{i<n}p_{in}}{2}\right)!}\nonumber\\
    &\frac{1}{2^{\frac{1}{2}\sum_{\ell=1}^nr_\ell-\sum_{1\le i<j\le n }p_{ij}}\left(\prod_{1\le i<j\le n }p_{ij}!\right)} \nonumber\\
    &\prod_{1\le i<j\le n }\left(-\sqrt{\frac{H(A_i)H(A_j)}{(1-H(A_i))(1-H(A_j))}}\right)^{p_{ij}},\label{0.15}
\end{align}
which is equal to $\mathbb{E}\left[\prod_{i=1}^n X_i^{r_i}\right]$,
where $X_1, \cdots, X_n$ are the multi-normal distribution 
defined in Proposition \ref{prop.6.4}.  
Now,   by multinomial expansion   we see  that for any positive 
 integer $k$ and   
for any $  (t_1,\cdots,t_n) \in \RR^d$,
\begin{eqnarray*}
\EE\left[\sum_{i=1}^n t_iD_a(A_i)  \right]^k 
&=&\sum_{k_1+\cdots+k_n=k}\left({k\atop k_1!\cdots k_n!}\right)
\mathbb{E}\left[\prod_{i=1}^n t_i^{k_i}D_a^{k_i}(A_i)\right] \\
&\overset{a \rightarrow \infty} {\rightarrow} &
\sum_{k_1+\cdots+k_n=k}\left({k\atop k_1!\cdots k_n!}\right)
\mathbb{E}\left[\prod_{i=1}^n t_i^{k_i}X_i^{k_i} \right] \\
& =&  
\EE\left[\sum_{i=1}^n t_iX_i  \right]^k \,. 
\end{eqnarray*}
By the method of moments (see e.g. 
  \citet[Theorem 30.2]{Billingsley}) it follows that  
$$\sum_{i=1}^n t_iD_a(A_i) \overset{d}{\rightarrow} \sum_{i=1}^n t_iX_i
\quad \hbox{as $a\rightarrow \infty$}.$$
}
Part (i) of Theorem \ref{theorem 0.1} follows 
then from the Cram\'er-Wold theorem
(e.g. \citet[Theorem 29.4]{Billingsley}).
 
Now,  we   prove the part (ii) of this theorem 
  by proving the weak convergence of finite dimensional distributions 
  and by verifying  a tightness condition. The finite dimensional weak convergence of $Q_{H,a}$ can be shown directly by part (i), i.e. for any finite measurable sets $A_1, \cdots, A_n$ in $\mathcal{X}^d$, we have
\begin{align*}
(Q_{H,a}(A_1), \cdots, Q_{H,a}(A_n)) \stackrel{d}{\rightarrow} (B_H^o(A_1), \cdots, B_H^o(A_n)).
\end{align*}
By Theorem 2 of \citet{bickel1971}, to show \eqref{e.2.24}  we only need to check the tightness condition, i.e, inequality (2) of \citet{bickel1971},  with  $\gamma_1=\gamma_2=2$, $\beta_1=\beta_2=1$ and $\mu=2H$.
Obviously,  $\mu$ is finite and nonatomic. For every pair of   Borel sets 
 $A$ and $B$ in $\mathcal{B}(\RR^d)$,  
  by the proof of part (i) of  this theorem  we   have  
\begin{align*}
&\mathbb{E}[|Q_{H,a}(A)|^2|Q_{H,a}(B)|^2]\\
&\qquad\qquad =[H(A)(1-H(A))][H(B)(1-H(B))]\mathbb{E}[D_a^2(A)D_a^2(B)]\\
&\qquad\qquad =[H(A)(1-H(A))][H(B)(1-H(B))]\\
&\qquad\qquad \times \left(1+2\frac{H(A)H(B)}{(1-H(A))(1-H(B))}+o(1)\right)\\
&\qquad\qquad  =3H(A)^2H(B)^2-H(A)^2H(B)-H(A)H(B)^2+H(A)H(B)+o(1)\\
&\qquad\qquad  \leq \mu(A)\mu(B).\ 
\end{align*}
The last inequality is due to the fact that $H(\cdot)\in (0,1)$ and thus $H(\cdot)^2 \leq H(\cdot)$. Therefore, the tightness condition    on $D(\RR ^d)$
is verified.
\fin 
\subsection{Proof of Theorem \ref{theorem 2.17}}
\begin{proof} Once we have Proposition \ref{prop.3.2}-\ref{prop.3.4}, 
the proofs of part (i) of this theorem for the various processes  
except the generalized Dirichlet process 
follow from a similar argument to that in  the proof of part (i) of Theorem \ref{theorem 0.1}. So, we shall omit the details. 

When $P \sim \GDP(a,r,H)$,   We need   the following result about  the variance of $P$  
from \citet{lijoi2005a}:  
\begin{align*}
\Var\left[P(A)\right]=H(A)\left(1-H(A)\right)\mathcal{I}_{a,r}\,, 
\end{align*}
where $\mathcal{I}_{a,r}$  is given by 
\begin{align}
\mathcal{I}_{a,r}&=a(r!)^a\sum_{k=1}^r\int_0^{\infty} \frac{x}{(k+x)^2\prod_{j=1}^r(j+x)^a}dx\nonumber\\
&=\frac{a(r!)^a\Gamma(ra)}{r^{ra}\Gamma(ra+2)}\sum_{j=1}^rF_D^{(r-1)}\left(ra,{\bf{a}_k^*};ra+2;\frac{1}{r}\bf{J}_{r-1}\right)\,. \label{7.2}
\end{align}
 Here  ${\bf a}_{k}^*=(a,\cdots,a+2,\cdots,a)^T$ is a $r-1$ dimensional vector where  the $k$-th element is   $a+2$  and all other elements are equal to $a$;   ${\bf J}_{r-1}=(1,\cdots,r-1)^T$;
and $F_D^{(r-1)}$ is the fourth Lauricella multiple hypergeometric function (see e.g. \citet{exton1976} for more details).  

Letting $x=\frac{t}{a}$  we have 
\begin{align}
 \mathcal{I}_{a,r}
 &=a\sum_{k=1}^r\frac{1}{k^2}\int_0^{\infty} \frac{x}{\left(1+\frac{x}{k}\right)^2\prod_{j=1}^r\left(1+\frac{x}{j}\right)^a}dx\nonumber\\
&=a\sum_{k=1}^r\frac{1}{ak^2}\int_0^{\infty} \frac{t}{a\left(1+\frac{t}{ak}\right)^2\prod_{j=1}^r\left(1+\frac{t}{aj}\right)^a}dt\,. \nonumber 
\end{align}
When $a$ is large, we can approximate $\mathcal{I}_{a,r}$ by 
\begin{align}
 \mathcal{I}_{a,r} 
&=\frac{1}{a}\sum_{k=1}^r\frac{1}{k^2}\int_0^{\infty} te^{-\left(\sum_{j=1}^r \frac{1}{j}\right)t}dt+o\left(\frac{1}{a}\right)\nonumber\\
&= \frac{\sum_{k=1}^r\left(\frac{1}{k}\right)^2}{\left(\sum_{j=1}^r\frac{1}{j}\right)^2a}+o\left(\frac{1}{a}\right).\label{7.3}
\end{align}
%
Denote 
 \begin{equation}c=\frac{\left(\sum_{j=1}^r\frac{1}{j}\right)^2}{\sum_{k=1}^r\left(\frac{1}{k}\right)^2}  \,. 
 \end{equation} 
Then 
 \begin{equation} 
 D_a(\cdot)=\sqrt{\frac{ca}{H(\cdot)(1-H(\cdot))}}\left(P(\cdot)-H(\cdot)\right)\,. 
 \label{e.6.5} 
  \end{equation}
We shall  prove the result for $n=3$ and the general $n$ case can be handled in a similar way. 

By the integral representation of the confluent form of the fourth Lauricella hypergeometric function (see  e.g. formula (1.4.3.9) in \citet{exton1976}), the joint probability density  of $P(A_1),P(A_2)$ admits  the  following form: 
\begin{align}
\rho (x_1,x_2)
 =&\frac{\Gamma(ra)}{\Gamma(raH_1)\Gamma(raH_2)\Gamma(raH_3)} \frac{(r!)}{r^{ra}\Gamma(ra)} \frac{\Gamma(raH_1)\Gamma(raH_2)\Gamma(raH_3)}{[\Gamma(aH_1)\Gamma(aH_2)\Gamma(aH_3)]^r} 
  \nonumber\\
& \times x_1^{2aH_1-1}x_2^{2aH_2-1}(1-x_1-x_2)^{2aH_3-1}  \nonumber\\
 &\times \int_0^{\infty}\int_{0\leq u_1^{(1)}+\cdots+u_{r-1}^{(1)}\leq 1}\int_{0\leq u_1^{(2)}+\cdots+u_{r-1}^{(2)}\leq 1}\int_{0\leq u_1^{(3)}+\cdots+u_{r-1}^{(3)}\leq 1} \xi^{ra-1}\nonumber\\
&\times  \exp \bigg\{-\xi+\xi\left(\frac{\sum_{k=1}^{r-1}ku_k^{(1)}x_1+\sum_{k=1}^{r-1}ku_k^{(2)}x_2+\sum_{k=1}^{r-1}ku_k^{(3)}(1-x_1-x_2)}{r}\right)\bigg\}\nonumber\\
&\times  \prod_{i=1}^3 \left[u_1^{i}\cdots u_{r-1}^{(i)}\left(1-u_1^{i}-\cdots -u_{r-1}^{(i)}\right)\right]^{aH_i-1}du_1^{(1)}\cdots du_{r-1}^{(1)} \nonumber\\
& \qquad  du_1^{(2)}\cdots du_{r-1}^{(2)}du_1^{(3)}\cdots du_{r-1}^{(3)}d\xi,\label{7.4}
\end{align}
where $H_i=H(A_i)$ for $i=1,2,3$. 
Using the expression of \eqref{e.6.5} and the above density form 
\eqref{7.4}, we can obtain the probability density function of $D(A_1),D(A_2)$ as follows:
\begin{align}
 f(y_1,y_2)=& \mathcal{J}_1 \times f_1(y_1, y_2) \,, 
\nonumber\\
\label{7.5}
\end{align}
where 
\begin{align}
&\mathcal{J}_1=\frac{\sqrt{H_1(1-H_1)H_2(1-H_2)}\Gamma(ra)}{ca\Gamma(raH_1)\Gamma(raH_2)\Gamma(raH_3)}\left(\sqrt{\frac{H_1(1-H_1)}{ca}}y_1+H_1\right)^{raH_1-1}\nonumber\\
&\times  \left(\sqrt{\frac{H_2(1-H_2)}{ca}}y_2+H_2\right)^{raH_2-1}\left(H_3-\frac{\sqrt{H_1(1-H_1)}y_1+\sqrt{H_2(1-H_2)}y_2}{\sqrt{ca}}\right)^{raH_3-1}  
\end{align}
and 
\begin{align} 
f_1(y_1, y_2) =&\int_0^{\infty}\idotsint_{0\leq u_1^{(1)}+\cdots+u_{r-1}^{(1)}\leq 1}\idotsint_{0\leq u_1^{(2)}+\cdots+u_{r-1}^{(2)}\leq 1}\nonumber\\
&\idotsint_{0\leq u_1^{(3)}+\cdots+u_{r-1}^{(3)}\leq 1}\prod_{i=1}^3 \left[u_1^{i}\cdots u_{r-1}^{(i)}(1-u_1^{i}-\cdots -u_{r-1}^{(i)})\right]^{aH_i-1}\xi^{ra-1}\nonumber\\
&\times  \exp \Bigg\{-\xi+\frac{\xi}{r}\bigg[\sum_{k=1}^{r-1}ku_k^{(1)}\left(\frac{\sqrt{H_1(1-H_1)}y_1}{\sqrt{ca}}+H_1\right)\nonumber\\
&+\sum_{k=1}^{r-1}ku_k^{(2)}\left(\frac{\sqrt{H_2(1-H_2)}y_2}{\sqrt{ca}}+H_2\right)\nonumber\\
&+\sum_{k=1}^{r-1}ku_k^{(3)}\left(H_3-\frac{\sqrt{H_1(1-H_1)}y_1+\sqrt{H_2(1-H_2)}y_2}{\sqrt{ca}}\right)\bigg] \Bigg\}\nonumber\\
&du_1^{(1)}\cdots du_{r-1}^{(1)} du_1^{(2)}\cdots du_{r-1}^{(2)}du_1^{(3)}\cdots du_{r-1}^{(3)}d\xi\,. \label{e.6.9}
\end{align}
To determine if the  density of $f(y_1,y_2)$ has a limit 
or not and if yes, to find the limiting  density, we shall find the limits of $
\mathcal{J}_1$ and $f_1(y_1, y_2)$ separately.

\medskip 
\noindent \textbf{Step 1:  Limit of $\mathcal{J}_1$.} 

By Stirling's formula $\Gamma(z)=\sqrt{2\pi}z^{z-\frac{1}{2}}e^{-z}\left(1+o\left(\frac{1}{z}\right)\right)$ when $z$ is large,  we have the following approximation: 
\begin{align}
 \mathcal{J}_1
 =&\frac{ra}{ca2\pi}\sqrt{\frac{(1-H_1)(1-H_2)}{H_3}}\left(\sqrt{\frac{(1-H_1)}{caH_1}}y_1+1\right)^{raH_1-1}\left(\sqrt{\frac{(1-H_2)}{caH_2}}y_2+1\right)^{raH_2-1}\nonumber\\
&\times \left(1-\frac{\sqrt{H_1(1-H_1)}y_1+\sqrt{H_2(1-H_2)}y_2}{H_3\sqrt{ca}}\right)^{raH_3-1}\left(1+o\left(\frac{1}{a}\right)\right),\nonumber\\
 =&\frac{r}{c2\pi}\sqrt{\frac{(1-H_1)(1-H_2)}{H_3}}\left(1+o\left(\frac{1}{a}\right)\right)\exp \bigg \{ ra \bigg[H_1\log\left(\sqrt{\frac{(1-H_1)}{caH_1}}y_1+1\right)\nonumber\\
&+H_2 \log \left(\sqrt{\frac{(1-H_2)}{caH_2}}y_2+1\right)\nonumber\\
&+H_3 \log \left(1-\frac{\sqrt{H_1(1-H_1)}y_1+\sqrt{H_2(1-H_2)}y_2}{H_3\sqrt{ca}}\right) \bigg] \bigg\}\,. \nonumber
\end{align}
An application of  $\log(1+\frac{1}{\sqrt{z}})=\frac{1}{\sqrt{z}}-\frac{1}{2z}+  o\left(\frac{1}{z}\right) $ for large $z$ yields
\begin{align}
 \mathcal{J}_1=&\frac{r}{c}\frac{\sqrt{(1-H_1)(1-H_2)}}{2\pi \sqrt{H_3}}\nonumber\\
&\times \exp \bigg\{ -\frac{r}{c}\frac{y_1^2+y_2^2+2y_1y_2\sqrt{\frac{H_1H_2}{(1-H_1)(1-H_2)}}}{2\frac{H_3}{(1-H_1)(1-H_2)}}\bigg\}\left(1+o\left(\frac{1}{a}\right)\right).\label{7.6}
\end{align}

\medskip
\noindent\textbf{Step 2: The limit of $f_1(y_1, y_2)$.    This is much more complicated. We first obtain the leading terms of $f_1$ as $a\rightarrow \infty$}.

 To make the presentation  clear, 
 denote the integrating variables by 
  $${\bf{z}}=\left(u_1^{(1)},\cdots, u_{r-1}^{(1)}, u_1^{(2)}, \cdots, u_{r-1}^{(2)}, u_1^{(3)}, \cdots, u_{r-1}^{(3)}, \xi\right)^T$$ 
  and denote 
\begin{align}
&g\left(u_1^{(1)},\cdots, u_{r-1}^{(1)}, u_1^{(2)}, \cdots, u_{r-1}^{(2)}, u_1^{(3)}, \cdots, du_{r-1}^{(3)}, \xi\right) \nonumber\\
&\qquad    := \sum_{i=1}^3 (aH_i-1)\log \left[u_1^{(i)}\cdots u_{r-1}^{(i)}(1-u_1^{(i)}-\cdots -u_{r-1}^{(i)})\right]+(ra-1)\log (\xi)\nonumber\\
& \qquad   -\xi+\frac{\xi}{r}\bigg[\sum_{k=1}^{r-1}ku_k^{(1)}\left(\frac{\sqrt{H_1(1-H_1)}y_1}{\sqrt{ca}}+H_1\right)+\sum_{k=1}^{r-1}ku_k^{(2)}\left(\frac{\sqrt{H_2(1-H_2)}y_2}{\sqrt{ca}}+H_2\right)\nonumber\\
&\qquad \qquad   +\sum_{k=1}^{r-1}ku_k^{(3)}\left(H_3-\frac{\sqrt{H_1(1-H_1)}y_1+\sqrt{H_2(1-H_2)}y_2}{\sqrt{ca}}\right)\bigg]\,. \label{7.7}
\end{align}
This function $g$  attains  its maximum at its critical point 
\begin{equation}
{\bf{z_0}}=\left(u_{1,0}^{(1)},\cdots, u_{r-1,0}^{(1)}, u_{1,0}^{(2)}, \cdots, u_{r-1,0}^{(2)}, u_{1,0}^{(3)}, \cdots, u_{r-1,0}^{(3)}, \xi_0\right)^T\,, 
\end{equation} 
where 
\begin{empheq}[left=\empheqlbrace]{align*} 
&u_{k,0}^{(i)}=\frac{1}{(r-k)\left(1+\cdots+\frac{1}{r}\right)}\,,
\quad k=1,\cdots, r-1\,, \ i=1,2,3 \,;\\
& \xi_0=(ra-1)\left(1+\cdots+\frac{1}{r}\right)\,.
\end{empheq} 
By an elementary  calculation  we   have
\[ 
g'({\bf z_0})=\begin{pmatrix}
\left[(1+\cdots+\frac{1}{r})\left(1-\frac{H_1}{r}\right)+(1+\cdots+\frac{1}{r})\left(a-\frac{1}{r}\right)\frac{\sqrt{H_1(1-H_1)}y_1}{\sqrt{ca}}\right]{\bf J}_{r-1}\\ \\
\left[(1+\cdots+\frac{1}{r})\left(1-\frac{H_2}{r}\right)+(1+\cdots+\frac{1}{r})\left(a-\frac{1}{r}\right)\frac{\sqrt{H_2(1-H_2)}y_2}{\sqrt{ca}}\right]{\bf J}_{r-1}\\ \\
\left[(1+\cdots+\frac{1}{r})\left(1-\frac{H_1}{r}\right)-(1+\cdots+\frac{1}{r})\left(a-\frac{1}{r}\right)\frac{\sqrt{H_1(1-H_1)}y_1+\sqrt{H_2(1-H_2)}y_2}{\sqrt{ca}}\right]{\bf J}_{r-1}\\
0
\end{pmatrix}\,,
\]
where ${\bf J}_{r-1}=(1,\cdots,r-1)^T$  and 
\begin{align*}
g''({\bf{z_0}})=\left(\begin{array}{cccc}
M_1&  &  & N_1\\
 & M_2&&N_2\\
 &  & M_3&N_3\\
N_1^T&N_2^T&N_3^T&-\frac{1}{(ra-1)\left(\sum_{k=1}^{r-1}\frac{1}{k}\right)^2}
\end{array}\right)\,, 
\end{align*}
where  the empty entries should be  filled with a  $(r-1)\times (r-1)$  
dimensional zero    matrix and where
for   $i=1,2,3$, $M_i$ is a $(r-1)\times (r-1)$ matrix given by  
\[
M_i=-(aH_i-1)\left(\sum_{k=1}^{r-1}\frac{1}{k}\right)^2\left(\begin{array}{cccc}
\left[(r-1)^2+r^2\right] & r^2&\cdots& r^2\\
 r^2 & \left[(r-2)^2+r^2\right]&\cdots&r^2\\
 \vdots&\vdots&\ddots&\vdots\\
 r^2&r^2&\cdots& \left[(1)^2+r^2\right]
\end{array}\right)\,, 
\]
and $ N_i$ is a $(r-1)$ column vector given by 
\[
N_i=\begin{pmatrix}
\frac{1\left(\frac{\sqrt{H_i(1-H_i)}y_i}{\sqrt{ca}}+H_i\right)}{r}\\
\frac{2\left(\frac{\sqrt{H_i(1-H_i)}y_i}{\sqrt{ca}}+H_i\right)}{r}\\
\vdots\\
\frac{(r-1)\left(\frac{\sqrt{H_i(1-H_i)}y_i}{\sqrt{ca}}+H_i\right)}{r} 
\end{pmatrix}\,. 
\]
With these notations    we  have   when $a$ is large 
\begin{align}
 f_1(y_1, y_2)=&\frac{(r!)}{r^{ra}\Gamma(ra)}\frac{\Gamma(raH_1)\Gamma(raH_2)\Gamma(raH_3)}{\left[\Gamma(aH_1)\Gamma(aH_2)\Gamma(aH_3)\right]^r}\int_0^{\infty}\idotsint_{0\leq u_1^{(1)}+\cdots+u_{r-1}^{(1)}\leq 1}\nonumber\\
&\idotsint_{0\leq u_1^{(2)}+\cdots+u_{r-1}^{(2)}\leq 1}\idotsint_{0\leq u_1^{(3)}+\cdots+u_{r-1}^{(3)}\leq 1}e^{g\left(\bf{z}\right)}d\bf{z}\nonumber\\
 =&\frac{(r!)(aH_1)^{\frac{r-1}{2}}(aH_2)^{\frac{r-1}{2}}(aH_3)^{\frac{r-1}{2}}}{r^{\frac{3}{2}}(\sqrt{2\pi})^{3r-2}e^{-ra}(ra)^{ra-\frac{1}{2}}}  \int_0^{\infty}\idotsint_{0\leq u_1^{(1)}+\cdots+u_{r-1}^{(1)}\leq 1}\nonumber\\
&\idotsint_{0\leq u_1^{(2)}+\cdots+u_{r-1}^{(2)}\leq 1} \idotsint_{0\leq u_1^{(3)}+\cdots+u_{r-1}^{(3)}\leq 1}\exp\Big\{g({\bf{z_0}})\nonumber\\
& + g'({\bf{z_0}})^T({\bf{z}}-{\bf{z_0}})+\frac{1}{2}({\bf{z}}-{\bf{z_0}})^Tg''({\bf{z_0}})({\bf{z}}-{\bf{z_0}})\Big\} \left(1+o\left(\frac{1}{a}\right)\right)d\bf{z}\,. \label{7.8}
\end{align}

\noindent\textbf{Step 3: Evaluation of the leading term of $f_1(y_1, y_2)$}.

In order to evaluate 
 the integral of  \eqref{7.8}, we use the change of variables 
\begin{align*}
&u_k^{(i)}-u_{k,0}^{(i)}=\frac{t_k^{(i)}}{\sqrt{aH_i}\sqrt{(r-k)^2+r^2}\left(\sum_{k=1}^{r-1}\frac{1}{k}\right)}, 
\end{align*}
for $k=1,\cdots,r-1$, $i=1,2,3$, and  
\begin{align*} 
&\xi-\xi_0=\sqrt{ra-1}\left(\sum_{k=1}^{r-1}\frac{1}{k}\right)s\,. 
\end{align*}
  Thus, we have 
\begin{align}
 f_1(y_1, y_2)=&\frac{(r!)(aH_1)^{\frac{r-1}{2}}(aH_2)^{\frac{r-1}{2}}(aH_3)^{\frac{r-1}{2}}}{r^{\frac{3}{2}}(\sqrt{2\pi})^{3r-2}e^{-ra}(ra)^{ra-\frac{1}{2}}} \sqrt{ra-1}\left(\sum_{k=1}^{r}\frac{1}{k}\right)\nonumber\\
&\times \prod_{i=1}^3\prod_{k=1}^{r-1}\frac{1}{\sqrt{aH_i}\sqrt{(r-k)^2+r^2}\left(\sum_{k=1}^{r}\frac{1}{k}\right)}\nonumber \\
&\times  \int_{-\infty}^{\infty}\cdots\int_{-\infty}^{\infty}\exp\Big\{g({\bf{z_0}}) + {\bf{b}}^T{\bf{t}}-\frac{1}{2}{\bf{t}}^TA{\bf{t}}\Big\} \left(1+o\left(\frac{1}{a}\right)\right)d\bf{t} \nonumber \\
 =&\frac{(r!)^3\left(\sum_{k=1}^{r}\frac{1}{k}\right)^3}{r^{\frac{3}{2}}\prod_{k=1}^{r-1}\left(\sqrt{(r-k)^2+r^2}\right)^3}\left(\det A\right)^{-\frac{1}{2}}\nonumber \\
&\qquad\qquad \times  \exp\left\{\frac{1}{2}{\bf{b}}^TA^{-1}{\bf{b}}\right\}\left(1+o\left(\frac{1}{a}\right)\right),\label{7.10}
\end{align}
where ${\bf{t}}=\left(t_1^{(1)},\cdots,t_{r-1}^{(2)},t_1^{(2)},\cdots,t_{r-1}^{(2)},t_1^{(3)},\cdots,t_{r-1}^{(3)},s\right)^T$, 
and  where a  direct calculation from $g'({\bf z_0})$ gives 
\begin{equation} 
{\bf{b}}=\begin{pmatrix}
-\frac{\sqrt{r}\sqrt{1-H_1}y_1}{\sqrt{c}}B\\
-\frac{\sqrt{r}\sqrt{1-H_2}y_2}{\sqrt{c}}B\\
-\frac{\sqrt{r}\left(\sqrt{H_1(1-H_1)}y_1+\sqrt{H_2(1-H_2)}y_2\right)}{\sqrt{c}\sqrt{H_3}}B\\
0
\end{pmatrix}\,. \label{e.7.15a} 
 \end{equation}
The matrix  $A$ in \eqref{7.10} can be  found directly from $g''({\bf z_0})$   and will  be
given below   when we study it. The above last identity 
\eqref{7.10}  follows from the fact that  
\[
e^{g({\bf{z_0}})}=\frac{\left[(ra-1)\left(\sum_{k=1}^{r}\frac{1}{k}\right)\right]^{ra-1}}{\left[r!\left(\sum_{k=1}^{r}\frac{1}{k}\right)^r\right]^{a-3}}e^{-(ra-1)}
\]
 and the multivariate Gaussian integral formula. 
   By a simple algebra  we can write 
\begin{equation}
A=\left(\begin{array}{cccc}
A_0 &  &  & B_1\\
 &  A_0&&B_2\\
 &  & A_0&B_3\\
B_1^T&B_2^T&B_3^T&1
\end{array}\right)\,,  \label{e.6.15} 
\end{equation} 
where 
$A_0$ is a $(r-1) \times (r-1)$ matrix whose entries are  
\[
[A_0]_{ij}=\begin{cases}
1 &\qquad  \mbox{if } i =j\\
\frac{r^2}{\sqrt{(r-i)^2+r^2}\sqrt{(r-j)^2+r^2}} & \qquad \mbox{if } i \neq j\,. 
\end{cases}
\]
and   for $i=1,2,3$ 
\[
B_i=-\begin{pmatrix}
\frac{1\sqrt{ra-1}\left(\frac{\sqrt{H_i(1-H_i)}y_i}{\sqrt{ca}}+H_i\right)}{r\sqrt{aH_i}\sqrt{(r-1)^2+r^2}}\\
\frac{2\sqrt{ra-1}\left(\frac{\sqrt{H_i(1-H_i)}y_i}{\sqrt{ca}}+H_i\right)}{r\sqrt{aH_i}\sqrt{(r-2)^2+r^2}}\\
\vdots\\
\frac{(r-1)\sqrt{ra-1}\left(\frac{\sqrt{H_i(1-H_i)}y_i}{\sqrt{ca}}+H_i\right)}{r\sqrt{aH_i}\sqrt{(1)^2+r^2}}
\end{pmatrix}= -\begin{pmatrix}
\frac{1\sqrt{H_1}}{\sqrt{r}\sqrt{(r-1)^2+r^2}}\left(1+o\left(\frac{1}{\sqrt{a}}\right)\right)\\
\frac{2\sqrt{H_1}}{\sqrt{r}\sqrt{(r-2)^2+r^2}}\left(1+o\left(\frac{1}{\sqrt{a}}\right)\right)\\
\vdots\\
\frac{(r-1)\sqrt{H_1}}{\sqrt{r}\sqrt{(1)^2+r^2}}\left(1+o\left(\frac{1}{\sqrt{a}}\right)\right)
\end{pmatrix}\,. 
\]

\noindent\textbf{Step 4:    The inverse and the 
determinant of $A$} 

We need to  find $A ^{-1}$.  First we find 
$A_0^{-1}$.  From the expression of $A_0$  we can write 
\[
A_0=D+{\bf{v}}{\bf{v}}^T=D^{\frac{1}{2}}\left(I+D^{-\frac{1}{2}}{\bf{v}}{\bf{v}}^TD^{-\frac{1}{2}}\right)D^{\frac{1}{2}}\,, 
\]
where  
\[
{\bf{v}}=\begin{pmatrix}
\frac{r}{\sqrt{(r-1)^2+r^2}}\\
\frac{r}{\sqrt{(r-2)^2+r^2}}\\
\vdots\\
\frac{r}{\sqrt{(1)^2+r^2}}
\end{pmatrix}\]
  and 
\[
D=\left(\begin{array}{cccc}
\frac{(r-1)^2}{(r-1)^2+r^2} & 0&\cdots& 0\\
0 & \frac{(r-2)^2}{(r-2)^2+r^2}&\cdots&0\\
 \vdots&\vdots&\ddots&\vdots\\
 0&0&\cdots &\frac{1^2}{1^2+r^2}
\end{array}\right)\,.
\]
 Hence, the determinant of $A_0$ is given by 
\begin{align*}
\det (A_0)&=\det\left(D^{\frac{1}{2}}\right)\det\left(I+D^{-\frac{1}{2}}{\bf{v}}{\bf{v}}^TD^{-\frac{1}{2}}\right)\det\left(D^{\frac{1}{2}}\right)\\
&=\det\left(D^{\frac{1}{2}}\right)\left(1+\left(D^{-\frac{1}{2}}{\bf{v}}\right)^T\left(D^{-\frac{1}{2}}{\bf{v}}\right)\right)\det\left(D^{\frac{1}{2}}\right)\\
&=(\frac{(r-1)!}{\prod_{k=1}^{r-1}\sqrt{(r-k)^2+r^2}})^2(1+\sum_{k=1}^{r-1}\frac{r^2}{(r-k)^2})\\
&=\frac{\sum_{k=1}^r(\frac{r!}{k})^2}{\prod_{k=1}^{r-1}(\sqrt{(r-k)^2+r^2})^2}\,. 
\end{align*}
Now,   by the Sherman-Morrison formula we have
\[
A_0^{-1}=(D+ff^T)^{-1}=D^{-1}-\frac{1}{1+f^TD^{-1}f} D^{-1}ff^TD^{-1}\,.
\]
As a result we obtain 
\[
[A_0^{-1}]_{ij}=\begin{cases}
-\frac{1}{\sum_{k=1}^r(\frac{1}{k})^2}\frac{\sqrt{(r-i)^2+r^2}\sqrt{(r-j)^2+r^2}}{(r-i)^2(r-j)^2}\qquad \qquad &\mbox{if } i \neq j\\
\frac{[\sum_{k=1}^r(\frac{1}{k})^2-\frac{1}{(r-i)^2}]}{\sum_{k=1}^r(\frac{1}{k})^2}\frac{[(r-i)^2+r^2]}{(r-i)^2} &\mbox{if } i = j\,. 
\end{cases}
\]
The determinant of  $A$ can be computed as follows.
\begin{eqnarray}
\det (A)&=&\left(\det A_0 \right)^3\left(1-B_1^TA_0^{-1}B_1-B_2^TA_0^{-1}B_2-B_3^TA_0^{-1}B_3\right)\nonumber\\
&=&\left(\det A_0 \right)^3\left(1-B^TA_0^{-1}B\right)\,,  
\label{e.7.15} 
\end{eqnarray} 
where
\[
B= -\begin{pmatrix}
\frac{1}{\sqrt{r}\sqrt{(r-1)^2+r^2}}\left(1+o\left(\frac{1}{\sqrt{a}}\right)\right)\\
\frac{2}{\sqrt{r}\sqrt{(r-2)^2+r^2}}\left(1+o\left(\frac{1}{\sqrt{a}}\right)\right)\\
\vdots\\
\frac{(r-1)}{\sqrt{r}\sqrt{(1)^2+r^2}}\left(1+o\left(\frac{1}{\sqrt{a}}\right)\right)
\end{pmatrix}\,.
\]
From the relation \eqref{e.7.15} of expressing  $\det(A)$ by  
$\det(A_0)$  and $B$ we have 
\begin{align}
\det (A) =\left(\det A_0 \right)^3\left(1-B^TA_0^{-1}B\right)=\frac{\left(\sum_{k=1}^r\frac{r!}{k}\right)^2\left(\sum_{k=1}^r\left(\frac{r!}{k}\right)^2\right)^2}{r\prod_{k=1}^{r-1}\left(\sqrt{(r-k)^2+r^2}\right)^6}\,. \label{7.11}
\end{align} 
Moreover, by the block Gaussian elimination method, we   find $A^{-1}$ as\\  
\begin{equation}
\begin{pmatrix}
A_0^{-1}\left(I_0+\frac{B_1B_1^TA_0^{-1}}{m}\right) & \frac{A_0^{-1}B_1B_2^TA_0^{-1}}{m}&\frac{A_0^{-1}B_1B_3^TA_0^{-1}}{m}& -\frac{A_0^{-1}B_1}{m}\\
\frac{A_0^{-1}B_2B_1^TA_0^{-1}}{m} & A_0^{-1}\left(I_0+\frac{B_2B_2^TA_0^{-1}}{m}\right)&\frac{A_0^{-1}B_2B_3^TA_0^{-1}}{m}&-\frac{A_0^{-1}B_2}{m}\\
 \frac{A_0^{-1}B_3B_1^TA_0^{-1}}{m}&\frac{A_0^{-1}B_3B_2^TA_0^{-1}}{m}&A_0^{-1}\left(I_0+\frac{B_3B_3^TA_0^{-1}}{m}\right)&-\frac{A_0^{-1}B_3}{m}\\
-\frac{B_1^TA_0^{-1}}{m}&-\frac{B_2^TA_0^{-1}}{m}&-\frac{B_3^TA_0^{-1}}{m}&\frac{1}{m}
\end{pmatrix}\,,
\label{e.7.18} 
\end{equation} 
where 
\[
m=1-B_1^TA_0^{-1}B_1-B_2^TA_0^{-1}B_2-B_3^TA_0^{-1}B_3 
\]
 and $I_0$ is the $(r-1) \times (r-1)$-dimensional identity matrix.

 \noindent\textbf{Step 5:   the limit of 
$f_1(y_1, y_2)$}.

 Combining  \eqref{e.7.18} and \eqref{e.7.15a},  
 one finds 
\begin{align}
 {\bf{b}}^TA^{-1}{\bf{b}}  =&\frac{r}{c}B^TA_0^{-1}B\Bigg((1-H_1)y_1^2+(1-H_2)y_2^2\nonumber\\
&\qquad +\frac{\left(\sqrt{H_1(1-H_1)}y_1+\sqrt{H_2(1-H_2)}y_2\right)^2}{H_3} 
\Bigg)\nonumber\\
 =&(\frac{r}{c}-1)\frac{y_1^2+y_2^2+2y_1y_2\sqrt{\frac{H_1H_2}{(1-H_1)(1-H_2)}}}{\frac{H_3}{(1-H_1)(1-H_2)}}\left(1+o\left(\frac{1}{a}\right)\right)\,.  \label{7.12}
\end{align}
Substituting the   expression 
\eqref{7.12} and the formula \eqref{7.11}  
for the determinant into   \eqref{7.10} yields 
\begin{align}
f_1(y_1, y_2)=\frac{c}{r}\exp\left\{\left(\frac{r}{c}-1\right)\frac{y_1^2+y_2^2+2y_1y_2\sqrt{\frac{H_1H_2}{(1-H_1)(1-H_2)}}}{2\frac{H_3}{(1-H_1)(1-H_2)}}\right\}\left(1+o\left(\frac{1}{a}\right)\right).\label{7.13}
\end{align}
  Combining \eqref{7.13} with the asymptotic    \eqref{7.6}  of $\mathcal{J}_1$ and the relationship $f(y_1,y_2)= \mathcal{J}_1 \times f_1(y_1, y_2)$, we   see $f(y_1, y_2)$ converges to the desired  Gaussian density, which completes the proof of part (i) of the theorem when $P \sim \GDP(a,r,H)$. 

 The proof of part (ii) of this theorem follows from the same argument as that in the proof of part (ii) of Theorem \ref{theorem 0.1}.
\end{proof}

\end{document}